\documentclass[reqno]{amsart}
\usepackage{amssymb}
\usepackage{amsmath}
\usepackage[mathscr]{euscript}

\makeatletter
\@addtoreset{equation}{section}
\makeatother

\renewcommand\thefigure{\thesection.\@arabic\c@figure}
\renewcommand\thetable{\thesection.\@arabic\c@table}

\newtheorem{theorem}{Theorem}[section]
\newtheorem{lemma}[theorem]{Lemma}
\newtheorem{proposition}[theorem]{Proposition}
\newtheorem{corollary}[theorem]{Corollary}

\newtheorem{remark}[theorem]{Remark}

\newcommand{\mc}[1]{{\mathcal #1}}
\newcommand{\mf}[1]{{\mathfrak #1}}
\newcommand{\mb}[1]{{\mathbf #1}}
\newcommand{\bb}[1]{{\mathbb #1}}
\newcommand{\bs}[1]{{\boldsymbol #1}}
\newcommand{\ms}[1]{{\mathscr #1}}

\renewcommand{\>}{\rangle}

\renewcommand{\Cap}{{\rm cap}}

\begin{document}

\title[Metastability of Markov chains]{Metastability of reversible
  finite state Markov processes}

\author{J. Beltr\'an, C. Landim}

\address{\noindent IMCA, Calle los Bi\'ologos 245, Urb. San C\'esar
  Primera Etapa, Lima 12, Per\'u and PUCP, Av. Universitaria cdra. 18,
  San Miguel, Ap. 1761, Lima 100, Per\'u. 
\newline e-mail: \rm
  \texttt{johel@impa.br} }

\address{\noindent IMPA, Estrada Dona Castorina 110, CEP 22460 Rio de
  Janeiro, Brasil and CNRS UMR 6085, Universit\'e de Rouen, Avenue de
  l'Universit\'e, BP.12, Technop\^ole du Madril\-let, F76801
  Saint-\'Etienne-du-Rouvray, France.  \newline e-mail: \rm
  \texttt{landim@impa.br} }


\begin{abstract}
  We prove the metastable behavior of reversible Markov processes on
  finite state spaces under minimal conditions on the jump rates. To
  illustrate the result we deduce the metastable behavior of the Ising
  model with a small magnetic field at very low temperature.
\end{abstract}

\maketitle

\section{Introduction}
\label{sec0}

Metastability is a phenomenon observed in thermodynamic systems close
to a first--order phase transition. Describing the evolution among
competing metastable states or from a metastable state to a stable
state in stochastic lattice spin systems at low temperatures is still
a subject of considerable interest. We refer to \cite{h1, ov, b2, h2}
for recent monographs on the subject.

Inspired from the metastable behavior of condensed zero-range
processes \cite{bl3} and from the metastable behavior of random walks
among random traps \cite{jlt1, jlt2}, we proposed in \cite{bl2} a
definition of metastability and developed some techniques,
particularly effective in the reversible case, to prove the
metastability of sequences of Markov processes on countable state
spaces.

To present the approach introduced in \cite{bl2} in the simplest
possible context, we examine in this article the metastable behavior
of reversible Markov processes on finite state spaces. The main result
of the article, Theorem \ref{s19}, describes all metastable behaviors
of such processes in all time scales under the minimal conditions
\eqref{24}, \eqref{30} on the jump rates.

The minimal assumptions \eqref{24}, \eqref{30} are clearly satisfied
by all Markovian dynamics studied so far. This includes the Glauber
dynamics with a small external field at very low temperature in two
\cite{ns1, ns2} and three \cite{bc} dimensions, anisotropic Glauber
dynamics \cite{ko1, ko2}, conservative Kawasaki dynamics \cite{hnos1,
  hnos2, hnos, ghnos1}, birth and death processes \cite{s} and the
reversible dynamics considered in \cite{bm}.

Theorem \ref{s19} asserts the existence of time scales in which a
metastable behavior is observed. To apply this result to specific
models, as pointed out in Remark \ref{mr1}, one needs to compute the
capacity between metastable sets and the hitting probabilities of
metastable sets. In some cases, as in the Kawasaki dynamics, the exact
calculation of the hitting probabilities is impossible, but one can at
least determine if the asymptotic hitting probability is strictly
positive or not. In these cases, an exact description of the
metastable behavior of the process is not available. It is only known
that asymptotically the process spends an exponential time, of a
computable mean, in a metastable set at the end of which it jumps to
some other metastable set with an unknown probability, where the same
phenomenon is observed.

In contrast with the pathwise approach proposed in \cite{cgov}, the
one presented in this article does not highlight the saddle
configurations visited by the process when moving from one metastable
state to another. However, to compute the exact depths of the valleys,
a calculation which relies on a precise estimation of the capacities,
one needs to characterize the saddle configurations. This is clearly
illustrated in Section \ref{si2} where the saddle configurations of a
valley $\mc E_\sigma$, denoted by $\bb W(\sigma)$, appear when we
compute the capacities between the metastable sets of the Ising model.

The lack of precise results on the saddle configurations is
compensated by an exact description of the asymptotic dynamics among
wells. We are able, in particular, with similar methods to the ones
introduced in Bovier et al. \cite{begk1, begk2}, to show the existence
of sequences $\theta_N$ for which $T_N/\theta_N$ converges to a mean
one exponential distribution, if $T_N$ represents the time the process
leaves a metastable set. Furthermore, we also prove the asymptotic
independence of $T_N/\theta_N$ and $\eta_{T_N}$, where $\eta$
represents the Markov process, a question not considered before. The
proof of this asymptotic independence requires the convergence of the
average jump rates, defined in \eqref{35}, which is, in most cases,
the main technical difficulty in the deduction of metastability.

To illustrate the main result, we consider in Section \ref{si1} the
metastable behavior of the two dimensional Ising model with a small
external field at very low temperature, the model of Neves and
Schonmann \cite{ns1, ns2}, and a case in which all parameters can be
exactly computed.

\section{Notation and Results}
\label{sec3}

We say that a sequence of positive real numbers $(\alpha_N : N\ge 1)$
is of lower magnitude than a similar sequence $(\beta_N : N\ge 1)$,
$\alpha_N \prec \beta_N$ or $\beta_N \succ \alpha_N$, if
$\alpha_N/\beta_N$ vanishes as $N\uparrow\infty$. We say that two
positive sequences $(\alpha_N : N\ge 1)$, $(\beta_N : N\ge 1)$ are of
the same magnitude, $\alpha_N \approx \beta_N$, if there exists a
finite positive constant $C_0$ such that
\begin{equation*}
C_0^{-1} \;\le\, \liminf_{N\to\infty} \frac{\alpha_N}{\beta_N} \;\le\;
\limsup_{N\to\infty} \frac{\alpha_N}{\beta_N} \;\le\; C_0\;. 
\end{equation*}
Finally, $\alpha_N\preceq \beta_N$ or $\beta_N \succeq \alpha_N$ means
that $\alpha_N \prec \beta_N$ or $\alpha_N \approx \beta_N$.

We say that a set of sequences $(\alpha_N (1): N\ge 1), \dots,
(\alpha_N (\ell): N\ge 1)$ is \emph{comparable} if for each $i\not =
j$ one of the three possibilities holds: either $\alpha_N (i) \prec
\alpha_N (j)$ or $\alpha_N (j) \prec \alpha_N (i)$ or $\alpha_N
(i)/\alpha_N (j)$ converges to a constant $c_{i,j}\in (0,\infty)$.
Hence, for example, the possibility that the sequence $\alpha_N
(i)/\alpha_N (j)$ oscillates between two finite values and does not
converge is excluded.
\smallskip

Fix a finite set $E$ and sequences $\{\lambda_N(j):N\ge 1\}$, $0\le
j\le \mf n$, such that $ \lambda_{N}(\mf n) \prec \lambda_N(\mf n -1)
\prec \cdots \prec \lambda_N(0) \equiv 1$.  Consider a Markov process
$\{\eta^N_t : t\ge 0\}$ on $E$ with jump rates denoted by $R_N(x,y)$,
$x\not = y\in E$. We assume that the process is \emph{irreducible},
that the unique stationary state, denoted by $\mu_N$, is
\emph{reversible}, and that the jump rates satisfy the following
multi-scale hypothesis.  There exists $a: E\times E \to \{0, \dots,
\mf n\}$ such that
\begin{equation}
\label{24}
R_N(x,y) = r_N(x,y) \, \lambda_N(a(x,y))\;, \quad x\not = y\in E\;,
\end{equation}
where $\lim_{N\to\infty} r_N(x,y) = r(x,y) \in (0,\infty)$, $x\not =
y$. We assume, without loss of generality, that $a(x,y)=0$ for some
$x\not =y$. We assume, furthermore, that products of the rates
$\lambda_N(j)$ are comparable. More precisely, we suppose that for any
$(k_1, \dots, k_{\mf n}) \in \bb Z^{\mf n}$,
\begin{equation}
\label{30}
\lim_{N\to \infty} \prod_{i=1}^{\mf n} \lambda_N(i)^{k_i} \;=\; C_0 
\end{equation}
for some constant $C_0\in [0,\infty]$ which depends on $(k_0, \dots,
k_{\mf n})$.

Fix two states $x$, $y$ in $E$. By irreducibility, there exits a path
$x=x_0, x_1, \dots, x_n= y$ such that $n \le |E|$, $R_N(x_i,
x_{i+1})>0$, $0\le i<n$. By the detailed balance condition,
\begin{equation}
\label{33}
\frac{\mu_N(y)}{\mu_N(x)} \;=\; \prod_{i=0}^{n-1} 
\frac{R_N(x_i,x_{i+1})}{R_N(x_{i+1},x_i)} \;\cdot
\end{equation}
It follows from assumptions \eqref{24} and \eqref{30} that the
sequences $\{\mu_N(x) : N\ge 1\}$, $x\in E$, are comparable. In fact,
there exist $\mf m\ge 1$, sequences $M_N(\mf m) \prec \cdots \prec
M_N(1) \prec M_N(0) \equiv 1$, a function $b:E\to \{0, \dots, \mf m\}$
and a function $m: E \to (0,\infty)$ such that
\begin{equation}
\label{32}
\mu_N(x) \;=\; m_N(x) \, M_N(b(x))\;, \quad x\in E\;,
\end{equation}
where $\lim_{N\to\infty} m_N(x) = m(x) \in (0,\infty)$. We may choose
each sequence $M_N(j)$ to be equal to $\prod_{i=1}^{\mf n}
\lambda_i(N)^{k_i}$ for an appropriate choice of $(k_1, \dots, k_{\mf
  n})$ with $\sum_i |k_i|\le 4 |E|$.

Let $G_N: E\times E\to \bb R_+$ be given by $G_N(x,y) = \mu_N(x)
R_N(x,y)$ and note that $G_N$ is symmetric. As above, by \eqref{24}
and \eqref{30} the sequences $\{G_N(x,y) : N\ge 1\}$, $x\not = y \in
E$, are comparable. Moreover, there exist $\mf j \ge 1$, sequences
$G_N(\mf j) \prec \cdots \prec G_N(1) \prec G_N(0) \equiv 1$, a
function $c:E\times E \to \{0, \dots, \mf j\}$ and a function $g: E
\to (0,\infty)$ such that
\begin{equation}
\label{27}
G_N(x,y) \;=\; g_N(x,y) \, G_N(c(x,y))\;, \quad x\;, y \in E\;,
\end{equation}
where $\lim_{N\to\infty} g_N(x,y) = g(x,y) \in (0,\infty)$.  Here also
each sequence $G_N(j)$ may be chosen equal to $\prod_{i=1}^{\mf n}
\lambda_i(N)^{k_i}$ for an appropriate choice of $(k_1, \dots, k_{\mf
  n})$ with $\sum_i |k_i|\le 4 |E|+1$.
\smallskip

Denote by $\mb P_x^N$, $x\in E$, the probability measure on the path
space $D(\bb R_+, E)$ induced by the Markov process $\{\eta^N_t : t\ge
0\}$ starting from $x$. Expectation with respect to $\mb P_x^N$ is
denoted by $\mb E_x^N$.

For a subset $A$ of $E$, denote by $T_A$ the
hitting time of $A$:
\begin{equation*}
T_A \;=\; \inf\{t>0 : \eta^N_t \in A\}\; .
\end{equation*}
When $A$ is a singleton $\{x\}$, we denote $T_{\{x\}}$ by $T_x$.

For a proper subset $F$ of $E$, denote by $\{\eta^{F}_t : t\ge 0\}$
the trace of the Markov process $\{\eta^{N}_t : t\ge 0\}$ on $F$. We
refer to \cite[Section 2]{bl2} for a precise definition.  $\eta^{F}_t$
is a Markov process on $F$ and we denote by $R^{F}_N(x,y)$, $x$, $y\in
F$, its jump rates. Let $r^{F}_N(A,B)$, $A$, $B\subset F$, $A\cap
B=\varnothing$, be the average jump rates of $\eta^{F}_t$ from $A$
to $B$:
\begin{equation}
\label{35}
r^{F}_N(A,B) \;=\; \frac{1}{\mu_N(A)}  \sum_{x\in A} 
\mu_N(x) \sum_{y\in B} R^{F}_N(x,y)\;.
\end{equation}

The main theorem of this article describes all metastable behaviors of
the process $\{\eta^N_t : t\ge 0\}$.

\renewcommand{\theenumi}{\arabic{enumi}}
\renewcommand{\labelenumi}{({\bf \theenumi})}

\begin{theorem}
\label{s19} 
There exist $\mf M \ge 1$, sequences $\{\theta_N(k) : N\ge 1\}$, $1\le
k\le \mf M$, $1 \prec \theta_N(1) \prec \cdots \prec \theta_N(\mf M)$,
and, for each $1\le k\le \mf M$, a partition $\mc E^{(k)}_1, \dots,
\mc E^{(k)}_{\nu(k)}$, $\Delta_k$ of the state space $E$ such that for
all $1\le k\le \mf M$:

\renewcommand{\theenumi}{\arabic{enumi}}
\renewcommand{\labelenumi}{({\bf P\theenumi})}

\begin{enumerate}
\item $1< \nu(k)<\nu(k-1)$, $k\ge 2$.

\item For $k\ge 2$, $1\le i\le \nu(k)$, $\mc E^{(k)}_i = \cup_{a\in
    I_{k,i}} \mc E^{(k-1)}_a$, where $I_{k,1}, \dots, I_{k,\nu(k)}$
  are disjoint subsets of $\{1, \dots, \nu(k-1)\}$.

\item For $1\le i\le \nu(k)$, $\mu_N(x) \approx \mu_N(\mc E^{(k)}_i)$ for
  all states $x$ in $\mc E^{(k)}_i$.

\item Let $\mc E^{(k)} = \cup_{i=1}^{\nu(k)} \mc E^{(k)}_i$.  For all
  $1\le i \not = j \le \nu(k)$, the following limits exist
\begin{equation*}
\mf r_k (i,j) \;:=\;
\lim_{N\to \infty} \theta_N(k) \, r^{\mc E^{(k)}}_N(\mc E^{(k)}_i, \mc
E^{(k)}_j) \;.
\end{equation*}

\item Property {\rm ({\bf M1}')} of metastability holds: For every
  $1\le i\le \nu(k)$, every state $x$ in $\mc E^{(k)}_i$ and
  every $\delta>0$,
\begin{equation*}
\lim_{N\to \infty} \max_{y\in \mc E^{(k)}_i} \mb P^N_y \big[
T_{x} > \delta \theta_N(k) \big] \;=\; 0\;. 
\end{equation*}

\item Property {\rm ({\bf M2})} of metastability holds: Let $\Psi_k :
  \mc E^{(k)} \to \{1, \dots, \nu(k)\}$ be given by
\begin{equation*}
\Psi_k (x) \;=\; \sum_{i=1}^{\nu(k)} i\, \mb 1\{ x\in \mc E_i^{(k)} \}\; .
\end{equation*}
Denote by $\{\eta^{N, k}_t : t\ge 0\}$ the trace of the process
$\{\eta^{N}_t : t\ge 0\}$ on $\mc E^{(k)}$. For every $1\le i\le
\nu(k)$, $x\in \mc E_i^{(k)}$, under the measure $\mb P^N_x$, the
blind speeded up (non-Markovian) process $X^{N,k}_{t} =
\Psi_k(\eta^{N, k}_{t \theta_N(k)})$ converges to the Markov process
on $\{1\, \dots, \nu(k)\}$ starting from $i$ and characterized by the
rates $\mf r_k(i,j)$.

\item Property {\rm ({\bf M3}')} of metastability holds: For every
  $t>0$,
\begin{equation*}
\lim_{N\to \infty} \max_{x\in E} \, \mb E^N_x \Big[
\int_0^t \mb 1\{ \eta^N_{s\theta_N(k)} \in \Delta_k\} \, ds  \Big] \;=\; 0\;. 
\end{equation*}

\end{enumerate}
\end{theorem}

The sets $\mc E^{(k)}_i$, $1\le i\le \nu(k)$, are called the
metastable states at level $k$ or, simply, $k$-metastates. Property
(P2) asserts that as we pass from a description in the time scale
$\theta_N(k-1)$ to a description in the longer time scale
$\theta_N(k)$, the new metastates are larger and obtained as unions of
$(k-1)$-metastates. Moreover, by property (P3), all states in any
metastable set have measure of the same magnitude.

Condition (P5) asserts that, with a probability increasing to one, any
state in a metastable set is visited before the process leaves the
metastable set. The process therefore thermalizes in the metastable
state or, in other words, reaches a local equilibrium, before leaving
the metastable state.

Condition (P7) states that on the time scale $\theta_N(k)$, the time
spent outside the union of all metastates is negligible. To examine
the behavior of the process in this time scale we may therefore
restrict our attention to the trace process $\{\eta^{N,k}_t : t\ge
0\}$ speeded up by $\theta_N(k)$. 

It follows from properties (P5) and (P6) that the speeded up trace
process $\{\eta^{N,k}_{t\theta_N(k)} : t\ge 0\}$ thermalizes in each
metastable set $\mc E^{(k)}_i$ and then, at the end of an exponential
time, jumps to another metastable set. By property (P4) the rate at
which the process jumps from one metastable set to another is given by
the asymptotic mean rate at which the speeded up trace process
jumps. Theorem \ref{s19} gives, therefore, a complete description of
the evolution of the process on each time scale $\theta_N(k)$.

\begin{remark}
\label{mr1}
{\rm
In order to apply this result to concrete models, we proceed as
follows. Consider the Markov process on $E$ obtained by suppressing
all jumps $R_N(x,y)$ of magnitude smaller than $1$: $R_N(x,y) \prec
1$. Note that this Markov process may not be irreducible. Denote by
$\nu=\nu(1)$ the number of irreducible classes and by $\mc E_1, \dots,
\mc E_\nu$ the irreducible classes. These sets are the $1$-metastates.
Let 
\begin{equation}
\label{h1}
\theta_{N,i} \;=\; \frac{\mu_N(\mc E_i)}
{\Cap_N(\mc E_i , \breve{\mc E}_i)}\;,\quad 1\le i\le \nu\;,
\end{equation}
where $\Cap_N(A,B)$ represents the capacity between $A$ and $B$,
defined in Section \ref{ssec1}, and $\breve{\mc E}_i = \cup_{j\not =i}
\mc E_j$. By Proposition \ref{s24} the sequences $(\theta_{N,i} :N\ge 1)$,
$1\le i\le \nu$, are comparable. Let $\theta_{N} = \theta_N(1) =
\min\{\theta_{N,i}: 1\le i\le \nu\}$. Since the sequences are
comparable the following limits exist
\begin{equation*}
\lambda (i) \;=\;
\lim_{N\to\infty}\frac{\theta_{N}}{\theta_{N,i}}\;\in\;
[0,\infty)\;, \quad 1\le i\le \nu\;. 
\end{equation*}
By Lemma \ref{s29} and the first remark formulated at the end of
Section \ref{ssec2}, for every $1\le i\not = j\le \nu$, the limits
below also exist and do not depend on $x\in \mc E_i$:
\begin{equation}
\label{h2}
p(i,j) \;=\; \lim_{N\to\infty} \mb P^N_x \big[T_{\mc E_j} =
T_{\breve{\mc E}_i} \big]\;. 
\end{equation}
By \eqref{41a}, $\mf r_1(i,j) = \lambda(i) p(i,j)$.

Hence, the characterization of the $1$-metastates is very simple and
the computation of $\theta_{N,i}$ (the depth of the valley $\mc E_i$,
as we shall see) is feasible. This computation provides the slowest
time scale $\theta_N(1)$ in which a metastable behavior is
observed. To determine the exact asymptotic evolution in this time
scale, we need to compute \eqref{h2} which may be difficult or even
impossible. In several cases, however, one may at least discriminate
the pairs $(i,j)$ for which $\mf r_1(i,j)$ is strictly positive. This
permits to iterate the argument and gives an imprecise picture of the
metastable behavior. In the time scale $\theta_N(1)$ the process
remains in the $1$-metastate $\mc E_i$ for a rate $\lambda(i)$
exponential time at the end of which it jumps to one of the remaining
metastates such that $p(i,j)>0$.

Consider the Markov process on $\{\mc E_1, \dots, \mc E_\nu\}$
(instead of $\{1, \dots, \nu\}$) with rates $\mf r_1(i,j)$ and denote
by $\nu(2)$ the number of its irreducible classes, and by $\mc
E^{(2)}_1, \dots, \mc E^{(2)}_{\nu(2)}$ the irreducible classes. Note
that properties (P1) and (P2) are fulfilled and that we need only to
know if $p(i,j)$ is strictly positive or not to determine the
irreducible classes.  Compute \eqref{h1} and \eqref{h2} for this new
class of sets to obtain the second time scale $\theta_N(2)$ and the
rates $\mf r_2(i,j)$. Iterating this scheme we completely characterize
the metastable behavior of the Markov process.}
\end{remark}

We conclude this section with some comments.  In statistical mechanics
models, the rates $R_N(x,y)$ are usually exponential and given by
$e^{N h(x,y)}$ for some function $h:E\times E \to \bb R$. Assumptions
\eqref{24}, \eqref{30} are trivially satisfied in this context.

In some models examined in statistical mechanics the time scales
$\theta_N(k)$, $1\le k < \mf M$, correspond to the nucleation phase of
the system, which may be very intricate even for simple dynamics due
to the variety of valleys and the complexity of their geometries. In
most case, one only investigates the behavior in the largest time
scale, $\theta_N(\mf M)$, where one observes either an exponential
jump from a metastable to a stable state, or a Markovian evolution
among competing metastable states.

We prove in \eqref{10} and \eqref{37} that the process never jumps
from a metastable set to another metastable set which has probability
of smaller order: $\mf r_k(i,j)=0$ if $\mu_N(\mc E^{(k)}_j) \prec
\mu_N(\mc E^{(k)}_i)$.

\section{The Ising model at low temperature}
\label{si1}

To illustrate the methods presented in the first part of this article,
we examin in this section the metastable behavior of the Ising model
at low temperature following Neves and Schonmann \cite{ns1}.

We consider the two dimensional nearest neighbor ferromagnetic Ising
model on a finite torus $\Lambda_L = \bb T_L \times \bb T_L$, $L\ge
1$, where $\bb T_L = \{1, \dots, L\}$ is the discrete one-dimensional
torus with $L$ points. The Hamiltonian is written as
\begin{equation*}
\bb H (\sigma) \; =\; -\; \frac 12 \sum_{\<x,y\>} \sigma(x) \sigma(y)
\;-\; \frac h2 \sum_{x\in \Lambda_L} \sigma(x) \; ,
\end{equation*}
where $\sigma (x) \in \{ -1, 1\}$, the first sum runs over the pairs
of nearest neighbors sites of $\Lambda_L$, counting each pair only
once, and the second is taken over $\Lambda_L$. We will always
consider $h > 0$.

At inverse temperature $\beta>0$, the Gibbs measure $\mu_\beta$
associated to the Hamiltonian $\bb H$ is given by
\begin{equation*}
\mu_\beta (\sigma) \;=\; \frac 1{Z_\beta} e^{-\beta \bb H(\sigma)}\;,
\end{equation*}
where $Z_\beta$ is the normalizing partition function.

The Glauber dynamics on the state space $\Omega = \Omega_L = \{ -1,
1\}^{\Lambda_L}$, also known as the Ising model, is the
continuous-time Markov process whose generator $L_\beta$ acts on
functions $f:\Omega \to \bb R$ as
\begin{equation*}
(L_\beta f)(\sigma)\;=\;  \sum_{x\in \Lambda_L} 
c(x,\sigma) \, [f(\sigma^{x}) - f(\sigma)]\;,
\end{equation*}
where $\sigma^{x}$ is the configuration obtained from $\sigma$ by
flipping the spin at $x$:
\begin{equation*}
\sigma^{x}(y) \;=\; 
\begin{cases}
\sigma(y) & \text{if $y\not = x$}, \\
- \sigma(x) & \text{if $y = x$}, \\
\end{cases}
\end{equation*}
where the rates $c(x,\sigma)$ are given by
\begin{equation*}
c (x,\sigma) \;=\; \exp\big\{-\beta \,
[\bb H(\sigma^{x}) - \bb H(\sigma)]_+ \big\}\;,
\end{equation*}
and where $a_+$, $a\in \bb R$, stands for the positive part of $a$:
$a_+ = \max\{a,0\}$.  The Markov process $\{\sigma^\beta_t : t\ge 0\}$
with generator $L_\beta$ is reversible with respect to the Gibbs
measures $\mu_\beta$, $\beta>0$, and ergodic. Denote by
$R_\beta(\sigma, \sigma')$ the rate at which the process jumps from
$\sigma$ to $\sigma'$ so that $R_\beta(\sigma, \sigma')$ vanishes
unless $\sigma'= \sigma^x$ for some $x\in \Lambda_L$, in which case
$R_\beta(\sigma, \sigma^x) = c(x,\sigma)$.

In this model the process jumps from a state $\sigma$ to the state
$\sigma^x$ at rate $1$ if $\mu_\beta(\sigma) \le
\mu_\beta(\sigma^x)$. In particular, by the detailed balance
condition, $\mu_\beta(\sigma) R_\beta (\sigma, \sigma^x) =
\min\{\mu_\beta(\sigma) , \mu_\beta(\sigma^x)\}$ so that $G_\beta
(\sigma, \sigma^x) = \min\{\mu_\beta(\sigma) , \mu_\beta(\sigma^x)\}$.

We examine in this section the metastable behavior of the Markov
process $\{\sigma^\beta_t : t\ge 0\}$ on $\Omega$ as the temperature
vanishes. To avoid less interesting cases, following \cite{ns1} we
assume from now on that $0<h<2$, that $2/h \not\in \bb N$ and that
$L>(n_0+1)^2+1$, where $n_0 = [2/h]$ and $[r]$ stands for the integer
part of $r$.

Let $I$ be an interval of the one dimensional torus $\bb T_L$.  The
sets $I\times \bb T_L$, $\bb T_L \times I\subset \Lambda_L$ are called
rings, while rectangles are subsets of the form $I\times J$, where
$I$, $J$ are non-empty proper intervals of $\bb T_L$.

To describe all metastable behaviors of the Ising model, we need to
define the time scales at which they occur, the metastable sets
associated to each time scale, and the asymptotic dynamics which
specifies at which rate the process jumps from one metastable state to
another. We start defining the $n_0+1$ time scales. For $1\le k\le
n_0-1$ let
\begin{equation*}
\theta_\beta (k) = e^{k \beta h}\;, 
\quad \theta_\beta (n_0) = e^{\beta (2-h)}\; , \quad 
\theta_\beta (n_0+1) = e^{\beta \, c(h)}\;,
\end{equation*}
where $c(h) = 4 (n_0+1) - h [(n_0+1)n_0+1]$.  Note that
$\theta_\beta(1) \prec \cdots \prec \theta_{\beta} (n_0+1)$.

The presentation of the metastable sets requires some notation.
Denote by $\Omega_o \subset \Omega$ the set of configurations whose
total jump rate $\sum_{x\in\Lambda_L} R_\beta(\sigma, \sigma^x)$
vanishes as $\beta\uparrow\infty$. This is the set of configurations
in which a negative spin has at most one positive neighbor and in
which a positive spin has at most two negative neighbors. This set
contains the configurations $+ \bf 1$, $- \bf 1$, which are the
configurations with all spins positive, negative, respectively, and
configurations formed by positive rectangles and rings of length and
width larger than $2$ in a background of negative spins.  In these
latter configurations, to fulfill the prescribed conditions the
positive rectangles and rings may not be at graph distance $2$.

For a configuration $\sigma$ in $\Omega_o$, denote by $\ell (\sigma)$
the smallest length or width of the positive rectangles of $\sigma$.
By convention, $\ell ( - \mb 1) = 0$, $\ell ( + \mb 1) = L$ and
$\ell(\sigma)=L$ if $\sigma$ contains no positive rectangles, but only
positive rings. Let $N_r(\sigma)$ be the number of positive
$\ell(\sigma)\times m$ rectangles of $\sigma$ for some
$m>\ell(\sigma)$, and let $N_s(\sigma)$ be the number of positive
$\ell (\sigma) \times \ell (\sigma)$ squares of $\sigma$.

We may now introduce the metastable states $\Omega_{o,k}$ appearing in
the time scale $\theta_\beta (k)$, $1\le k\le n_0+1$. For $1\le k\le
n_0$, let
\begin{equation*}
\Omega_{o,k} \;=\; \{\sigma \in \Omega_o : \ell(\sigma)>k\} \cup 
\{- \mb 1\} \;, \quad \Omega_{o,n_0+1} \;=\; \{+ \mb 1 , - \mb 1\} \;.
\end{equation*}
Note that $\Omega_{o} = \Omega_{o,1} \supset \cdots \supset
\Omega_{o,n_0+1}$. The metastables states appearing in the time scale
$\theta_\beta (k)$, $1\le k\le n_0+1$, are all the elements of
$\Omega_{o,k}$.

To depict how the process jumps from one metastable state to another
in the different time scales, we need to introduce several sets. We
use the terminology of graph theory to name some of them.  Denote by
$\bb D(\sigma)$ the set of direct successors of the configuration
$\sigma$ in $\Omega_o$, $\sigma \not = + \mb 1, - \mb 1$. If $\ell
(\sigma) =2$, $\bb D(\sigma)$ is the set of configurations obtained
from $\sigma$ by flipping all positive spins from one of the two sides
of length $2$ of a positive $2 \times m$ rectangle, $m> 2$, and of
configurations obtained from $\sigma$ by flipping all spins of a
positive $2\times 2$ square of $\sigma$.  Clearly, $|\bb D(\sigma)| =
2N_r(\sigma) + N_s(\sigma)$.  When $3\le \ell (\sigma) \le n_0$, $\bb
D(\sigma)$ is the set of configurations obtained from $\sigma$ by
flipping all positive spins from one of the sides of length
$\ell(\sigma)$ of a positive $\ell(\sigma) \times m$ rectangle, $m\ge
\ell(\sigma)$. In this case, $|\bb D(\sigma)| = 2N_r(\sigma) + 4
N_s(\sigma)$. For $\ell (\sigma) > n_0$, $\bb D(\sigma)$ is the set of
configurations obtained by first flipping a negative spin from a site
which has a neighbor site with a positive spin, and then flipping in
any order all negative spins surrounded by two positive spins. Note
that in this latter case two or more positive rectangles may be
replaced by the smallest rectangle which contains them all. For this
reason an exact description of the direct successors of a
configuration in the case $\ell (\sigma) > n_0$ is more complicated.

For $\sigma \in \Omega_o$, $\sigma \not = \pm \mb 1$, denote by $\bb
W(\sigma)$ the set of saddle points of the configuration $\sigma$. For
$2\le \ell (\sigma) \le n_0$, $\bb W(\sigma)$ is the set of
configurations obtained from $\sigma$ by flipping $\ell(\sigma)-1$
positive spins from a side of length $\ell(\sigma)$ of a positive
$\ell(\sigma) \times m$ rectangle of $\sigma$, $m\ge \ell(\sigma)$.
Note that $|\bb W(\sigma)| = 2 \ell(\sigma) N_r(\sigma) + 4
\ell(\sigma) N_s(\sigma)$ for $3\le \ell (\sigma) \le n_0$ and $|\bb
W(\sigma)| = 4N_r(\sigma) + 4 N_s(\sigma)$ for $\ell(\sigma)=2$.  For
$\ell (\sigma) > n_0$, $\bb W(\sigma)$ consists of the set of
configurations obtained from $\sigma$ by flipping a negative spin from
a site which has one neighbor with a positive spin so that $|\bb
W(\sigma)|$ is equal to the sum of the perimeters of the positive
rectangles of $\sigma$ added to $2L$ times the number of positive
rings of $\sigma$.

For $\ell (\sigma) > n_0$, let $\bb W(\sigma, \sigma')$, $\sigma \in
\Omega_o$, $\sigma'\in \bb D(\sigma)$, be the subset of $\bb
W(\sigma)$ of all configurations which attain $\sigma'$ by flipping in
any order all negative spins surrounded by two positive spins, and let
$\bb W_j(\sigma)$, $1\le j\le 3$, be the configurations $\xi$ of $\bb
W(\sigma)$ with the following property.  The site where $\xi$ differs
from $\sigma$ has $3$ neighbors with negative spins.  Among these
three neighbors, $j$ sites have two neighbors with positive spins. The
case $j=3$ occurs when the configuration has two positive rectangles
or rings at distance $3$. Let $\bb W_j(\sigma, \sigma') = \bb
W_j(\sigma) \cap \bb W(\sigma, \sigma')$.

Fix a configuration $\sigma \in \Omega_o$ and let $\Omega_\sigma =
\Omega_o \setminus \{\sigma\}$. Recall that we denote by $T_A$ the
hitting time of a set $A\subset \Omega$. We prove in Lemma \ref{t03}
that $\mb P^\beta_\sigma [T_{\bb D(\sigma)} = T_{\Omega_\sigma}]$
converges, as $\beta\uparrow\infty$, to $1$ and that the process
reaches $\sigma'$ by first visiting a configuration of $\bb
W(\sigma)$.

Denote by $\bb S(\sigma)$ the set of successors of the configuration
$\sigma$ in $\Omega_o$, $\sigma \not = + \mb 1, - \mb 1$. The
difference between a successor and a direct successor is that the
critical length $\ell(\sigma')$ of a successor $\sigma'$ may not be
smaller than the one of the original configuration: $\ell(\sigma')\ge
\ell (\sigma)$. If $\ell (\sigma) =2$ or $\ell (\sigma) >n_0$, the set
of successors coincides with the set of direct successors: $\bb
S(\sigma) = \bb D(\sigma)$.  However, if $3\le \ell (\sigma) \le n_0$,
$\bb S(\sigma)$ is the set of configurations obtained from $\sigma$ by
flipping all positive spins from one of the two sides of length
$\ell(\sigma)$ of a positive $\ell(\sigma) \times m$ rectangle of
$\sigma$, $m> \ell(\sigma)$, and of configurations obtained from
$\sigma$ by flipping all spins of a positive $\ell(\sigma)\times
\ell(\sigma)$ square of $\sigma$.

The probability measure $p$ introduced below describes how the process
jumps from one metastable state to another in the appropriate time
scales.  For each configuration $\sigma \in \Omega_o$, define the
probability measure $p(\sigma, \,\cdot\,)$ on $\Omega_o$ as follows.
Let $p(\sigma, \sigma')=0$ for $\sigma'\not\in \bb S(\sigma)$. For
$\sigma'\in \bb S(\sigma)$ and $\ell (\sigma) =2 \le n_0$, let
\begin{equation}
\label{e04}
p(\sigma, \sigma') \;=\;
\begin{cases}
(8/3) [2 N_r + (8/3) N_s]^{-1} & \text{for $\sigma' \in \bb S_s(\sigma)$,} \\
[2 N_r + (8/3) N_s]^{-1} & \text{otherwise,}
\end{cases}
\end{equation}
where $\bb S_s(\sigma) \subset \bb S(\sigma)$ is the set of
configurations obtained from $\sigma$ by flipping all spins in a
positive $2\times 2$ square of $\sigma$. For $\sigma'\in \bb
S(\sigma)$ and $3\le \ell (\sigma) \le n_0$, let
\begin{equation}
\label{e05}
p(\sigma, \sigma') \;=\; 
\begin{cases}
4[2 N_r + 4 N_s]^{-1} & \text{for $\sigma' \in \bb S_s(\sigma)$,} \\
[2 N_r + 4 N_s]^{-1} & \text{otherwise,}
\end{cases} 
\end{equation}
where $\bb S_s(\sigma) \subset \bb S(\sigma)$ is the set of
configurations obtained from $\sigma$ by flipping all spins in a
positive $\ell(\sigma)\times \ell(\sigma)$ square of $\sigma$.
Finally, for $\sigma'\in \bb S(\sigma)$ and $\ell(\sigma)>n_0$, let
\begin{equation}
\label{e06}
p(\sigma, \sigma') \;=\; \frac{\sum_{j=1}^3 \frac j{j+1} \,
|\bb W_j (\sigma, \sigma')|} {\sum_{j=1}^3 \frac j{j+1} \,
|\bb W_j(\sigma)|}\;\cdot
\end{equation}

It remains to describe the rates at which the process leaves a
metastable state in the different time scales.  Let
$\theta:\Omega_o\setminus \{- \mb 1, +\mb 1\}\to\bb R_+$ be given by
\begin{equation}
\label{e07}
\theta(\sigma) \;=\;
\begin{cases}
(2/3) N_s(\sigma) + 2 N_r(\sigma) & \text{if $\ell =2\le n_0$,} \\
\frac{2\ell -1}{3\ell} \, |\bb W(\sigma)| & \text{if $3\le \ell \le
  n_0$,} \\
(1/2) \, |\bb W_1(\sigma)| \;+\; (2/3) \, |\bb W_2(\sigma)|
\; +\; (3/4) \, |\bb W_3(\sigma)| & \text{if $\ell > n_0$.} 
\end{cases}
\end{equation}

We are now in a position to state the first main result of this
section.  Fix $1\le k\le n_0+1$ and denote by $\sigma^{\beta,k}_{t}$
the trace of the process $\sigma^\beta_t$ on $\Omega_{o,k}$. Recall
that $\sigma^{\beta,k}_{t}$ is a Markov process on $\Omega_{o,k}$.

\begin{theorem}
\label{t04}
Fix $1\le k\le n_0$. As $\beta\uparrow\infty$, the Markov process
$\sigma^{\beta,k}_{t \theta_\beta (k)}$ converges to the Markov
process on $\Omega_{o,k}$ with jump rates $\mf r$ given by
\begin{equation*}
\mf r(\sigma, \sigma') \;=\;
\begin{cases}
\theta(\sigma) p(\sigma, \sigma') &\text{if $\sigma \in \Omega_{o,k}
\setminus \Omega_{o,k+1}$,} \\
0 &\text{if $\sigma \in \Omega_{o,k+1}$.}
\end{cases}
\end{equation*}
Moreover, the time spent outside $\Omega_{o,k}$ by the process
$\sigma^{\beta,k}_{t \theta_\beta (k)}$ is negligible: for all $t>0$
and $\sigma\in \Omega$,
\begin{equation*}
\lim_{\beta\to\infty} \mb E^\beta_\sigma \Big[ \int_0^t \mb 1\{ 
\sigma^{\beta,k}_{s \theta_\beta (k)} \not\in \Omega_{o,k} \}\, ds 
\Big]\;=\; 0\;. 
\end{equation*}
\end{theorem}

Fix a configuration $\sigma \in \Omega_{o,k}$, $1\le k\le n_0-1$, and
consider asymptotic behavior, as the temperature vanishes, of the
trace process $\sigma^{\beta,k}_{t}$ in the time scale $\theta_\beta
(k)$ starting from $\sigma$. Theorem \ref{t04} states that if
$\ell(\sigma)>k+1$, the configuration $\sigma$ is an absorbing point
for the asymptotic dynamics, while if $\ell(\sigma)=k+1$, the
asymptotic dynamics visits a sequence of configurations where each
element of the sequence differs from the previous one either by
flipping all positive spins of one of the two sides of length $k+1$ of
a $(k+1)\times m$ positive rectangle, $m>k+1$, or by flipping all
spins of a $(k+1)\times (k+1)$ positive square. After a finite number
of jumps, the process reaches a configuration whose positive
rectangles have all sides larger than $k+1$ and stays there forever.

For a configuration $\sigma \in \Omega_{o,n_0}$, Theorem \ref{t04}
states that in the time scale $\theta_\beta (n_0)$ the trace process
$\sigma^{\beta,n_0}_{t}$ sees its positive rectangles and rings to
increase gradually until the configurations $+ \mb 1$ is reached.

This result describes therefore the behavior of the Ising model in the
intermediate scales where first the small positive droplets are
removed and then the large positive droplets increase to eventually
occupy all space. To complete the picture of the metastable behavior
of the model it remains to specify how the process jumps from the
configuration $-\mb 1$ to the configuration $+ \mb 1$.

Denote by $\bb W(-\mb 1)$ the set of configurations which have a
positive $(n_0+1)\times n_0$ rectangle and an extra positive spin
which has a positive neighbor sitting on one of the sides of length
$n_0+1$ of the positive rectangle, all others spins being
negative. All configurations of $\bb W(-\mb 1)$ have the same measure.
Denote by $\bb W_1(-\mb 1)$ the configurations of $\bb W(-\mb 1)$
whose extra positive spin is next to the corner of the positive
rectangle and by $\bb W_2(-\mb 1)$ the other configurations of $\bb
W_1(-\mb 1)$. Let
\begin{equation*}
\theta(-\mb 1) \;=\; (1/2) \, |\bb W_1(-\mb 1)| \;+\; (2/3) \, |\bb
W_2(-\mb 1)| \;.
\end{equation*}

\begin{theorem}
\label{t08}
As $\beta\uparrow\infty$, the Markov process $\sigma^{\beta,n_0+1}_{t
  \theta_\beta (n_0+1)}$ converges to the Markov process on $\{-\mb 1,
+ \mb 1\}$ in which $+\mb 1$ is an absorbing state and which jumps
from $-\mb 1$ to $+ \mb 1$ at rate $\theta(-\mb 1)$.  Moreover, the
time spent outside $\{-\mb 1, + \mb 1\}$ by the process
$\sigma^{\beta,n_0+1}_{t \theta_\beta (n_0+1)}$ is negligible: for all
$t>0$ and $\sigma\in \Omega$,
\begin{equation*}
\lim_{\beta\to\infty} \mb E^\beta_\sigma \Big[ \int_0^t \mb 1\{ 
\sigma^{\beta,n_0+1}_{s \theta_\beta (n_0+1)} \not = \pm \mb 1\}\, ds 
\Big]\;=\; 0\;. 
\end{equation*}
\end{theorem}

\section{Capacities and hitting times}
\label{ssec1}

Denote by $D_N$ the Dirichlet form associated to the generator of the
Markov process introduced in Section \ref{sec3}:
\begin{equation*}
D_N(f)\;=\; \sum_{\{x,y\}\subset E} \mu_N(x) \, R_N(x,y) \, 
\{f(y) - f(x)\}^2 \;, \quad f:E\to \bb R\;, 
\end{equation*}
where in the sum on the right hand side each bond $\{x,y\}$ is counted
only once. Let $\Cap_N(A,B)$, $A$, $B\subset E$, $A\cap
B=\varnothing$, be the capacity between $A$ and $B$:
\begin{equation}
\label{25}
\Cap_N(A,B) \;=\; \inf_f D_N(f)\;,
\end{equation}
where the infimum is carried over all functions $f:E\to \bb R$ such
that $f(x) =1$ for all $x\in A$, and $f(x) =0$ for all $x\in B$.

We may compute the order of magnitude of the capacity between two
disjoint subsets of $E$.  A
self-avoiding path $\gamma$ from $A$ to $B$, $A$, $B\subset E$, $A\cap
B = \varnothing$, is a sequence of sites $(x_0, x_1, \dots, x_n)$ such
that $x_0\in A$, $x_n\in B$, $x_i \not = x_j$, $i\not =j$,
$R_N(x_i,x_{i+1})>0$, $0\le i <n$. Denote by $\Gamma_{A,B}$ the set of
self-avoiding paths from $A$ to $B$ and let
\begin{equation*}
G_N(A,B) \;:=\; \max_{\gamma\in \Gamma_{A,B}} G_N(\gamma)\;, \quad
G_N(\gamma) \;:=\; \min_{0\le i<n} G_N(x_i,x_{i+1})\;.
\end{equation*}
Note that there might be more than one optimal path and that
$G_N(\{x\},\{y\}) \ge G_N(x,y)$, with possibly a strict inequality.

We shall say that a bond $(x_p,x_{p+1})$ of a path $\gamma = (x_0,
x_1, \dots, x_n)$ is \emph{critical} if
\begin{equation*}
G_N(x_p,x_{p+1}) \;=\; \min_{0\le i<n} G_N(x_i,x_{i+1})
\;=\; G_N(\gamma) \;.
\end{equation*}

Note that for every disjoint sets $A$, $B$, $C$,
\begin{equation}
\label{11}
G_N(A,B\cup C) \;=\; \max\{G_N(A,B), G_N(A,C) \}\;.
\end{equation}
Indeed, the left hand side is greater or equal than the right hand
side because $G_N(A,D)\le G_N(A,D')$ if $D\subset D'$. On the other
hand, there exists a self-avoiding path $\gamma =(x_0, \dots, x_n)$
from $A$ to $B\cup C$ such that $G_N(A,B\cup C) = G_N(\gamma)$. We may
assume without loss of generality that $x_n$ belongs to $B$.  Hence,
$\gamma$ is a self-avoiding path from $A$ to $B$ and $G_N(\gamma) \le
G_N(A,B)$, which proves \eqref{11}.

\begin{lemma}
\label{s03}
Fix $A$, $B\subset E$ such that $A\cap B = \varnothing$. The capacity
$\Cap_N(A,B)$ is of the same magnitude of $G_N(A,B)$. There exists a
positive and finite constant $C_1$, depending only on $E$ and on the
limiting rates $r(x,y)$, such that
\begin{equation*}
  C_1^{-1} \;\le\; \frac{\Cap_N(A,B)}{G_N(A,B)} \;\le\; C_1
\end{equation*}
for all $N$ sufficiently large.
\end{lemma}

\begin{proof}
Fix two subsets $A$, $B$ of $E$ such that $A\cap B = \varnothing$. We
shall obtain an upper and a lower bound for $\Cap_N(A,B)$. We start
with a lower bound. Fix a self-avoiding path $\gamma= (x_0, x_1,
\dots, x_n)$ in $\Gamma_{A,B}$ such that $G_N(A,B) = \min_{0\le i <n}
G_N(x_i,x_{i+1})$. This path always exists because the number of paths
is finite. For any function $f:E\to \bb R$,
\begin{equation*}
D_N(f) \;\ge\; \sum_{i=0}^{n-1} G_N(x_i,x_{i+1}) [f(x_{i+1})-f(x_i)]^2
\; .
\end{equation*}
Therefore, minimizing over all $f:E\to \bb R$ such that $f(x)=1$,
$x\in A$, $f(y)=0$, $y\in B$, we get that $\Cap_N(A,B)$ is bounded
below by
\begin{equation*}
\inf_f  \sum_{i=0}^{n-1} G_N(x_i,x_{i+1}) [f(x_{i+1})-f(x_i)]^2\;,
\end{equation*}
where the infimum is carried over all functions $f:\{x_0, \dots, x_n\}
\to \bb R$ such that $f(x_0)=1$, $f(x_n)=0$. A simple computation
shows that this expression is equal to
\begin{equation*}
\Big\{ \sum_{i=0}^{n-1} \frac 1{G_N(x_i,x_{i+1})} \Big\}^{-1} \;,
\end{equation*}
which is bounded below, for $N$ large, by $C_1 \min_{0\le i <n}
G_N(x_i,x_{i+1})$ for some positive constant $C_1$ depending only on
$E$ and the asymptotic rates $r(x,y)$. By the definition of the path
$\gamma= (x_0, x_1, \dots, x_n)$, $\min_{0\le i <n} G_N(x_i,x_{i+1}) =
G_N(A,B)$, which proves the lower bound for the capacity.

We now turn to the upper bound. Denote by $\mf B_N$ the set of bonds
$(x,y)$ such that $G_N(x,y) > G_N(A,B)$. The state space $E$ can be
written as the disjoint union of maximal connected components. More
precisely, there exist disjoint subsets $A_1, \dots, A_m$ of $E$,
possibly singletons, fulfilling the next three conditions:
\begin{itemize}
\item $E = \cup_{1\le j\le m} A_j$;
\item for any $x$, $y \in A_j$, there exists a path $\gamma = (x=x_0,
  x_1, \dots, x_p=y)$ such that $G_N(x_i,x_{i+1}) > G_N(A,B)$ for
  all $0\le i <p$;
\item for any $x\in A_j$, $y\in A_k$, $j\not = k$, there does not
  exist a path $\gamma = (x=x_0, x_1, \dots, x_p=y)$ such that
  $G_N(x_i,x_{i+1}) > G_N(A,B)$ for all $0\le i <p$.
\end{itemize}

Note that if $A_j\cap A\not = \varnothing$ then $A_j\cap B =
\varnothing$. Otherwise, there would be a self-avoiding path $(x_0,
\dots, x_n)$ from $A$ to $B$ such that $G_N(x_i,x_{i+1})>G_N(A,B)$ for
all $0\le i<n$, in contradiction with the definition of $G_N(A,B)$.

Consider a self-avoiding path $\gamma = (x_0, x_1, \dots, x_n)$ in
$\Gamma_{A,B}$ such that $G_N(A,B) = \min_{0\le i <n}
G_N(x_i,x_{i+1})$. The path $\gamma$ may have bonds $(x_i,x_{i+1})$ in
$\mf B_N$. We claim, however, that there exists a bond $(x_q,x_{q+1})$
in $\gamma$ such that $G_N(x_q,x_{q+1}) = G_N(A$, $B)$ and such that
there is no maximal connected component $A_j$ of $E$ such that
\begin{equation}
\label{02}
A_j \cap [A\cup \{x_0, \dots, x_q\}] \not = \varnothing
\text{ and } A_j \cap [B\cup \{x_{q+1}, \dots, x_n\}] \not = \varnothing \;.
\end{equation}

To prove this claim, let $L\ge 1$ be the number of critical bonds in
$\gamma$ and fix a critical bond $(x_p,x_{p+1})$ for which \eqref{02}
does not hold. There exists therefore a maximal connected component
$A_j$ of $E$ such that $A_j \cap [A\cup \{x_0, \dots, x_p\}] \not =
\varnothing$ and $A_j \cap [B\cup \{x_{p+1}, \dots, x_n\} ] \not =
\varnothing$. By overlapping the bond $(x_p,x_{p+1})$ by a path in
$A_j$, we construct a new self-avoiding path $\gamma' = (x'_0, \dots,
x'_{n'})$ from $A$ to $B$ with possibly different initial or final
point which avoids the bond $(x_p,x_{p+1})$.

Since all bonds which belong to $\gamma'$ and not to $\gamma$ are in
$\mf B_N$ and since $G_N(x_i, x_{i+1})$ $\ge G_N(x_p, x_{p+1}) =
G_N(A,B)$, $0\le i<n$, $G_N(x'_i, x'_{i+1}) \ge G_N(A,B)$ for all
$0\le i<n'$. On the other hand, since $\gamma'$ is a self-avoiding path
from $A$ to $B$, $\min_{0\le i<n'} G_N(x'_i, x'_{i+1}) \le
G_N(A,B)$. Hence, $\min_{0\le i<n'} G_N(x'_i, x'_{i+1}) = G_N(A,B)$.

On the other hand, since all bonds which belong to $\gamma'$ and not
to $\gamma$ are in $\mf B_N$ and since $(x_p, x_{p+1})$ does not
belong $\gamma'$, the number of critical bonds of $\gamma'$ is at most
$L-1$. It might be smaller than $L-1$ if the set $A_j$ overlaps more
than one critical bond of $\gamma$.

If the new path $\gamma'$ fulfills condition \eqref{02}, the claim is
proved. If it does not, we apply the algorithm again. Since the
algorithm reduces the number of critical bonds by at least one, after
a finite number of iterations we obtain a path satisfying \eqref{02}
as claimed.

We now define a function $f$ equal to $1$ on the set $A$, equal to $0$
on the set $B$ and we show that the Dirichlet form of $f$ is bounded
by $C_1 G_N(A,B)$ for some finite constant $C_1$ which depends only on
$E$. Let $(x_p,x_{p+1})$ be a critical bond of a path $\gamma =(x_0,
\dots, x_n)$ satisfying condition \eqref{02}. Define $f:E\to \bb R$ as
follows. Let $f(x)=1$ for $x\in A$, $f(y)=0$, $y\in B$. Define $f$ on
$\gamma$ as $f(x_i)=1$, $0\le i\le p$, $f(x_j)=0$, $p+1\le i\le n$. On
each maximal connected component $A_j$ which intersects $A\cup \{x_0,
\dots, x_p\}$, set $f=1$. Similarly, on each maximal connected
component $A_k$ which intersects $\{x_{p+1}, \dots, x_n\} \cup B$ set
$f=0$. Property \eqref{02} ensures that this can be done. On the
remaining sites we define $f$ to be a fixed arbitrary constant
$\omega$. Note that with this definition $f$ is constant on each
maximal connected component $A_k$.

It remains to examine the Dirichlet form of $f$. There are three types
of nonvanishing terms in this Dirichlet form. The first one is
$G_N(x_p,x_{p+1}) = G_N(A,B)$. The second and third types are
expressions of the form $G_N(x,y) (1-\omega)^2$, $G_N(x,y) \omega^2$,
where $(x,y)$ does not belong to $\mf B_N$. In particular, the
contribution to the Dirichlet form of $f$ of these expressions is
bounded by $C_1 G_N(A,B) \{\omega^2 + (1-\omega)^2\}$ for some finite
constant which depends only on $E$. This proves that $D_N(f) \le C_1
G_N(A,B)$. Since $f$ is equal to $1$ on the set $A$ and is equal to
$0$ on the set $B$, $\Cap_N(A,B) \le C_1 G_N(A,B)$, which proves the
lemma.
\end{proof}

This lemma presents a typical estimation of asymptotic capacities. We
first obtain a lower bound of the Dirichlet form, uniform over all
functions $f$, by disregarding some bonds. Then, we prove an upper
bound for a specific candidate, believed to be close to the optimal
function in view of the proof of the lower bound. This time, however,
no bond can be neglected in the Dirichlet form.

Of course, the function $f$ proposed in the proof of the previous
lemma gives only the correct magnitude of the capacity $\Cap_N(A,B)$ and
not its exact asymptotic value. The computation of the exact
asymptotic value requires a detailed information of the jump rates and
has to be done model by model.

We may prove, however, that under certain assumptions the capacity
between two sets conveniently rescaled converges. Fix two disjoint
subsets of $E$: $A$, $B\subset E$, $A\cap B=\varnothing$. By
definition, $G_N(A,B) = G_N(x,y)$ for some $x$, $y\in E$. By
\eqref{27}, $G_N(x,y) =g_N(x,y) \, G_N(j)$ for some $0\le j \le \mf
j\, $. Let $\mf g_N(A,B) = G_N(j) \preceq 1$ so that $G_N(A,B)/ \mf
g_N(A,B)$ converges, as $N\uparrow\infty$, to some number in
$(0,\infty)$.

\begin{lemma}
\label{s21}
Fix two disjoint subsets of $E$: $A$, $B\subset E$, $A\cap
B=\varnothing$. Let $f_N : E\to [0,1]$ be the function $f_N(x) = \mb
P^N_x[T_A < T_B]$. Assume that $f_N$ converges pointwisely to some
function $f$. Denote by $\mf B(A,B)$ the set of pairs $\{x,y\}$ such
that $G_N(x,y) \approx \mf g_N(A,B)$.  Then, $f(y)=f(x)$ if $G_N(x,y)
\succ \mf g_N(A,B)$ and 
\begin{equation*}
\lim_{N\to\infty} \frac {\Cap_N(A,B)}{\mf g_N(A,B)}  \;=\; 
\sum_{\{x,y\} \in \mf B(A,B)} g(x,y) [f (y) - f (x)]^2 
\in (0,\infty)\;,
\end{equation*}
where $g(x,y)$ has been introduced in \eqref{27}.
\end{lemma}

\begin{proof}
Fix two disjoint subsets of $E$: $A$, $B\subset E$, $A\cap
B=\varnothing$ and let $f_N : E\to [0,1]$ be the function $f_N(x) =
\mb P^N_x[T_A < T_B]$. It is well known that
\begin{equation*}
\Cap_N(A,B) \;=\; D_N (f_N) \;=\;
\sum_{\{x,y\}\subset E}  G_N(x,y) [f_N (y) - f_N(x)]^2 \;. 
\end{equation*}

We first show that $f(y) = f(x)$ if $G_N(x,y) \succ \mf
g_N(A,B)$. Indeed, fix such a pair and note that
\begin{equation*}
G_N(x,y)\, [f_N (y) - f_N(x)]^2 \;\le\;  \Cap_N(A,B)\;.
\end{equation*}
By Lemma \ref{s03}, the right hand side is bounded above by $C_1 \mf
g_N(A,B)$ for some finite constant $C_1$ independent of $N$.  Since
$f_N$ converges to $f$ pointwisely and since $G_N(x,y) \succ \mf
g_N(A,B)$, $f(y)=f(x)$, proving the claim.

Let $\mf B = \mf B(A,B)$. To prove a lower bound for the capacity,
note that
\begin{equation*}
\Cap_N(A,B) \;\ge\;  
\sum_{\{x,y\} \in \mf B} G_N(x,y) \, [f_N (y) - f_N(x)]^2 \;. 
\end{equation*}
In view of \eqref{27}, as $N\uparrow\infty$ the right hand side
divided by $\mf g_N(A,B)$ converges to
\begin{equation*}
\sum_{\{x,y\}\in \mf B} g(x,y) [f (y) - f(x)]^2 \;. 
\end{equation*}

To prove the upper bound, recall the variational formula \eqref{25}
for the capacity to write
\begin{equation*}
\Cap_N(A,B) \;\le\; 
\sum_{\{x,y\}\subset E} G_N(x,y) [f (y) - f(x)]^2\;.
\end{equation*}
Since $f(y) = f(x)$ if $G_N(x,y) \succ \mf g_N(A,B)$, and since $f$
is absolutely bounded by $1$, we may restrict the sum to the pairs
$(x,y)$ in $\mf B(A,B)$. Hence,
\begin{equation*}
\limsup_{N\to\infty} \frac{\Cap_N(A,B)}{\mf g_N(A,B)} \;\le\; 
\sum_{\{x,y\} \in \mf B} g(x,y) [f (y) - f(x)]^2\;,
\end{equation*}
which proves the second assertion of the lemma. Moreover, by Lemma
\ref{s03} and since $G_N(A,B) \approx \mf g_N(A,B)$, the limit belongs
to $(0,\infty)$.
\end{proof}

This result shows that the sequence of capacities are comparable if
the sequence of hitting functions $f_{A,B}^N(x) = \mb P^N_x[T_A<T_B]$
converge. This remark highlights the interest of the next result.
Recall that we denote by $R^F_N(x,y)$, $x$, $y\in F$, the jump rates
of the trace process $\{\eta^F_t : t\ge 0\}$, $F\subset E$.

\begin{lemma}
\label{s29}
For every subset $F$ of $E$, the sequences $(R^F_N(x,y):N\ge 1)$,
$x\not = y\in F$, are comparable. Moreover, for every subsets $A$, $B$
of $E$, $A\cap B=\varnothing$, and for every $x\in E$, the following
limits exist
\begin{equation*} 
f_{A,B}(x)\;:=\; \lim_{N\to\infty} \mb P^{N}_x [T_{A} < T_{B}] \;.
\end{equation*}
\end{lemma}

\begin{proof}
It follows from the displayed formula presented just after Corollary
6.2 in \cite{bl2} that
\begin{equation*}
R^F_N(x,y) \;=\; \frac{ \sum_{z\in E} R_N (x,y) \, R_N(w,z) \;+\; R_N(x,w)
\,R_N(w,y)}{\sum_{z\in E} R_N(w,z)}
\end{equation*}
if $F = E\setminus \{w\}$. Iterating this formula, we may show that
for every proper subset $F$ of $E$, $R^F_N(x,y)$ may be expressed as a
ratio of sums of products of the rates $R_N(\,\cdot\,,\,\cdot\,)$.
The sum in the numerator contains only products with the same number
of terms and the same thing happens in the denominator. In particular,
by assumption \eqref{30}, the sequences $\{R^F_N(x,y) : N\ge 1\}$,
$x\not = y \in F$, are comparable. This proves the first assertion of
the lemma.

If we denote by $p^F_N(x,y)$ the jump probabilities associated to the
rates $R^F_N(x,y)$,
\begin{equation*}
p^F_N(x,y) \;=\; \frac{R^F_N(x,y)}{\sum_{z\in F} R^F_N(x,z)}\;, \quad 
x\not = y \in F\;,
\end{equation*}
$p^F_N(x,y)$ converges to some $p^F(x,y)$ as $N\uparrow\infty$.

Denote by $\mb P^{N,F}_x$, $x\in F$, the probability on the path space
$D(\bb R_+, F)$ induced by the trace process $\{\eta^F_t : t\ge 0\}$
starting from $x$. Clearly, $\mb P^{N}_x [T_{A} < T_{B}] = \mb
P^{N,F}_x [T_{A} < T_{B}]$, for $F=\{x\}\cup A\cup B$. If $x$ does not
belong to $A\cup B$, last probability is equal to $\sum_{y\in A}
p^F_N(x,y)$ and we proved that this expression converges as
$N\uparrow\infty$.
\end{proof}

\begin{corollary}
\label{s32}
For every subset $F$ of $E$ and every subsets $A$, $B$ of $F$, $A\cap
B=\varnothing$, the ratio of mean rates
\begin{equation*}
\frac{r^F_N(A,B)}{r^F_N(A,F\setminus A)} \;=\;
\frac{\sum_{x\in A} \sum_{y\in B} \mu^N(x) \, R^F_N (x,y)}
{\sum_{x\in A} \sum_{z\in F\setminus A} \mu^N(x) \, R^F_N (x,z)}
\end{equation*}
converges to some number $p_F(A,B)\in [0,1]$ as $N\uparrow\infty$.
\end{corollary}

\begin{proof}
It follows from the explicit formula for the rates $R^F_N$, derived in
the proof of the previous lemma, from equation \eqref{33} and from
assumption \eqref{30} that the sequences $(\mu^N(x) \, R^F_N(x,y):N\ge
1)$, $x\not = y\in F$, are comparable. The result is a simple
consequence of this observation.
\end{proof}

\section{The shallowest valleys}
\label{ssec5}

Recall the definition of a valley with an attractor introduced in
\cite{bl2}. To avoid long sentences, in this article we call a valley
with an attractor simply a valley. We describe in this section the
shallowest valleys and we show that their depths are comparable.

We shall say that there exists an \emph{open path} from $x$ to $y$ if
there exists a sequence $x=x_0, x_1, \dots, x_n=y$ such that $R_N(x_i,
x_{i+1})\approx 1$, $0\le i < n$. Two sites $x \not = y$ are said to
be equivalent, $x \sim y$, if there exist an open path from $x$ to $y$
and an open path from $y$ to $x$.  If we also declare any site to be
equivalent to itself, $\sim$ is an equivalent relation. We denote by
$\mc C_1, \mc C_2, \dots, \mc C_{\alpha}$ the equivalent classes.

Some equivalent classes are connected to other equivalent classes by
open paths. By drawing an arrow from a set $\mc C_i$ to a set $\mc
C_j$ if there exist $x\in\mc C_i$, $y\in \mc C_j$ such that $R_N(x,
y)\approx 1$, the set $\{\mc C_1, \dots, \mc C_{\alpha}\}$ becomes an
oriented graph with no directed loops. We denote by $\mc E_1, \mc E_2,
\dots, \mc E_{\nu}$ the leaves of this graph, in the terminology of
graph theory, the equivalent classes with no successors. Denote by
$\Delta$ the union of the remaining sets so that $\{\mc E_1, \dots,
\mc E_{\nu}, \Delta\}$ forms a partition of $E$:
\begin{equation}
\label{17}
E \;=\; \mc E \cup \Delta \;, \quad
\mc E \;=\; \mc E_1 \cup  \cdots \cup \mc E_{\nu}
\; .
\end{equation}
For $1\le i\le \nu$, let $\breve{\mc E}_i$ be the union of all leaves
except $\mc E_i$:
\begin{equation*}
\breve{\mc E}_i \;=\; \bigcup_{j\not = i} {\mc E}_j\;.
\end{equation*}

By construction, all sites in an equivalent class $\mc C_j$ have
probability of the same magnitude: there exists a finite, positive
constant $C_0$ such that for all $1\le j\le \alpha$,
\begin{equation}
\label{f01}
C_0^{-1} \;\le\; \frac{\mu_N(x)}{\mu_N(y)} \;\le\; C_0\;, \quad x\;, y
\in \mc C_j\;.
\end{equation}

We may also estimate the capacity between two states in a leave $\mc
E_i$. 

\begin{lemma}
\label{s01}
Fix $1\le i\le \nu$.  There exists a finite constant $C_1$, which
depends only on $E$, such that for any $x\not = y$ in $\mc E_i$,
\begin{equation*}
C_1^{-1}  \;\le\; \frac{\Cap_N(\{x\}, \{y\})}{\mu_N(\mc E_i)}
\;\le\; C_1 \;.
\end{equation*}
\end{lemma}

\begin{proof}
Fix $1\le i\le \nu$ and $x\not = y$ in $\mc E_i$.  Consider a function
$f:E\to \bb R$ such that $f(x)=1$, $f(y)=0$ and fix a self-avoiding
open path $\gamma =(x=x_0, \dots, x_n=y)$ from $x$ to $y$. By Schwarz
inequality,
\begin{equation*}
\begin{split}
& 1 \;=\; [f(y) - f(x)]^2 \;=\; \Big\{ \sum_{i=0}^{n-1}
\{f(x_{i+1})-f(x_i)\} \Big\}^2 \\
& \quad \le\; 
\sum_{i=0}^{n-1} \mu_N(x_i) R_N(x_i, x_{i+1}) \{f(x_{i+1})-f(x_i)\}^2 
\, \sum_{i=0}^{n-1} \frac 1{\mu_N(x_i) R_N(x_i, x_{i+1})}\; \cdot
\end{split}
\end{equation*}
Therefore, $D_N(f)$, is bounded below by 
\begin{equation*}
\sum_{i=0}^{n-1} \mu_N(x_i) R_N(x_i, x_{i+1})
\{f(x_{i+1})-f(x_i)\}^2 \;\ge\; \Big\{ \sum_{i=0}^{n-1} 
\frac 1{\mu_N(x_i) R_N(x_i, x_{i+1})} \Big\}^{-1}\;.
\end{equation*}
Since $\gamma$ is an open path, $R_N(x_i,x_{i+1})$ is of order
one. Hence, by \eqref{f01}, there exists a constant $C_1$ which
depends only on $E$ such that for any function $f:E\to \bb R$ such
that $f(x)=1$, $f(y)=0$, $D_N(f) \ge C^{-1}_1 \mu_N(\mc E_i)$.  This
proves that $\Cap_N(\{x\}, \{y\})\ge C^{-1}_1 \mu_N(\mc E_i)$.

To prove the reverse inequality, consider the function $f_*:E\to \bb
R$ which is equal to $1$ at $x$ and is $0$ elsewhere. Clearly,
\begin{equation*}
D_N(f_*) \;=\;  \mu_N(x) \sum_{z\not =x} R_N(x,z)\;.  
\end{equation*}
By hypothesis, $R_N(x,z) \le C_0$ so that $D_N(f_*) \le C'_0 \mu_N(\mc
E_i)$, proving the lemma.
\end{proof}

Recall Theorem 2.6 of \cite{bl2} which presents sufficient conditions
for a triple to be a valley in the context of reversible Markov
processes.

Fix a leave $\mc E_i$, $1\le i\le \nu$, and a site $x$ in $\mc E_i$.
Denote by $\mc E_i$ the set $\mc E_i$ as well as the constant sequence
of sets $(\mc E_i, \mc E_i , \dots)$ and by $x$ not only the site $x$
but also the constant sequence equal to $x$. This convention is used
from now on without further notice. Denote by $\mc B_i$ the set of
sites in $\Delta$ of measure of lower magnitude than $\mc E_i$: $\mc
B_i = \{y\in \Delta : \mu_N(y) \prec \mu_N(\mc E_i)\}$. Note that $\mc
B_i$ is the union of some equivalence classes.

\begin{lemma}
\label{s02}
Fix $1\le i\le \nu$ and $x$ in $\mc E_i$.  The triple $(\mc E_i, \mc
E_i \cup \mc B_i, x)$ is a valley of depth $\theta_{N,i} := \mu_N(\mc
E_i)/\Cap_N(\mc E_i , [\mc E_i \cup \mc B_i]^c)$.
\end{lemma}

\begin{proof}
By \cite[Theorem 2.6]{bl2}, to show that $(\mc E_i, \mc B_i \cup \mc
E_i, x)$ is a valley of depth $\mu_N(\mc E_i)/$ $\Cap_N(\mc E_i , [\mc
E_i \cup \mc B_i]^c)$ we need to check that $\mu_N(\mc B_i)/\mu_N(\mc
E_i)$ vanishes as $N\uparrow\infty$ and that
\begin{equation}
\label{01}
\lim_{N\to\infty} \frac{\Cap_N(\mc E_i, [\mc B_i\cup \mc E_i]^c)}
{\Cap_N(x)}\,=\,0\, ,
\end{equation}
where $\Cap_N(x) = \min_{y\in \mc E_i} \Cap_N(\{x\}, \{y\})$.  The
first condition follows from the definition of the set $\mc B_i$.  The
second one is simple to check. Fix a positive function $f: E \to \bb
R$ bounded by one and constant in $\mc E_i$ and in $[\mc B_i\cup \mc
E_i]^c$. In the expression of the Dirichlet form $D_N(f) =
\sum_{y,z} G_N(y,z) [f(z)-f(y)]^2$, there are two types of
non-vanishing terms.  Either $y$ belongs to $\mc E_i$ and we may
estimate $G_N(y,z) [f(z)-f(y)]^2$ by $\mu_N(\mc E_i) \max_{y\in \mc
  E_i, z\not \in \mc E_i} R_N(y,z)$ or $y$ does not belong to $\mc
E_i$ and we may estimate $G_N(y,z) [f(z)-f(y)]^2$ by $\mu_N(\mc B_i)$
because $R_N(y,z) \le C_0$. Both expressions are of an order much
smaller than the one of $\mu_N(\mc E_i)$ because $\mc E_i$ has no
successors.  Therefore, \eqref{01} follows from Lemma \ref{s01},
proving that $(\mc E_i, \mc B_i \cup \mc E_i, x)$ is a valley.
\end{proof}

Next lemma shows that a leave is attained from any site in a time
scale of magnitude one. Recall that $T_A$, $A\subset E$, stands for
the hitting time of $A$. 

\begin{lemma}
\label{s05}
There exists a finite constant $C_0$, independent of $N$, such that  
\begin{equation*}
\max_{x \in \Delta} \mb E^N_x \big[ T_{\mc E} \big]
\;\le\; C_0\;.
\end{equation*}
\end{lemma}

\begin{proof}
Denote by $\{\tau_j : j\ge 0\}$ the jump times of the Markov process
$\{\eta^N_t : t\ge 0\}$:
\begin{equation*}
\tau_0 \;=\; 0\;, \quad \tau_{j+1} \;=\; 
\inf\{t>\tau_j : \eta^N_t \not = \eta^N_{\tau_j}\}\;, \quad j\ge 0\;.
\end{equation*}
Denote by $\{Y^N_k : k\ge 0\}$ the jump chain associated to the Markov
process $\{\eta^N_t : t\ge 0\}$, i.e., the discrete time Markov chain
formed by the successive sites visited by $\eta^N_t$:
\begin{equation*}
Y^N_k \;=\; \eta^N_{\tau_k}\;, \quad k\ge 0\;.
\end{equation*}

For each site $x$ in $\Delta$, there exists an open path $\gamma
=(x=x_0, x_1, \dots, x_{n(x)})$ such that $x_{n(x)}\in\mc E$,
$R_N(x_i, x_{i+1}) > C_0$, $0\le i < n(x)$, for some positive constant
$C_0$, independent of $N$, whose value may change from line to
line. In particular,
\begin{equation}
\label{14}
\mb P^N_x \big[ Y^N_k = x_k : 0\le k\le n(x) \big] \;\ge\; C_0\;.
\end{equation}
Let $n=\max \{ n(x) : x\in\Delta\}$.

By the strong Markov property and decomposing the space according to
the partition $\{T_{\mc E} \le \tau_n\}$, $\{T_{\mc E} > \tau_n\}$,
for every $x\in \Delta$, since on the set $\{T_{\mc E} > \tau_n\}$,
$T_{\mc E} = \tau_n + T_{\mc E} \circ \tau_n$,
\begin{equation*}
\mb E^N_x \big[ T_{\mc E} \big] \;=\; \mb E^N_x \big[ \min\{ T_{\mc
  E}, \tau_n \}  \big] \;+\; \mb E^N_x \big[ \mb 1 \{T_{\mc E} >
\tau_n\} \mb E^N_{\eta^N_{\tau_n}} \big[ T_{\mc E}  \big] \, \big] \;.
\end{equation*}
As $\eta^N_{\tau_n}$ belongs to $\Delta$ when $T_{\mc E} > \tau_n$, it
follows from the previous identity that
\begin{equation*}
\max_{x\in \Delta} \mb E^N_x \big[ T_{\mc E} \big]
\;\le\; \frac{\max_{x\in \Delta} \mb E^N_x \big[ 
\min\{ T_{\mc E}, \tau_n \} \big]}{1 - \max_{x\in \Delta} 
\mb P^N_x \big[ T_{\mc E} > \tau_n\} \big]} \;\cdot
\end{equation*}
It follows from \eqref{14} that the denominator is bounded below by a
strictly positive constant $C_0$. To estimate the numerator, observe
that $T_{\mc E} = \sum_{j\ge 1} \tau_j \mb 1\{A_j\}$, where $A_j = \{
Y^N_0\in \Delta, \dots, Y^N_{j-1} \in \Delta, Y^N_j \in \mc E\}$. Hence,
$\min\{ T_{\mc E}, \tau_n \} = \sum_{1\le j< n} \tau_j \mb 1\{A_j\} +
\tau_n \mb 1\{B_n\}$, where $B_n = \cup_{j\ge n}A_j$. Since on the set
$A_j$, $Y^N_k \in \Delta$, $0\le k<j$, on $A_j$ the random time $\tau_j$
can be estimated by the sum of $j$ mean $C_0$ independent exponential
random variables. Hence,
\begin{equation*}
\max_{x\in \Delta} \mb E^N_x \big[ \min\{ T_{\mc E}, \tau_n \} \big]
\;\le\; C_0 \sum_{j=1}^n j \;,
\end{equation*}
which concludes the proof of the lemma.
\end{proof}

A similar argument permits to increase the negligible set $\mc B_i$ of
the valley $(\mc E_i, \mc E_i \cup \mc B_i, x)$.

\begin{lemma}
\label{s06}
Fix $1\le i\le \nu$ and $x$ in $\mc E_i$.  The triple $(\mc E_i, \mc
E_i \cup \Delta, x)$ is a valley of depth $\theta_{N,i} = \mu_N(\mc
E_i)/\Cap_N(\mc E_i , [\mc E_i \cup \mc B_i]^c)$. Moreover,
$\Cap_N(\mc E_i , [\mc E_i \cup \mc B_i]^c) \approx \Cap_N(\mc E_i ,
\breve{\mc E}_i)$.
\end{lemma}

\begin{proof}
Fix $1\le i\le \nu$ and $x$ in $\mc E_i$.  By Lemmas \ref{s02} and
\ref{s04}, to prove the first claim of the proposition we need to show
that for every $\delta>0$
\begin{equation*}
\lim_{N\to \infty} \max_{y\in \Delta \setminus \mc B_i}
\mb P_y^N \big[ T_{\breve{\mc E}_i} > \delta \theta_{N,i} \big]\; =\;
0\; .
\end{equation*}

Fix $y$ in $\Delta \setminus \mc B_i$. By definition, $\mu_N(y)
\succeq \mu_N(\mc E_i)$. In particular, there is no open path from $y$
to $\mc E_i$. Indeed, if $y=y_0, \dots, y_{n-1}\not\in\mc E_i \,,\,
y_n\in\mc E_i$ is an open path from $y$ to $\mc E_i$, the relation
$\mu_N(y_{n-1}) \succeq \mu_N(y) \succeq \mu_N(\mc E_i)$ contradicts
the identity $\mu_N(y_{n-1}) R_N(y_{n-1}, y_n) = \mu_N(y_{n}) R_N(y_n,
y_{n-1})$ because $\max_{z\in\mc E_i, z'\not\in\mc E_i} R_N(z,z') \prec
1$.

Recall that we denote by $\{Y^N_k : k\ge 0\}$ the jump
chain associated to the Markov process $\eta^N_t$. Its jump
probabilities $p_N(x,y)$, $x,y\in E$, $x\not =y$, are given by
\begin{equation*}
p_N(x,y) \;=\; \frac{R_N(x,y)}{\sum_{z\in E} R_N(x,z)} \; \cdot
\end{equation*} 
In view of \eqref{24}, as $N\uparrow\infty$, $p_N(x,y)$ converges to
some $p(x,y) \in [0,1]$ such that $\sum_y p(x,y)=1$.  Let $\{Z_k :
k\ge 0\}$ be the discrete time Markov chain associated to the jump
probabilities $p(x,y)$. Note that the Markov chain $Z_k$ may not be
irreducible .

Clearly, we may couple both chains in a way that for every $n\ge 1$
\begin{equation}
\label{28}
\lim_{N\to\infty}
\mb P_y^N \Big[ \bigcup_{k=1}^n \{Y^N_k \not = Z_k\} \Big] \;=\; 0\; .
\end{equation}

On the other hand, before reaching $\mc E$ the Markov chain $Z_k$ only
uses open bonds. Since there is no open path from $y$ to $\mc E_i$ and
since there are open paths from $y$ to $\breve{\mc E}_i$ , the chain
$Z_k$ eventually reaches $\breve{\mc E}_i$.  Hence,
\begin{equation*}
\lim_{n\to\infty}
\mb P_y^N \Big[ \bigcap_{k=1}^n \{Z_k \not \in \breve{\mc E}_i\} \Big] 
\;=\; 0\; .
\end{equation*}

Recall that $\{\tau_n : n\ge 1\}$ stands for the jump times of the
Markov process $\eta^N_t$. On the set $[\cap_{k=1}^n \{Y^N_k =
Z_k\}] \cap [ \cup_{k=1}^n \{Z_k \in \breve{\mc E}_i\}]$,
$T_{\breve{\mc E}_i} \le \tau_k$ for some $k\le n$, and $\tau_k$ may
be bounded by the sum of $k$ mean $C_0$ i.i.d.\!  exponential random
variables, for some finite constant $C_0$, independent of $N$.
Therefore, since $\theta_{N,i}\succ 1$, for every $n\ge 1$,
\begin{equation*}
\lim_{N\to\infty}
\mb P_y^N \Big[ T_{\breve{\mc E}_i} > \delta \theta_{N,i} \,,\,
\bigcap_{k=1}^n \{Y^N_k = Z_k\} \,,\, \bigcup_{k=1}^n 
\{Z_k \in \breve{\mc E}_i\} \Big] \;=\; 0\; ,
\end{equation*}
which proves the first assertion of the lemma.

In view of Lemma \ref{s03}, to prove the second claim, it is enough to
show that $G_N(\mc E_i , [\mc E_i \cup \mc B_i]^c) \approx G_N(\mc E_i
, \breve{\mc E}_i)$. In fact, we assert that
\begin{equation}
\label{07}
G_N(\mc E_i , [\mc E_i \cup \mc B_i]^c) \;=\;
 G_N(\mc E_i , \breve{\mc E}_i)
\end{equation}
for all $N$ sufficiently large.

On the one hand, since $\breve{\mc E}_i \subset [\mc E_i \cup \mc
B_i]^c$, $G_N(\mc E_i , \breve{\mc E}_i) \le G_N(\mc E_i , [\mc E_i
\cup \mc B_i]^c)$. On the other hand, since the set $E$ is finite,
there exists a path $\gamma = (x_0, \dots, x_n)$ in $\Gamma_{\mc E_i ,
  [\mc E_i \cup \mc B_i]^c}$ such that $G_N(\mc E_i , [\mc E_i \cup
\mc B_i]^c) = G_N(\gamma)$. By definition $x_0\in \mc E_i$, $x_n
\not\in \mc E_i \cup \mc B_i$, and we may assume without loss of
generality that $x_1\not\in \mc E_i$.

Since $x_1\not\in \mc E_i$ and $\mc E_i$ is a leave, $G_N(x_0,x_1) \prec
\mu_N(\mc E_i)$ so that 
\begin{equation}
\label{06}
G_N(\gamma) \;=\;
\min_{0\le i <n} G_N(x_i, x_{i+1}) \; \prec \; \mu_N(\mc E_i)\;.
\end{equation}

Either $x_n$ belongs to $\Delta\setminus \mc B_i$ or $x_n$ belongs to
$\breve{\mc E_i}$. In the latter case, $\gamma$ is a path in
$\Gamma_{\mc E_i , \breve{\mc E_i}}$ so that $G_N(\gamma) \le G_N(\mc
E_i , \breve{\mc E_i})$ proving that $G_N(\mc E_i , [\mc E_i \cup \mc
B_i]^c) \le G_N(\mc E_i , \breve{\mc E_i})$.

If $x_n$ belongs to $\Delta\setminus \mc B_i$, by definition of $\mc
B_i$ and the leaves $\mc E_j$, there exists a self-avoiding path
$\tilde \gamma = (x_n , x_{n+1}, \dots, x_l)$ from $x_n$ to $\mc E$
such that $R_N(x_i, x_{i+1}) \ge C_0$, $n\le i < l$, for some finite
constant $C_0$, independent of $N$, whose value may change from line
to line. Since $x_n \in \Delta\setminus \mc B_i$, $\mu_N(x_n) \ge C_0
\mu_N(\mc E_i)$ and the same estimate holds for $\mu_N(x_j)$, $n<j\le
l$, because $R_N(x_i, x_{i+1}) \ge C_0$, $n\le i < l$, and $R_N(y,z)
\le C_0$ for all $y$, $z\in E$. From these estimates we derive two
facts. First, $x_l$ may not belong to $\mc E_i$ because
$\mu_N(x_{l-1}) \ge C_0 \mu_N(\mc E_i)$, $R_N(x_{l-1}, x_l) \approx 1$
and $R_N(y,z) \prec 1$ for all $y\in \mc E_i$, $z\in \mc
E^c_i$. Second, $\min_{n\le i <l} G_N(x_i, x_{i+1}) \ge C_0 \mu_N(\mc
E_i)$ because $R_N(x_i, x_{i+1}) \ge C_0$, $n\le i < l$.

Therefore, if $x_n$ belongs to $\Delta\setminus \mc B_i$, juxtaposing
the paths $\gamma$ and $\tilde \gamma$, in view of \eqref{06}, we
obtain a self-avoiding path from $\mc E_i$ to $\breve{\mc E_i}$ such
that $\min_{1\le i <l} G_N(x_i, x_{i+1}) = \min_{1\le i <n} G_N(x_i,
x_{i+1}) = G_N(\mc E_i , [\mc E_i \cup \mc B_i]^c)$ for $N$
sufficiently large. Hence, also in the case where $x_n$ belongs to
$\Delta\setminus \mc B_i$, $G_N(\mc E_i , \breve{\mc E_i}) \ge G_N(\mc
E_i , [\mc E_i \cup \mc B_i]^c)$, which proves \eqref{07}.
\end{proof}

We may in fact compute the asymptotic behavior of the depth
$\theta_{N,i}$ of the valley $(\mc E_i, \mc E_i \cup \Delta, x)$ with
the help of Lemma \ref{s21}.  Recall the definition of the Markov
chain $\{Z_k : k\ge 0\}$ introduced in the previous proposition.
Denote by $\mb P^Z_x$ the probability on the path space $D(\bb Z_+,
E)$ induced by the Markov chain $\{Z_k : k\ge 0\}$ starting from $x$.

\begin{lemma}
\label{s22}
Fix a subset $I$ of $\{1, \dots, \nu\}$ and let $J=\{1, \dots, \nu\}
\setminus I$, $f_N(x) = \mb P^N_x[T_{\mc E_I} < T_{\mc E_J}]$, where
$\mc E_I = \cup_{i\in I} \mc E_i$. We claim that 
\begin{equation*}
\lim_{N\to\infty} \mb P^N_x[T_{\mc E_I} < T_{\mc E_J}] \;=\;
f_{I,J}(x) \;:=\;
\mb P^Z_x[T_{\mc E_I} < T_{\mc E_J}]\;, \quad x\in E\;.
\end{equation*}
In particular,
\begin{equation*}
\lim_{N\to\infty} \frac {\Cap_N(\mc E_I, \mc E_J)}
{\mf g_N(\mc E_I,\mc E_J)}  \;=\; 
\sum_{(x,y) \in \mf B(\mc E_I, \mc E_J)} g(x,y) 
\, [f_{I,J} (y) - f_{I,J} (x)]^2 
\in (0,\infty)\;.
\end{equation*}
\end{lemma}

\begin{proof}
Clearly, for every $x\in E$,
\begin{equation*}
\lim_{n\to\infty} \mb P^Z_x \Big[ \bigcup_{k=1}^n \{Z_k \in \mc E\}
\Big] \;=\; 1\;.
\end{equation*}
It follows from this estimates and from \eqref{28} that 
for every $x \in \Delta$,
\begin{equation*}
\lim_{N\to\infty} \mb P_x^N \big[ T_{\mc E_I} < T_{\mc E_J} \big]
\;=\; \mb P_x^Z \big[ T_{\mc E_I} < T_{\mc E_J} \big]\;,
\end{equation*}
which proves the first assertion of the lemma. The second one follows
from the previous result and Lemma \ref{s21}.
\end{proof}

Note that this result is a particular case of Lemma \ref{s29}.  The
same argument provides the asymptotic value of the capacity between
$\mc E_i$ and $[\mc E_i \cup \mc B_i]^c$.

\begin{lemma}
\label{s23}
Fix $1\le i\le \nu$ and let $f_N:E\to [0,1]$ be given by $f_N(x) = \mb
P^N_x [T_{\mc E_i} < T_{[\mc E_i \cup \mc B_i]^c}]$. Then, $f_N$
converges pointwisely to $f_i (x) = \mb P^Z_x [T_{\mc E_i} < T_{[\mc E_i
  \cup \mc B_i]^c}]$. In particular,
\begin{equation*}
\lim_{N\to\infty} \frac {\Cap_N(\mc E_i, [\mc E_i \cup \mc B_i]^c)}
{\mf g_N(\mc E_i, [\mc E_i \cup \mc B_i]^c)}  \;=\; 
\sum_{(x,y) \in \mf B(\mc E_i, [\mc E_i \cup \mc B_i]^c)} g(x,y) 
\, [f_i (y) - f_i (x)]^2 \;.
\end{equation*}
\end{lemma}

We may now state the first main result of this section.

\begin{proposition}
\label{s30}
For $1\le i\le \nu$, $(\mc E_i, \mc E_i \cup \Delta, x)$, $x\in \mc
E_i$, is a valley of depth $\theta_{N,i} = \mu_N(\mc E_i)/\Cap_N(\mc
E_i, \breve{\mc E_i})$. Moreover,
\begin{equation*}
\lim_{N\to\infty} \frac {\Cap_N(\mc E_i, [\mc E_i \cup \mc B_i]^c)}
{\Cap_N(\mc E_i, \breve{\mc E_i})}  \;=\; 1\;.
\end{equation*}
\end{proposition}

\begin{proof}
Fix $1\le i\le \nu$ and note that the set $\mc E_J$ appearing in Lemma
\ref{s22} is equal to $\breve{\mc E_i}$ if $I=\{i\}$.  By \eqref{07},
$\mf g_N(\mc E_i, [\mc E_i \cup \mc B_i]^c) = \mf g_N(\mc E_i,
\breve{\mc E}_i)$ so that $\mf B(\mc E_i, [\mc E_i \cup \mc B_i]^c) =
\mf B(\mc E_i, \breve{\mc E}_i)$. Let $g_i = f_{I,J}$ when $I=\{i\}$.
Since there is an open path from any state in $\Delta \setminus \mc
B_i$ to $\breve{\mc E}_i$ and no open path from a state in $\Delta
\setminus \mc B_i$ to $\mc E_i$, $f_i(x) = \mb P^Z_x [T_{\mc E_i} <
T_{[\mc E_i \cup \mc B_i]^c}] = \mb P^Z_x [T_{\mc E_i} < T_{\breve{\mc
    E}_i}] = g_i(x)$. This proves the corollary in view of Lemmas
\ref{s22} and \ref{s23}.
\end{proof}

By Proposition \ref{s30} and Lemma \ref{s22},
\begin{equation}
\label{31}
u_i\;:=\; \lim_{N\to\infty} \frac {\mf g_N(\mc E_i, \breve{\mc E}_i)}
{\mu_N(\mc E_i)} \, \theta_{N,i} \; \in\; (0,\infty)\;.
\end{equation}  
In particular, the depths of the valleys are comparable.

\begin{proposition}
\label{s24}
The sequences $(\theta_{N,i}:N\ge 1)$, $1\le i\le\nu$, are comparable
and $\theta_{N,i} \succ 1$, $1\le i\le\nu$.
\end{proposition}

\begin{proof}
To prove this lemma, we have to show that, as $N\uparrow\infty$, the
sequences $\theta_{N,i}/\theta_{N,j}$, $i\not = j$, either vanish,
diverge, or converge. Fix $i\not = j$. By \eqref{31},
\begin{equation*}
\lim_{N\to\infty} \frac{\theta_{N,i}}{\theta_{N,j}} \;=\; 
\frac{u_i}{u_j} \lim_{N\to\infty} 
\frac {\mu_N(\mc E_i)}{\mu_N(\mc E_j)} \,
\frac {\mf g_N(\mc E_j, \breve{\mc E}_j)}
{\mf g_N(\mc E_i, \breve{\mc E}_i)}\;\cdot
\end{equation*}
By \eqref{f01} and \eqref{32}, $\mu_N(\mc E_k) = \mu_N(x_k) a_N$ for
some $x_k\in \mc E_k$ and some sequence $a_N$ which converges to some
$a\in (0,\infty)$ as $N\uparrow\infty$. On the other hand, by
definition, $\mf g_N(\mc E_k, \breve{\mc E}_k) = G_N(y_k,z_k) b_N =
\mu_N(y_k) R_N(y_k,z_k)b_N$ for some bond $(y_k,z_k)$, where $b_N$
converges to some limit $b\in (0,\infty)$ as $N\uparrow\infty$. Hence,
\begin{equation*}
\frac{\theta_{N,i}}{\theta_{N,j}}\;=\;
c_N \, \frac{\mu_N(x_i) \, \mu_N(y_j) \, R_N(y_j,z_j) }
{\mu_N(x_j) \,\mu_N(y_i) \,R_N(y_i,z_i)}
\end{equation*}
for some sequence $c_N$ which converges to some limit $c\in
(0,\infty)$ as $N\uparrow\infty$. In view of identity \eqref{33} and
assumption \eqref{30}, the sequences $\theta_{N,i}$ are
comparable. This proves the first assertion of the lemma.

Fix $1\le i\le \nu$ and recall the definition of $\theta_{N,i}$ given
in Proposition \ref{s30}. By Lemma \ref{s03}, it is enough to show
that $\mu_N(\mc E_i)/G_N(\mc E_i, \breve{\mc E}_i)\succ 1$. Fix a
self-avoiding path $\gamma =(x_0, \dots, x_n)$ from $\mc E_i$ to
$\breve{\mc E}_i$ such that $G_N(\gamma) = G_N(\mc E_i, \breve{\mc
  E}_i)$. There exists a bond $(x_j,x_{j+1})$ such that $x_j\in\mc
E_i$, $x_{j+1}\not\in\mc E_i$. By definition of $G_N(\gamma)$, by
\eqref{f01} and since $\mc E_i$ is a leave, $G_N(\gamma)\le
G_N(x_j,x_{j+1}) \prec \mu_N(x_j) \approx \mu_N(\mc E_i)$, proving the
second assertion of the lemma.
\end{proof}

\section{Metastability among the shallowest valleys}
\label{ssec2}

We describe in this section the asymptotic behavior of the Markov
process $\{\eta^N_t : t\ge 0\}$ on the smallest time scale needed for
the process to jump from one leave to another.

Let $\theta_N(1) = \min \{\theta_{N,i} : 1\le i\le \nu\}$ and denote by
$S_1$ the indices of the shallowest leaves, i.e., the ones whose
valleys have depth of magnitude $\theta_N(1)$:
\begin{equation*}
S_1\;=\; \big\{ i : \theta_{N,i} \approx \theta_N (1) \big\}\; .
\end{equation*}

Since, by Proposition \ref{s24}, the depths of the valleys are
comparable and since $\theta_N(1)$ is the depth of the shallowest
valley, $\theta_N(1)/\theta_{N,i}$ converges as $N\uparrow\infty$:
\begin{equation}
\label{08}
\lambda (i) \;:=\; \lim_{N\to\infty} \frac{\theta_N(1)}{\theta_{N,i}} 
\;\in\; (0,\infty)\;.  
\end{equation}

\begin{lemma}
\label{s25}
For any $1\le i\not = j\le \nu$, $\theta_N(1) \, r^{\mc E}_N(\mc E_i,
\mc E_j)$ converges, as $N\uparrow\infty$, to some number $r(i,j)\in
[0,\infty)$.
\end{lemma}

\begin{proof}
Fix $1\le i\not = j\le \nu$.  By \cite[Lemma 6.7]{bl2} and by
Proposition \ref{s30}, we may rewrite $\theta_N(1) \, r^{\mc E}_N(\mc
E_i, \mc E_j)$ as
\begin{equation*}
\theta_N(1) \, r^{\mc E}_N(\mc E_i, \breve{\mc E}_i)  \,
\frac{r^{\mc E}_N(\mc E_i, \mc E_j)}
{r^{\mc E}_N(\mc E_i, \breve{\mc E}_i)} \;=\;
\theta_N(1) \, \frac{\Cap_N (\mc E_i, \breve{\mc E}_i)}
{\mu_N(\mc E_i)}\, \frac{r^{\mc E}_N(\mc E_i, \mc E_j)}
{r^{\mc E}_N(\mc E_i, \breve{\mc E}_i)} 
\;=\; \frac{\theta_N(1)} {\theta_{N,i}}\, 
\frac{r^{\mc E}_N(\mc E_i, \mc E_j)}
{r^{\mc E}_N(\mc E_i, \breve{\mc E}_i)}\;\cdot
\end{equation*}
By \eqref{08}, $\theta_N(1)/\theta_{N,i}$ converges to $\lambda(i)$.
On the other hand, by Corollary \ref{s32}, $r^{\mc E}_N(\mc E_i, \mc
E_j)/r^{\mc E}_N(\mc E_i, \breve{\mc E}_i)$ converges, as
$N\uparrow\infty$, to some number $q(i,j)\in [0,1]$. This proves the
lemma with $r(i,j) = \lambda(i)\, q(i,j)$.
\end{proof}

Let $\Psi: \mc E\to \{1,\dots, \nu\}$ be given by $\Psi(x) =
\sum_{1\le i\le\nu} i \, \mb 1\{x\in \mc E_i\}$.

\begin{lemma}
\label{s26}
Fix $1\le i\le \nu$ and $x\in \mc E_i$. Under $\mb P^N_x$, the speeded
up process $X^N_t = \Psi(\eta^{\mc E}_{t\theta_N(1)})$ converges to a
Markov process on $\{1, \dots, \nu\}$ with rates
$r(\,\cdot\,,\,\cdot\,)$ starting from $i$.
\end{lemma}

\begin{proof}
We need to check that the assumptions of \cite[Theorem 2.7]{bl2} are
fulfilled. On the one hand, condition ({\bf H1}) follows from Lemma
\ref{s01} and Proposition \ref{s24} which asserts that $\theta_{N,i}
\uparrow\infty$ as $N\uparrow\infty$. On the other hand, condition
({\bf H0}) has been proven in Lemma \ref{s25}.
\end{proof}

Note that $\lambda(j)=0$ if $j\not\in S_1$. The points in $S^c_1$ are
therefore absorbing for the asymptotic dynamics. 

Recall Definition 3.7 of \cite{bl2}.  The main result of this section,
stated below in Proposition \ref{s07}, asserts that the Markov process
$\{\eta^N_t : t\ge 0\}$ exhibits a metastable behavior on the time
scale $\theta_N(1)$ with asymptotic dynamics characterized by the
jumps rates $r(i,j)$ introduced in Lemma \ref{s25}. Denote by $\{\bb
P_i : 1\le i\le \nu\}$ the laws on the path space $D(\bb R_+, \{1,
\dots, \nu\})$ of a Markov process on $\{1, \dots, \nu\}$ whose sites
in $S_1^c$ are absorbing and which jumps from $i\in S_1$ to $j\not =
i$ at rate $r(i,j)$.

\begin{proposition}
\label{s07}
Fix a site $x_i$ on each leave $\mc E_i$. The sequence of Markov
process $\{\eta^N_t : t\ge 0\}$ exhibits a metastable behavior on the
time scale $\theta_N(1)$ with metastates $\{\mc E_i : 1\le i\le
\nu\}$, metapoints $\{x_i : 1\le i\le \nu\}$ and asymptotic Markov
dynamics $\{\bb P_i : 1\le i\le \nu\}$.
\end{proposition}

\begin{proof}
Condition ({\bf M2}) has been proven in Lemma \ref{s26}.

To prove ({\bf M3'}), observe that for every $x\in E$,
\begin{equation*}
\mb E^N_x \Big[ \int_0^{t} \mb 1\{ \eta^N_{s\theta_N(1)} \in \Delta \} \, ds
\Big] \;\le\; \max_{y\in \Delta} \, \mb E^N_y \Big[ \int_0^{t} 
\mb 1\{ \eta^N_{s\theta_N(1)} \in \Delta \} \, ds \Big] \;.
\end{equation*}

Fix $y\in \Delta$ and denote by $U_k$, $V_k$, $k\ge 1$, the successive
lengths of the sojourns in $\Delta$ and $\Delta^c$:
\begin{equation*}
\begin{split}
& U_1 \;=\; \inf\{t>0 : \eta^N_t \not\in \Delta \}\;, \quad
V_1 \;=\; \inf\{t>0 : \eta^N_{t+U_1} \in \Delta \} \;, \\
& \quad
U_{k+1} \;=\; \inf\{t>0 : \eta^N_{t+V_k} \not\in \Delta\}  \;, \quad
V_{k+1} \;=\; \inf\{t>0 : \eta^N_{t+U_{k+1}} \in \Delta\}  \;.
\end{split}
\end{equation*}
Denote by $\{N_t : t\ge 0\}$ the counting process associated to the
sequence $\{V_k : k\ge 1\}$: $\{N_t = k\} = \{ V_1 + \dots + V_{k} \le
t < V_1 + \dots + V_{k+1}\}$, $k\ge 0$, and observe that 
\begin{equation*}
\int_0^{t \theta_N(1)} \mb 1\{ \eta^N_s \in \Delta \} \, ds \;\le\;
U_1 \;+\; \sum_{k=1}^{N_{t \theta_N(1)}} U_{k+1}\;. 
\end{equation*}

Let $\lambda_N = \max_{1\le j\le \nu} \max_{y\in \mc E_j} \sum_{z\not
  \in \mc E_j} R_N(y,z) \prec 1$.  We may estimate from below the
random variables $\{V_k :k\ge 1\}$ by independent exponential times of
rate $\lambda_N$: $V_k \ge \hat V_k$, where $\{\hat V_k : k\ge 1\}$ is
a sequence of i.i.d. mean $\lambda_N^{-1}$ exponential random
variables, independent also from the sequence $\{U_k :k\ge 1\}$.

Let $\{\hat N_t : t\ge 0\}$ be the Poisson process associated to the
sequence $\{\hat V_k : k\ge 1\}$. Since the sequence $\{\hat V_k :
k\ge 1\}$ is independent of the sequence $\{U_k :k\ge 1\}$, in view of
the previous estimate,
\begin{equation*}
\begin{split}
& \mb E^N_y \Big[ \int_0^{t\theta_N(1)} \mb 1\{ \eta^N_{s} \in \Delta \} \, 
ds \Big] \;\le\; 
\mb E^N_y \Big[U_1 + \sum_{k=1}^{\hat N_{t \theta_N(1)}} U_{k+1} \Big] \\
& \quad  \le\; \mb E^N_y \big[ U_1 \big] \;+\; \sum_{\ell \ge 0}
\mb E^N_y \Big[ \sum_{k=1}^{\ell} U_{k+1} \Big] 
P^N_y \Big[\hat N_{t \theta_N(1)} = \ell \big] \;.  
\end{split}
\end{equation*}
By Lemma \ref{s05} this expression is bounded by $C_0 \{ 1 + t\,
\theta_N(1) \, \lambda_N\}$, which proves condition ({\bf M3'}).

The proof of condition ({\bf M1}') is similar to the one of Lemma
\ref{s05}. However, we may not estimate the expectation of $T_{x_i}$
which might be very large if the process leaves the metastable set
$\mc E_i$ before reaching the state $x_i$. We may of course assume
that $\mc E_i$ is not a singleton so that $\sum_{z\in E} R_N(y,z)$ is
of magnitude one for all $y\in\mc E_i$.

Recall that we denote by $\{\tau_k : k\ge 0\}$ the successive jump
times of $\eta^N_t$ and by $\{Y^N_k : k\ge 0\}$ the jump chain.  Fix
$1\le i\le \nu$, $x_i\in \mc E_i$ and let now $\lambda_N = \max_{y\in
  \mc E_i, z\not \in \mc E_i} R_N(y,z) \prec 1$. For each $y\in \mc
E_i$, there exists an open path $\gamma = (y_0=y, \dots, y_{n(y)}=x)$
from $y$ to $x$ contained in $\mc E_i$. Let $n= \max \{n(y) : y\in \mc
E_i \,,\, y\not = x\}$. There exists a constant $a$, independent of
$N$, such that
\begin{equation*}
\max_{y\in \mc E_i}
\mb P^N_y \big[ Y^N_k \not = x_i \,, 0 \le k\le n \big] \;\le\; a\; <\;
1\; . 
\end{equation*}

On the one hand, for every $\ell\ge 1$, $y\in\mc E_i$,
\begin{equation*}
\mb P^N_y \big[ \tau_{\ell n} \le \min \{ T_{x_i}, T_{\mc E^c_i} \}
\big] \;\le\; 
\mb P^N_y \big[ Y^N_k \not = x_i \,,\, Y^N_k \in \mc E_i \,,\,
0 \le k\le \ell n \big]\;.
\end{equation*}
By the Markov property and by the previous estimate, this expression
is bounded by $a^\ell$.  On the other hand, since the process jumps
from $\mc E_i$ to $\mc E^c_i$ at rate $\lambda_N \prec 1$,
\begin{equation*}
\mb P^N_y \big[ \tau_{n\ell} \ge  T_{\mc E^c_i} \big] \;\le\;
\mb P^N_y \Big[ \bigcup_{k=1}^{n\ell} Y^N_k \not \in \mc E_i \Big]
\;\le\; C_0\, \ell \,n\, \lambda_N 
\end{equation*}
for some finite constant $C_0$ independent of $N$.

In view of the previous bounds, to estimate $\mb P^N_y [ T_{x_i} >
\delta \theta_N ]$ it remains to consider the term
\begin{equation*}
\mb P^N_y \big[ T_{x_i} > \delta \theta_N \,,\,
T_{x_i} < \tau_{n\ell} < T_{\mc E^c_i} \big] \;\le\;
\mb P^N_y \big[ \delta \theta_N  < \tau_{n\ell} 
< T_{\mc E^c_i} \big] \;.
\end{equation*}
Since $\sum_{z\in E} R_N(y,z)$ is of magnitude one for all $y\in\mc E_i$,
before leaving the set $\mc E_i$, we may estimate the times between
jumps by i.i.d. exponential random times with finite mean independent
of $N$. By Tchebycheff inequality, the previous expression is thus
bounded by $C_0 n\ell /\delta \theta_N$ for some finite constant $C_0$
independent of $N$.

We have thus proved that for every $\delta>0$, $y\in\mc E_i$,
\begin{equation*}
\mb P^N_y \big[ T_{x_i} > \delta \theta_N \big] \;\le\;
a^\ell \;+\; C_0\, \ell \,n\, \lambda_N \;+\; \frac{C_0 n\ell}{\delta
  \theta_N}\; \cdot
\end{equation*}
The second assertion of the lemma follows by taking $N\uparrow\infty$
and then $\ell\uparrow\infty$.
\end{proof}

We conclude this section with two remarks. Denote by $P_N(x,i,j)$,
$1\le i\not = j\le \nu$, $x\in \mc E_i$, the hitting
probabilities
\begin{equation*}
P_N(x,i,j) \;:=\; 
\mb P^N_x \big[ T_{\mc E_j} = T_{\breve{\mc E}_i} \big]\;.
\end{equation*}
By Lemma \ref{s29}, $P_N(x,i,j)$ converges to some $P(x,i,j)\in
[0,1]$. Since, by Proposition \ref{s30}, $(\mc E_i, \mc E_i \cup \Delta,
y)$, $y\in \mc E_i$, is a valley, it is not difficult to show that the
limit $P(x,i,j)$ does not depend on the starting point $x$. Therefore,
by Lemma \ref{s39}, for any $1\le i\not = j\le \nu$,
\begin{equation}
\label{41a}
r(i,j) \;=\; \lambda(i) \,p(i,j)\;, \quad \text{where}
\quad p(i,j) \;:=\; \lim_{N\to\infty}
\mb P^N_x \big[ T_{\mc E_j} = T_{\breve{\mc E}_i} \big]
\end{equation}
and where $\lambda(i)$ is defined in \eqref{08}.

Consider a leave $\mc E_i$, $i\in S_1$, and a leave $\mc E_j$ such
that $\mu_N(\mc E_j) \prec \mu_N(\mc E_i)$. By reversibility,
\begin{equation*}
\mu_N(\mc E_i) \, r^{\mc E}_N(\mc E_i, \mc E_j) \;=\; \mu_N(\mc E_j) 
\, r^{\mc E}_N(\mc E_j, \mc E_i) \;.
\end{equation*}
By \cite[Lemma 6.7]{bl2}, $r^{\mc E}_N(\mc E_j, \mc E_i) \le r^{\mc
  E}_N(\mc E_j, \breve{\mc E}_j) = \Cap_N(\mc E_j, \breve{\mc
  E}_j)/\mu_N(\mc E_j) = 1/\theta_{N,j}$, so that $\theta_N(1) \,
r^{\mc E}_N(\mc E_j, \mc E_i)$ is bounded. Therefore, $\theta_N(1) \,
r^{\mc E}_N(\mc E_i, \mc E_j)$ vanishes as $N\uparrow\infty$. We have
just proved that
\begin{equation}
\label{10}
r(i,j)\;=\; \lim_{N\to\infty} \theta_N(1) \, r^{\mc E}_N(\mc E_i, \mc E_j) \;=\; 0
\quad \text{for all $j$; } \mu_N(\mc E_j) \prec \mu_N(\mc
E_i)\; .
\end{equation}
Hence, in the asymptotic dynamics, the process may only jump from a
leave $\mc E_i$ to a leave $\mc E_j$ if the measure of $\mc E_j$ is of
the same or of a larger magnitude than the one of $\mc E_i$.

\section{Multiscale analysis.}
\label{ssec3}

In the previous section, we proved that the Markov process $\{\eta^N_t
: t\ge 0\}$ exhibits a metastable behaviour on the time scale
$\theta_N(1)$ with metastates $\{\mc E_i : 1\le i\le \nu\}$,
metapoints $\{x_i : 1\le i\le \nu\}$ and asymptotic Markov dynamics
$\{\bb P_i : 1\le i\le \nu\}$.

We describe in this section, by a recursive argument, the metastable
behaviour of the Markov process $\{\eta^N_t : t\ge 0\}$ on longer time
scales. In the statement of the hypothesis {\bf T} below, by
convention, $\theta_N(0) \equiv 1$, $\nu(0) = |E|$ and the sets $\mc
E^{(0)}_1, \dots, \mc E^{(0)}_{\nu(0)}$ are all singletons of $E$.

\smallskip
\noindent {\bf Assumption {\bf T} at level $\mf p$:}\;\; For each
$1\le k\le \mf p$ there exists a sequence $(\theta_N(k) : N\ge 1)$,
$1\prec \theta_N(k) \prec \theta_N(k+1)$, $1\le k <\mf p$, and a
partition $\{\mc E^{(k)}_1 , \dots, \mc E^{(k)}_{\nu(k)}, \Delta_k\}$
of the state space $E$, such that

\renewcommand{\theenumi}{\arabic{enumi}}
\renewcommand{\labelenumi}{({\bf T\theenumi})}

\begin{enumerate}
\item $1\le \nu(k)<\nu(k-1)$.

\item For $1\le i\le \nu(k)$, $\mc E^{(k)}_i = \cup_{j\in I_{k,i}}
  \mc E^{(k-1)}_j$, where $I_{k,1}, \dots, I_{k,\nu(k)}$ are disjoint
  subsets of $\{1, \dots, \nu(k-1)\}$. 

\item For all $1\le i\le \nu(k)$, $\mu_N(x) \approx \mu_N(\mc E^{(k)}_i)$ for
  all $x\in \mc E^{(k)}_i$.

\item There exists a positive constant $C_1$, independent of $N$, such
  that for all $1\le i\le \nu(k)$ and all $x$, $y \in \mc E^{(k)}_i$,
  $x\not = y$, $\Cap_N (x,y) \ge C_1 \mu_N(\mc
  E^{(k)}_i)/\theta_N(k-1)$.

\item For all $1\le i\le \nu(k)$, $\mu_N(\mc E^{(k)}_i)/\Cap_N(\mc
  E^{(k)}_i , \breve{\mc E}^{(k)}_i) \succeq \theta_N (k)$, where
  $\breve{\mc E}^{(k)}_i = \cup_{j\not = i} \mc E^{(k)}_j$.

\item Let 
\begin{equation*}
\mc E^{(k)} \;=\; \bigcup_{i=1}^{\nu(k)} \mc E^{(k)}_i \;,
\quad S_k = \Big\{ i : \frac{\mu_N(\mc E^{(k)}_i)}{\Cap_N(\mc E^{(k)}_i ,
\breve{\mc E}^{(k)}_i)} \approx \theta_N (k) \Big\}\;.
\end{equation*}
Then,
\begin{equation}
\label{37}
\begin{split}
&\quad \lim_{N\to\infty} \theta_N(k) \, r^{\mc E^{(k)}}_N(\mc E^{(k)}_i, \mc
E^{(k)}_j) \;=\;  \mf r_k(i,j)\;, \quad 1\le i\not =j \le \nu(k)\;, \\
& {}\qquad \sum_{j\not = i} \mf r_k(i,j)>0 \text{ for each $i\in S_k$ and }
\sum_{j\not = i} \mf r_k(i,j)=0 \text{ for each $i\not\in S_k$,} \\
&\qquad\quad \text{ and}\quad  \mf r_k(i,j) \;=\; 0 
\quad \text{if}\quad \mu_N(\mc E^{(k)}_j) \;\prec\;
\mu_N(\mc E^{(k)}_i) \; .
\end{split}
\end{equation}
Moreover, recall the definition of the speeded up blind process
$X^{N,k}_t = \Psi_k(\eta^{N, k}_{t \theta_N (k)})$ introduced in the
statement of Theorem \ref{s19}. For every $1\le i\le \nu(k)$, and
$x\in \mc E^{(k)}_i$, under the measure $\mb P^N_x$,
\begin{equation}
\label{20}
\text{the speeded up blind process $X^{N,k}_{t}$ converges}
\end{equation}
to a Markov process on $\{1\, \dots, \nu(k)\}$ characterized by rates
$\mf r_k(l,m)$, $1\le l \not = m\le \nu(k)$, starting from
$i$. 

\item Property ({\bf M1}') of metastability holds. For every $1\le
  i\le \nu(k)$, every $x\in \mc E^{(k)}_i$ and $\delta>0$,
\begin{equation*}
\lim_{N\to \infty} \max_{y\in \mc E^{(k)}_i} \mb P^N_y \big[
T_{x} > \delta \theta_N(k) \big] \;=\; 0\;. 
\end{equation*}

\item Property ({\bf M3}') of metastability holds. For every $t>0$,
\begin{equation*}
\lim_{N\to \infty} \max_{x\in E} \, \mb E^N_x \Big[
\int_0^t \mb 1\{ \eta^N_{s\theta_N(k)} \in \Delta_k\} \, ds  \Big] \;=\; 0\;. 
\end{equation*}

\end{enumerate}
\smallskip

Note that all these properties have been proved in the previous
section for $\mf p=1$ with $\nu(1) = \nu$; $\mc E^{(1)}_1, \dots, \mc
E^{(1)}_{\nu(1)}, \Delta_1$ given by the sets $\mc E_1, \dots, \mc
E_{\nu}, \Delta$ defined just before \eqref{17}; $\theta_N(1) =
\theta_N = p_N(1)$ defined at the beginning of Section \ref{ssec2};
and $\mf r_1 = r$ defined at Lemma \ref{s25}. 

The main result of this section states that if Assumption {\bf T}
holds at level $\mf p$ and $\nu(\mf p)\ge 2$, then it holds at level
$\mf p +1$.

To begin the recursive argument, suppose that $\nu(\mf p)>1$.  We
first describe the metastates at level $\mf p + 1$. We say that there
exists an open path from $\mc E^{(\mf p)}_a$ to $\mc E^{(\mf p)}_b$ if
there exists a sequence $a=a_0 , a_1, \dots, a_n =b$ such that $\mf
r_{\mf p}(a_k, a_{k+1})>0$, where $\mf r_{\mf p}$ is the asymptotic
jump rate introduced in \eqref{20}.  We say that two sets $\mc E^{(\mf
  p)}_a$, $\mc E^{(\mf p)}_b$ are equivalent, $\mc E^{(\mf p)}_a \sim
\mc E^{(\mf p)}_b$, if there exist an open path from $\mc E^{(\mf
  p)}_a$ to $\mc E^{(\mf p)}_b$ and an open path from $\mc E^{(\mf
  p)}_b$ to $\mc E^{(\mf p)}_a$.

Two equivalent sets $\mc E^{(\mf p)}_a$, $\mc E^{(\mf p)}_b$ have
measure of the same magnitude. Indeed, if $\mc E^{(\mf p)}_a$, $\mc
E^{(\mf p)}_b$ are equivalent, there exists an open path $(a=a_0\, ,
\dots, a_n = b\, , \dots, a_{n+m}=a)$ from $\mc E^{(\mf p)}_a$ to $\mc
E^{(\mf p)}_a$ passing by $\mc E^{(\mf p)}_b$. By \eqref{37},
$\mu_N(\mc E^{(\mf p)}_{a_i}) \preceq \mu_N(\mc E^{(\mf
  p)}_{a_{i+1}})$, $0\le i <n+m$. Since $\mc E^{(\mf p)}_{a_0} = \mc
E^{(\mf p)}_{a_{n+m}} = \mc E^{(\mf p)}_a$, we obtain that
\begin{equation}
\label{13}
\mu_N(\mc E^{(\mf p)}_a) \;\approx\; \mu_N(\mc E^{(\mf p)}_b)\;,
\end{equation}
as claimed.

We call a metastate in the time scale $\theta_N(k)$ a $k$-metastate.
If we declare a $\mf p$-metastate equivalent to itself, the relation
$\sim$ introduced in the penultimate paragraph becomes an equivalent
relation among the $\mf p$-metastates $\mc E^{(\mf p)}_1, \dots, \mc
E^{(\mf p)}_{\nu (\mf p)}$. Denote by $\mc C^{(\mf p +1)}_1, \mc
C^{(\mf p +1)}_2, \dots, \mc C^{(\mf p +1)}_{\alpha(\mf p +1)}$ the
equivalent classes.  Some equivalent classes are connected to other
equivalent classes. By drawing an arrow from a set $\mc C^{(\mf p
  +1)}_i$ to a set $\mc C^{(\mf p +1)}_j$ if there exist $\mc E^{(\mf
  p)}_a \subset \mc C^{(\mf p +1)}_i$, $\mc E^{(\mf p)}_b \subset \mc
C^{(\mf p +1)}_j$ such that $\mf r_{\mf p}(a,b)>0$, the set $\{\mc
C^{(\mf p +1)}_1, \dots, \mc C^{(\mf p +1)}_{\alpha (\mf p +1)}\}$
becomes an oriented graph with no directed loops. We denote by $\mc
E^{(\mf p +1)}_1, \mc E^{(\mf p +1)}_2, \dots, \mc E^{(\mf p +1)}_{\nu
  (\mf p+1)}$ the leaves of this graph, i.e., the set of equivalent
classes with no successors in the terminology of graph theory, and by
$\Delta^o_{\mf p+1}$ the union of the remaining sets so that $\{\mc
E^{(\mf p +1)}_1, \dots, \mc E^{(\mf p +1)}_{\nu (\mf p+1)},
\Delta_{\mf p +1}\}$, $\Delta_{\mf p +1} = \Delta^o_{\mf p+1} \cup
\Delta_{\mf p}$, forms a partition of $E$:
\begin{equation*}
E \;=\; \mc E^{(\mf p +1)} \cup \Delta_{\mf p +1} \;, \quad
\mc E^{(\mf p +1)} \;=\; \mc E^{(\mf p +1)}_1 \cup  \cdots \cup 
\mc E^{(\mf p +1)}_{\nu(\mf p+1)} \; .
\end{equation*}
For $1\le i\le \nu (\mf p+1)$, let $\breve{\mc E}^{(\mf p +1)}_i$ be
the union of all leaves except $\mc E^{(\mf p +1)}_i$:
\begin{equation*}
\breve{\mc E}^{(\mf p +1)}_i \;=\; \bigcup_{j\not = i} \mc E^{(\mf p +1)}_j\;.
\end{equation*}

We may now state the main result of this section. 

\renewcommand{\theenumi}{\arabic{enumi}}
\renewcommand{\labelenumi}{({\bf \theenumi})}

\begin{theorem}
\label{s37}
Let $\{\eta^N_t : t\ge 0\}$ be a sequence of irreducible, reversible
Markov processes on a finite state space $E$ satisfying assumptions
\eqref{24} and \eqref{30}. Suppose that Assumption {\bf T} at level
$\mf p$ holds and that $\nu(\mf p)\ge 2$. Define $\nu(\mf p+1)$, $\mc
E^{(\mf p+1)}_i$, $\breve{\mc E}^{(\mf p+1)}_i$, $1\le i\le \nu(\mf
p+1)$, $\mc E^{(\mf p+1)}$, $\Delta^o_{\mf p +1}$, $\Delta_{\mf p+1}$
as above. Then,
\begin{enumerate}
\item For $1\le i\le \nu (\mf p+1)$ and $x$ in $\mc E^{(\mf p +1)}_i$,
  the triple $(\mc E^{(\mf p +1)}_i, \mc E^{(\mf p +1)}_i \cup
  \Delta^o_{\mf p+1} ,x)$ is a valley for the trace process
  $\{\eta^{N,\mf p}_t : t\ge 0\}$ of depth $\theta_{N,i} = \mu_N(\mc
  E^{(\mf p +1)}_i)/$ $\Cap_N (\mc E^{(\mf p +1)}_i , \breve{\mc
    E}^{(\mf p +1)}_i)$.

\item The sequences $(\theta_{N,i} : N\ge 1)$, $1\le i\le \nu(\mf
  p+1)$ are comparable.

\item Let $\theta_N(\mf p+1) = \min\{\theta_{N,i} : 1\le i\le \nu (\mf
  p+1)\}$. Then, $\theta_N(\mf p+1) \succ \theta_N(\mf p)$.

\item Assumption {\bf T} at level $\mf p+1$ holds.
\end{enumerate}
\end{theorem}

In the next remark, we summarize what informations are needed in each
model to prove all its metastable behavior. It says, in essence, that
to prove the metastable behavior of a particular dynamics, we need
only to obtain informations on the measure, on the capacity and on the
hitting times of subsets of the process.

\begin{remark}
\label{s38}
{\rm In the applications, once the metastable behavior in the time
scale $\theta_N(1)$ among the shallowest valleys has been determined,
we shall use Theorem \ref{s37} to describe the metastable behavior of
the process in the longer time scales. We first characterize the
$k$-metastates following the recipe presented above the statement of
Theorem \ref{s37}. According to this theorem, the $k$-metastates form
valleys of different depths. To determine the time scale at which
metastability at level $k$ can be observed we need to compute the
depth of each valley. This computation requires estimates on the
capacities among metastates and estimates on the measure of each
metastate. Once this has been done, we may define the time scale
$\theta_N(k)$. At this point, to complete the description of the
metastable behavior of the process at level $k$, it remains to obtain
the rates $\mf r_k(i,j)$. Theorem \ref{s37} asserts that the
asymptotic rates $\mf r_k(i,j)$ exist. By \eqref{41} the rates may be
expressed in terms of the asymptotic depths of the valleys and the
hitting probabilities of the metastates. Hence, to conclude we need to
compute in each model, the limit of the hitting probabilities defined
in \eqref{38}, which exist in virtue of Lemma \ref{s29}.

For some evolutions, as the Kawasaki dynamics, it may be difficult to
obtain an exact expression for the limit of the hitting
probabilities. Nevertheless, if we may at least determine if the rates
$\mf r_k(i,j)$ are positive or equal to $0$, we may apply Theorem
\ref{s37} and determine the time scales at which a metastable behavior
is observed and the metastates at each time scale, without an exact
knowledge of the asymptotic dynamics among the metastates. }
\end{remark}

The proof of Theorem \ref{s37} is divided in several lemmas.  We first
show that conditions ({\bf T1}) and ({\bf T2}) are satisfied for
$k=\mf p+1$.

\begin{lemma}
\label{s33}
We have that $\nu(\mf p+1)<\nu(\mf p)$ and that $\mc E^{(\mf p+1)}_i =
\cup_{a\in I_{\mf p+1,i}} \mc E^{(\mf p)}_a$, $1\le i\le \nu(\mf
p+1)$, where $I_{\mf p+1,1}, \dots, I_{\mf p+1,\nu(\mf p+1)}$ are
disjoint subsets of $\{1, \dots, \nu(\mf p)\}$.
\end{lemma}

\begin{proof}
A $\mf p$-metastate $\mc E^{(\mf p)}_a$, $a\in S_{\mf p}$, is either
contained in $\Delta^o_{\mf p+1}$ or part of a larger leave $\mc
E^{(\mf p +1)}_i$, in the sense that $\mc E^{(\mf p)}_a \subsetneq \mc
E^{(\mf p +1)}_i$, because by \eqref{37} each $\mf p$-metastate whose
index belongs to $S_{\mf p}$ has at least one successor. In
particular, the number of leaves at level $\mf p+1$ is strictly
smaller than the number of $\mf p$-metastates so that $\nu(\mf p+1) <
\nu(\mf p)$, proving condition ({\bf T1}). Condition ({\bf T2})
follows from the construction.
\end{proof}

Next lemma shows that conditions ({\bf T3}), ({\bf T4}) are in force
for $k=\mf p+1$.

\begin{lemma}
\label{s09}
For all $1\le i\le \nu(\mf p+1)$, $x\in \mc E^{(\mf p +1)}_i$,
$\mu_N(x) \approx \mu_N(\mc E^{(\mf p +1)}_i)$. Moreover, there exists
a positive constant $C_1$ such that for all $1\le i\le \nu(\mf p+1)$
and all $x$, $y \in \mc E^{(\mf p +1)}_i$, $x\not = y$, $\Cap_N (x,y)
\ge C_1 \mu_N(\mc E^{(\mf p +1)}_i)/\theta_N(\mf p)$.
\end{lemma}

\begin{proof}
Fix $1\le i\le \nu(\mf p+1)$, $x\in \mc E^{(\mf p +1)}_i$. By
definition, the leave $\mc E^{(\mf p +1)}_i$ is the union of $\mf
p$-metastates: $\mc E^{(\mf p +1)}_i = \cup_{a\in I} \mc E^{(\mf
  p)}_a$, where $I$ is a subset of $\{1, \dots, \nu(\mf p)\}$. 

By \eqref{13}, all $\mf p$-metastates $\mc E^{(\mf p)}_a$, $a\in I$,
have measures of the same magnitude so that $\mu_N(\mc E^{(\mf p
  +1)}_i) \approx \mu_N(\mc E^{(\mf p)}_a)$ for all $a\in I$. By
assumption ({\bf T3}) for $k=\mf p$, $\mu_N(y) \approx \mu_N(\mc
E^{(\mf p)}_a)$ for all $y\in \mc E^{(\mf p)}_a$, which proves the
first claim of the lemma.

To prove the second claim, fix $x$, $y$ in $\mc E^{(\mf p +1)}_i$. If
$x$, $y$ belong to the same set $\mc E^{(\mf p)}_a$, the lemma follows
from assumption ({\bf T4}) for $k=\mf p$, the first part of the lemma
and the fact that $\theta_N(\mf p-1) \prec \theta_N(\mf p)$.

Assume that $x$, $y$ belongs to different $\mf p$-metastates, say
$x\in \mc E^{(\mf p)}_a$, $y\in \mc E^{(\mf p)}_b$, $a\not = b$. Since
$\mc E^{(\mf p)}_a \sim \mc E^{(\mf p)}_b$, there exists an open path,
$a=a_0, \dots, a_n = b$, from $\mc E^{(\mf p)}_a$ to $\mc E^{(\mf
  p)}_b$. This means that $\theta_N(\mf p) \, r^{\mc E^{(\mf
    p)}}_N(\mc E^{(\mf p)}_{a_m}, \mc E^{(\mf p)}_{a_{m+1}})$
converges to a positive number for $0\le m <n$. Therefore, by
\eqref{18}, there exists a positive number $C_0>0$, independent of $N$
and which may change from line to line, such that
\begin{equation*}
\theta_N(\mf p) \, \frac{\Cap_N(\mc E^{(\mf p)}_{a_m},
\mc E^{(\mf p)}_{a_{m+1}})}{\mu_N(\mc E^{(\mf p)}_{a_m})} \;\ge\; C_0
\end{equation*}
for all $N$ large enough and $0\le m <n$.  Since by Lemma \ref{s03}
$\Cap_N (A,B) \approx G_N (A,B)$, $G_N(\mc E^{(\mf p)}_{a_m}, \mc
E^{(\mf p)}_{a_{m+1}}) \ge C_0 \, \mu_N(\mc E_{a_m})/ \theta_N(\mf
p)$. There exists, in particular, a path $\gamma_m$ from $x_m\in \mc
E^{(\mf p)}_{a_m}$ to $y_{m+1}\in \mc E^{(\mf p)}_{a_{m+1}}$, $0\le m
<n$, with $G_N(\gamma_m) \ge C_0\, \mu_N(\mc E^{(\mf p)}_{a_m})/
\theta_N(\mf p)$.

By assumption ({\bf T4}) for $k=\mf p$ and similar arguments to the
ones used above, there exist a path $\gamma'_0$ from $x\in \mc E^{(\mf
  p)}_a$ to $x_0\in \mc E^{(\mf p)}_a$ such that $G_N(\gamma'_0) \ge
C_0\, \mu_N(\mc E^{(\mf p)}_{a})/$ $\theta_N(\mf p-1)$; paths
$\gamma'_m$ from $y_{m}\in \mc E^{(\mf p)}_{a_{m}}$ to $x_{m}\in \mc
E^{(\mf p)}_{a_{m}}$, $1\le m<n$, such that $G_N(\gamma_m') \ge C_0
\mu_N(\mc E^{(\mf p)}_{a_m})/\theta_N(\mf p-1)$; and a path
$\gamma'_n$ from $y_n\in \mc E^{(\mf p)}_b$ to $y\in \mc E^{(\mf
  p)}_b$ such that $G_N(\gamma_n') \ge C_0\, \mu_N(\mc E^{(\mf
  p)}_b)/\theta_N(\mf p-1)$.

Since, by the first part of the lemma, $\mu_N(\mc E^{(\mf p +1)}_i)
\approx \mu_N(\mc E^{(\mf p)}_c)$ for all $c\in I$, juxtaposing all
these paths, we obtain a path $\gamma$ from $x$ to $y$ such that
$G_N(\gamma) \ge C_0\, \mu_N(\mc E^{(\mf p +1)}_i)/\theta_N(\mf
p)$. This shows that
\begin{equation*}
\Cap_N(x,y) \;\approx\; G_N(\{x\},\{y\}) \;\succeq\; G_N(\gamma)
\;\ge\; \frac{C_0\, \mu_N(\mc E^{(\mf p +1)}_i)}{\theta_N(\mf
  p)}\; ,
\end{equation*}
which proves the lemma.
\end{proof}

We next show that condition ({\bf T7}) is in force on any time scale
longer than $\theta_N(\mf p)$.

\begin{lemma}
\label{s17}
Let $\{\theta_N : N\ge 1\}$ be a sequence such that $\theta_N \succ
\theta_N (\mf p)$. Then for every $1\le i\le \nu(\mf p +1)$, $x\in \mc
E^{(\mf p+1)}_i$ and $\delta>0$,
\begin{equation*}
\lim_{N\to \infty} \max_{y\in \mc E^{(\mf p+1)}_i} \mb P^N_y \big[
T_{x} > \delta \, \theta_N \big] \;=\; 0\;. 
\end{equation*}
\end{lemma}

\begin{proof}
Fix $1\le i\le \nu(\mf p +1)$, $x, y\in \mc E^{(\mf p+1)}_i$ and
$\delta>0$. Denote by $\mc E^{(\mf p)}_a$, $\mc E^{(\mf p)}_b \subset
\mc E^{(\mf p+1)}_i$ the $\mf p$-metastates which contain $x$, $y$,
respectively. Since $\theta_N \succ \theta_N (\mf p)$, by the strong
Markov property, for every $t>0$ and for every $N$ large enough,
\begin{equation}
\label{21}
\mb P^N_y \big[ T_{x} > \delta \theta_N \big]
\;\le\; \mb P^N_y \big[ T_{\mc E^{(\mf p)}_a} > t \, \theta_N (\mf p) \big]
\;+\; \max_{z \in \mc E^{(\mf p)}_a} 
\mb P^N_z \big[ T_{x} > \delta \theta_N/2 \big]\; .
\end{equation} 

We claim that both expression vanishes as $N\uparrow\infty$ and then
$t\uparrow\infty$.  Denote by $T^{(\mf p)}_{\mc E^{(\mf p)}_a}$ the
hitting time of $\mc E^{(\mf p)}_a$ by the trace process $\eta^{N,\mf
  p}_t$ defined just before \eqref{20}. The first term on the right
hand side of the previous formula is bounded above by
\begin{equation*}
\mb P^N_y \Big[ \int_0^t \mb 1\{ \eta^N_{s \theta_N (\mf p)} 
\in \Delta_{\mf p} \} \, ds > \epsilon \Big] 
\;+\; \mb P^N_y \Big[ T^{(\mf p)}_{\mc E^{(\mf p)}_a} > (t-\epsilon)
\, \theta_N (\mf p) \Big]
\end{equation*}
for every $0<\epsilon<t$.  By property ({\bf T8}) for $k=\mf p$, the
first term vanishes as $N\uparrow\infty$ for every $\epsilon>0$. By
the convergence of the process $\Psi_{\mf p} (\eta^{N,\mf p}_{t
  \theta_N (\mf p)})$ to the Markov process with rates $\mf r_{\mf
  p}(i,j)$, assumed in ({\bf T6}), the second term converges as
$N\uparrow\infty$ to $\bb P_b [ T_a > (t-\epsilon) ]$, where $T_a$
stands for the hitting time of $a$. Since $\mc E^{(\mf p +1)}_i$ is a
leave, the asymptotic dynamics is a irreducible Markov process on the
set of indices $c \in \{1\, \dots, \nu(\mf p)\}$ such that $\mc
E^{(\mf p)}_c \subset \mc E^{(\mf p +1)}_i$. In particular, $\bb P_b [
T_a > (t-\epsilon) ]$ vanishes as $t\uparrow\infty$. This proves that
the first term in \eqref{21} vanishes as $N\uparrow\infty$ and then
$t\uparrow\infty$.

The second term in \eqref{21} vanishes as $N\uparrow\infty$ by
property ({\bf T7}) for $k=\mf p$. This proves the lemma.
\end{proof}

Next lemma shows that we may from now on restrict our attention to the
trace process $\{\eta^{N,\mf p}_t : t\ge 0\}$ in our investigation of the
metastability of $\{\eta^{N}_t : t\ge 0\}$ on a time scale longer than
$\theta_N(\mf p)$.

\begin{lemma}
\label{s16}
Assume that the trace process $\{\eta^{N,\mf p}_t : t\ge 0\}$ satisfy
condition ({\bf T8}) on some time scale $\theta_N \succ \theta_N (\mf
p)$ and for some subset $\Delta^*_{\mf p +1}$ of $\mc E^{(\mf p)}$:
\begin{equation*}
\lim_{N\to \infty} \max_{x\in \mc E^{(\mf p)}} \mb E^N_x \Big[ \int_0^t 
\mb 1\{ \eta^{N,\mf p}_{s \theta_N} \in \Delta^*_{\mf p +1} \} \, ds
\Big]\;=\; 0\;.
\end{equation*}
Then, the same property holds for the Markov process $\{\eta^N_t :
t\ge 0\}$ with $\Delta_{\mf p} \cup \Delta^*_{\mf p +1}$ in place of
$\Delta^*_{\mf p +1}$:
\begin{equation*}
\lim_{N\to \infty} \max_{x\in E_N} \mb E^N_x \Big[ \int_0^t 
\mb 1\{ \eta^{N}_{s \theta_N} \in \Delta_{\mf p} \cup 
\Delta^*_{\mf p +1} \} \, ds \Big]\;=\; 0\;.
\end{equation*}
\end{lemma}

\begin{proof}
Fix $x \in E$ and observe that
\begin{equation*}
\begin{split}
& \mb E^N_x \Big[ \int_0^t \mb 1\{ \eta^{N}_{s \theta_N} 
\in \Delta_{\mf p} \cup \Delta^*_{\mf p +1} \} \, ds \Big] \\ 
&\qquad \;\le\; \mb E^N_x \Big[ \int_0^t \mb 1\{ \eta^{N}_{s \theta_N} 
\in \Delta_{\mf p}\} \, ds \Big] \;+\; \max_{y\in \mc E^{(\mf p)}}
\mb E^N_y \Big[ \int_0^t \mb 1\{ \eta^{N, \mf p}_{s \theta_N} 
\in \Delta^*_{\mf p +1} \} \, ds \Big]\;.    
\end{split}
\end{equation*}
The second term vanishes as $N\uparrow\infty$ by assumption.
The first one is bounded by
\begin{equation*}
\frac {\theta_N(\mf p)}{\theta_N} \sum_{n=0}^{[\theta_N/\theta_N(\mf p)]} 
\mb E^N_x \Big[ \int_{nt}^{(n+1)t} \mb 1\{ \eta^{N}_{s \theta_N(\mf p)} 
\in \Delta_{\mf p}\} \, ds \Big]\;,
\end{equation*}
where $[r]$ stands for the integer part of $r$. By the Markov
property, this expression is bounded above by
\begin{equation*}
2\, \max_{y\in E} 
\mb E^N_y \Big[ \int_{0}^{t} \mb 1\{ \eta^{N}_{s \theta_N(\mf p)} 
\in \Delta_{\mf p}\} \, ds \Big]\; ,
\end{equation*}
which vanishes as $N\uparrow\infty$ in virtue of ({\bf T8}) for $k=\mf
p$.
\end{proof}

Consider the trace process $\{\eta^{N,\mf p}_t : t\ge 0\}$. By formula
(6.12) in \cite{bl2}, its invariant probability measure is the measure
$\mu_N$ conditioned to $\mc E^{(\mf p)}$, and by \cite[Lemma 6.9]{bl2}
the capacity between two disjoint subsets of $\mc E^{(\mf p)}$ for the
trace process $\{\eta^{N,\mf p}_t : t\ge 0\}$ is equal to the the
capacity for the original process divided by $\mu_N(\mc E^{(\mf p)})$.

The evolution the trace process $\{\eta^{N,\mf p}_t : t\ge 0\}$ on
$\mc E^{(\mf p)}$ is similar to the one of $\{\eta^{N}_t : t\ge 0\}$
among the shallowest valleys. We claim, for instance, that $(\mc
E^{(\mf p +1)}_i , \breve{\mc E}^{(\mf p +1)}_i$, $x)$, $x\in \mc
E^{(\mf p +1)}_i$, $1\le i\le \nu(\mf p +1)$, are valleys for the
trace process $\{\eta^{N,\mf p}_t : t\ge 0\}$.  The proof of this
assertion is divided in several steps. We first show that
\begin{equation}
\label{19}
G_N(\mc E^{(\mf p +1)}_i, \breve{\mc E}^{(\mf p +1)}_i) \;\prec\;
\frac{\mu_N(\mc E^{(\mf p +1)}_i)}{\theta_N(\mf p)}\;, 
\quad 1\le i\le \nu(\mf p +1)\; . 
\end{equation}

Indeed, since $\mc E^{(\mf p +1)}_i$ is a leave, there is no open path
from some $\mc E^{(\mf p)}_a \subset \mc E^{(\mf p +1)}_i$ to some
$\mc E^{(\mf p)}_b\not \subset \mc E^{(\mf p +1)}_i$. Therefore, since
by Lemma \ref{s09} $\mu_N(x) \approx \mu_N(\mc E^{(\mf p +1)}_i)$,
$x\in \mc E^{(\mf p +1)}_i$, by the definition of the average rate,
\begin{equation*}
\begin{split}
& \lim_{N\to\infty} \theta_N(\mf p) \, r^{\mc E^{(\mf p)}}_N 
\Big(\mc E^{(\mf p +1)}_i \,,\,\bigcup_b \mc E^{(\mf p)}_b \Big)\\
&\qquad =\; \lim_{N\to\infty} \theta_N(\mf p) \sum_{a} r^{\mc E^{(\mf p)}}_N 
\Big(\mc E^{(\mf p)}_a \,,\, \bigcup_b \mc E^{(\mf p)}_b \Big) \;=\; 0
\; ,
\end{split}
\end{equation*}
where the sum is performed over all $\mf p$-metastates $\mc E^{(\mf
  p)}_a \subset \mc E^{(\mf p +1)}_i$ and the union over all $\mf
p$-metastates $\mc E^{(\mf p)}_b\not \subset \mc E^{(\mf p
  +1)}_i$. Hence, by \cite[Lemma 6.7]{bl2} and Lemma \ref{s03},
\begin{equation}
\label{22}
\lim_{N\to\infty} \theta_N(\mf p) \, 
\frac{ G_N \Big(\mc E^{(\mf p +1)}_i ,\bigcup_b \mc E^{(\mf p)}_b
  \Big)}{\mu_N(\mc E^{(\mf p +1)}_i)} \;=\; 0\;.
\end{equation}
This proves \eqref{19} in view of \eqref{11} and because $\breve{\mc
  E}^{(\mf p +1)}_i \subset \cup_b \, \mc E^{(\mf p)}_b$. 

Recall the definition of the set $\Delta^o_{\mf p+1}$ introduced just
before Lemma \ref{s09}. Denote by $\mc B^{(\mf p +1)}_i$, $1\le i\le
\nu(\mf p +1)$, the union of all $\mf p$-metastates $\mc E^{(\mf
  p)}_b$ which have measure of lower magnitude than $\mc E^{(\mf p
  +1)}_i$ and which are contained in $\Delta^o_{\mf p+1}$. Let also
\begin{equation*}
\mc F^{(\mf p +1)}_i \;=\; \mc E^{(\mf p)} \setminus \big[ \mc E^{(\mf
  p +1)}_i \cup \mc B^{(\mf p+1)}_i \big]\;, \quad 1\le i\le \nu(\mf p+1)
\;.
\end{equation*}

\begin{lemma}
\label{s15}
Fix $1\le i\le \nu (\mf p+1)$ and $x$ in $\mc E^{(\mf p +1)}_i$.  The
triple $(\mc E^{(\mf p +1)}_i, \mc E^{(\mf p +1)}_i \cup \mc B^{(\mf p
  +1)}_i ,x)$ is a valley for the trace process $\{\eta^{N,\mf p}_t :
t\ge 0\}$ of depth $\theta_{N,i} = \mu_N(\mc E^{(\mf p +1)}_i)/$
$\Cap_N (\mc E^{(\mf p +1)}_i , \mc F^{(\mf p +1)}_i)$. Moreover,
$\theta_{N,i} \succ \theta_N(\mf p)$.
\end{lemma}

\begin{proof}
Fix $1\le i\le \nu (\mf p+1)$ and $x$ in $\mc E^{(\mf p +1)}_i$.  In
view of Theorem 2.6, formula (6.12) and Lemma 6.9 in \cite{bl2}, we
only need to check that
\begin{equation*}
\lim_{N\to\infty} \max_{y\in \mc E^{(\mf p +1)}_i} 
\frac{\Cap_N(\mc E^{(\mf p +1)}_i , \mc F^{(\mf p +1)}_i)}
{\Cap_N(x,y)}\;=\; 0\;.
\end{equation*}
This follows from Lemma \ref{s09}, Lemma \ref{s03}, \eqref{11} and
\eqref{22}.

It remains to show that $\theta_{N,i} \succ \theta_N(\mf p)$.  Since
$\mc F^{(\mf p +1)}_i$ is contained in $\cup_b \mc E^{(\mf p)}_b$,
where the union is performed over all $\mf p$-metastates which are not
contained in $\mc E^{(\mf p +1)}_i$, and since $\Cap_N(A,B) \le
\Cap_N(A,C)$ if $B\subset C$, by Lemma \ref{s03}, $\theta_N(\mf
p)/\theta_{N,i}$ is bounded above by
\begin{equation*}
C_1\, \theta_N(\mf p) \, 
\frac{ G_N \Big(\mc E^{(\mf p +1)}_i ,\bigcup_b \mc E^{(\mf p)}_b
  \Big)}{\mu_N(\mc E^{(\mf p +1)}_i)}  
\end{equation*}
for some finite constant $C_1$ independent of $N$. By \eqref{22} this
expression vanishes as $N\uparrow\infty$.
\end{proof}

Denote by $\mb P_x^{N,\mf p}$, $x\in \mc E^{(\mf p)}$, the probability
on the path space $D(\bb R_+, \mc E^{(\mf p)})$ induced by the trace
process $\{\eta^{N,\mf p}_t : t\ge 0\}$ starting from $x$.

\begin{lemma}
\label{s18}
Fix $1\le i\le \nu (\mf p+1)$ and $x$ in $\mc E^{(\mf p +1)}_i$.  The
triple $(\mc E^{(\mf p +1)}_i, \mc E^{(\mf p +1)}_i \cup \Delta^o_{\mf
  p+1} ,x)$ is a valley for the trace process $\{\eta^{N,\mf p}_t :
t\ge 0\}$ of depth $\theta_{N,i} = \mu_N(\mc E^{(\mf p +1)}_i)/$ $\Cap_N
$ $(\mc E^{(\mf p +1)}_i , \mc F^{(\mf p +1)}_i)$. Moreover, $\Cap_N
(\mc E^{(\mf p +1)}_i , \mc F^{(\mf p +1)}_i) \approx \Cap_N$ $(\mc
E^{(\mf p +1)}_i , \breve{\mc E}^{(\mf p +1)}_i)$.
\end{lemma}

\begin{proof}
Fix $1\le i\le \nu (\mf p+1)$ and recall the definition of
$\theta_{N,i}$ introduced in Lemma \ref{s15}. By this lemma and by
Lemma \ref{s04}, to prove the first assertion we need to show that for
every $\delta>0$, 
\begin{equation*}
\lim_{N\to\infty} \max_{y\in \Delta^o_{\mf p+1} \setminus \mc B^{(\mf p
  +1)}_i} \mb P^{N, \mf p}_y \big[ T_{\breve{\mc E}^{(\mf p +1)}_i} > 
\delta \theta_{N,i} \big] \;=\; 0\; .
\end{equation*}
Since, by Lemma \ref{s15}, $\theta_{N,i} \succ \theta_N(\mf p)$, it
is enough to show that 
\begin{equation*}
\lim_{t\to\infty} 
\lim_{N\to\infty} \max_{y\in \Delta^o_{\mf p+1} \setminus \mc B^{(\mf p
  +1)}_i} \mb P^{N, \mf p}_y \big[ T_{\breve{\mc E}^{(\mf p +1)}_i} > 
t\,\theta_N(\mf p)  \big] \;=\; 0\; .
\end{equation*}

Fix $y\in \Delta^o_{\mf p+1}\setminus \mc B^{(\mf p +1)}_i$. By
definition, $y$ belongs to some $\mf p$-metastate $\mc E^{(\mf p)}_b
\not \subset \mc E^{(\mf p +1)}_i$ and $\mu_N(\mc E^{(\mf p)}_b) \succeq
\mu_N(\mc E^{(\mf p +1)}_i)$. We claim that there is no open path from 
$\mc E^{(\mf p)}_b$ to $\mc E^{(\mf p +1)}_i$.

Indeed, suppose that there is an open path. In this case, since
$\mu_N(\mc E^{(\mf p)}_b) \succeq \mu_N(\mc E^{(\mf p +1)}_i)$, by
\eqref{37}, we necessarily have $\mu_N(\mc E^{(\mf p)}_b) \approx
\mu_N(\mc E^{(\mf p +1)}_i)$. Considering the last two $\mf
p$-metastates of the open path from $\mc E^{(\mf p)}_b$ to $\mc
E^{(\mf p +1)}_i$, we find a $\mf p$-metastate $\mc E^{(\mf p)}_c \not
\subset \mc E^{(\mf p +1)}_i$, $\mu_N(\mc E^{(\mf p)}_c) \approx
\mu_N(\mc E^{(\mf p +1)}_i)$, and a $\mf p$-metastate $\mc E^{(\mf
  p)}_a \subset \mc E^{(\mf p +1)}_i$ such that $\mf r_{\mf
  p}(c,a)>0$. Therefore, by \eqref{37} and by reversibility,
\begin{equation*}
\begin{split}
& \mf r_{\mf p}(a,c) \;=\; \lim_{N\to\infty} \theta_N(\mf p)\,
r^{\mc E^{(\mf p)}}_N (\mc E^{(\mf p)}_a , \mc E^{(\mf p)}_c) \\
&\qquad \;=\; \lim_{N\to\infty} 
\frac{\mu_N(\mc E^{(\mf p)}_c)}{\mu_N(\mc E^{(\mf p)}_a)}\,
\theta_N(\mf p)\, r^{\mc E^{(\mf p)}}_N (\mc E^{(\mf p)}_c , 
\mc E^{(\mf p)}_a) \;=\; \mf r_{\mf p}(a,c) 
\lim_{N\to\infty} 
\frac{\mu_N(\mc E^{(\mf p)}_c)}{\mu_N(\mc E^{(\mf p)}_a)} \;>\; 0\;,   
\end{split}
\end{equation*}
which contradicts the fact that $\mc E^{(\mf p +1)}_i$ is a leave.

By \eqref{20} with $k=\mf p$, starting from $y$ the process $X^{N,\mf
  p}_{t\theta_N(\mf p)}$ converges to the Markov process on $\{1,
\dots, \nu(\mf p)\}$ with rates $\mf r_{\mf p}$ starting from
$b$. Therefore,
\begin{equation*}
\lim_{N\to\infty} \mb P^{N, \mf p}_y \big[ T_{\breve{\mc E}^{(\mf p +1)}_i} > 
t\,\theta_N(\mf p)  \big] \;\le \; \bb P_b \big[ T_{A} > t \big]\;,
\end{equation*}
where $A = \{c : \mc E^{(\mf p)}_c \subset \breve{\mc E}^{(\mf p
  +1)}_i\}$. Since there is no open path from $\mc E^{(\mf p)}_b$ to
$\mc E^{(\mf p +1)}_i$ and since $\mc E^{(\mf p)}_b\subset
\Delta^o_{\mf p+1}$, the state $b$ is transient for the Markov process
on $\{1, \dots, \nu(\mf p)\}$ with rates $\mf r_{\mf p}$ and all its
limit points are contained in $A$. Hence, $\bb P_b [ T_{A} > t ]$
vanishes as $t\uparrow\infty$. This proves the first assertion of the
lemma.

To prove the second statement, note that $\Cap_N (\mc E^{(\mf p
  +1)}_i, \mc F^{(\mf p +1)}_i) \succeq \Cap_N$ $(\mc E^{(\mf p +1)}_i
, \breve{\mc E}^{(\mf p +1)}_i)$ because $\breve{\mc E}^{(\mf p +1)}_i
\subset \mc F^{(\mf p +1)}_i$.

By Lemma \ref{s03}, to prove the reverse inequality we may replace the
capacities by the function $G_N$. There exists a path $\gamma = (x_0,
\dots, x_n)$ from $\mc E^{(\mf p +1)}_i$ to $\mc F^{(\mf p +1)}_i$
such that $G_N(\gamma) = G_N(\mc E^{(\mf p +1)}_i, \mc F^{(\mf p
  +1)}_i)$. If $x_n$ belongs to $\breve{\mc E}^{(\mf p +1)}_i$, we
have that $G_N(\gamma) \le G_N(\mc E^{(\mf p +1)}_i , \breve{\mc
  E}^{(\mf p +1)}_i)$ and the statement is proved.

If, on the other hand, $x_n$ belongs to some metastate $\mc E^{(\mf
  p)}_b \subset \Delta^o_{\mf p+1} \setminus \mc B^{(\mf p +1)}_i$ we
proceed as follows. We have already showed in the first part of the
proof that there exists an open path from $\mc E^{(\mf p)}_b$ to
$\breve{\mc E}^{(\mf p +1)}_i$. Repeating the arguments presented in
the proof of Lemma \ref{s09} and keeping in mind the second assertion
of \eqref{37}, we show that there exists a path $\tilde \gamma$ from
$x_n$ to $\breve{\mc E}^{(\mf p +1)}_i$ such that $G_N(\tilde \gamma)
\ge C_0 \mu_N(x_n)/\theta_N(\mf p)$ for some finite constant $C_0$
independent of $N$. By definition of $\mc B^{(\mf p +1)}_i$, this
latter expression is bounded below $C_0 \mu_N(\mc E^{(\mf p
  +1)}_i)/\theta_N(\mf p)$. By \eqref{22}, $G_N(\gamma) = G_N(\mc
E^{(\mf p +1)}_i, \mc F^{(\mf p +1)}_i) \prec \mu_N(\mc E^{(\mf p
  +1)}_i)/\theta_N(\mf p)$.  Hence, if we denote by $\gamma \oplus
\tilde \gamma$ the juxtaposition of $\gamma$ and $\tilde \gamma$, we
have a path $\gamma \oplus \tilde \gamma$ from $\mc E^{(\mf p +1)}_i$
to $\breve{\mc E}^{(\mf p +1)}_i$ such that $G_N(\gamma \oplus \tilde
\gamma) = G_N(\gamma) = G_N(\mc E^{(\mf p +1)}_i, \mc F^{(\mf p
  +1)}_i)$.  This proves the second assertion of the lemma.
\end{proof}

It follows from the two previous lemmas that the depth $\theta_{N,i}$
of the valley $(\mc E^{(\mf p +1)}_i$, $\mc E^{(\mf p +1)}_i \cup
\Delta^o_{\mf p+1} ,x)$, $1\le i\le \nu(\mf p+1)$, is of the same
magnitude as $\mu_N(\mc E^{(\mf p +1)}_i) / \Cap_N$ $(\mc E^{(\mf p
  +1)}_i , \breve{\mc E}^{(\mf p +1)}_i)$ and much larger than
$\theta_N(\mf p)$. 

Fix a subset $I$ of $\{1, \dots, \nu(\mf p +1)\}$ and let $J=\{1,
\dots, \nu(\mf p+1)\} \setminus I$, $\mc E_{K, \mf p+1} = \cup_{i\in
  K} \mc E^{(\mf p+1)}_i$, $K\subset \{1, \dots, \nu(\mf p+1)\}$. By
Lemma \ref{s29}, the following limit exists
\begin{equation*}
f^{(\mf p+1)}_{I,J}(x) \;:=\; 
\lim_{N\to\infty} \mb P^{N}_x[T_{\mc E_{I, \mf p+1}} < 
T_{\mc E_{J, \mf p+1}}] \;.
\end{equation*}
In particular, by Lemma \ref{s21},
\begin{equation*}
\lim_{N\to\infty} \frac {\Cap_N(\mc E^{(\mf p+1)}_I, \mc E^{(\mf p+1)}_J)}
{\mf g_N(\mc E^{(\mf p+1)}_I,\mc E^{(\mf p+1)}_J)}  \;=\; 
\frac 12 \sum g(x,y) 
\, [f^{(\mf p+1)}_{I,J} (y) - f^{(\mf p+1)}_{I,J} (x)]^2 \;\in\;
(0,\infty) \;,
\end{equation*}
where the sum on the right hand side is performed over all pairs
$(x,y) \in \mf B(\mc E^{(\mf p+1)}_I, \mc E^{(\mf p+1)}_J)$.

By the same reasons, the limit 
\begin{equation*}
f^{(\mf p+1)}_{i}(x) \;:=\; 
\lim_{N\to\infty} \mb P^{N}_x\Big[T_{\mc E^{(\mf p+1)}_{i}} < 
T_{\mc F^{(\mf p +1)}_i}\Big] 
\end{equation*}
exists and
\begin{equation*}
\lim_{N\to\infty} \frac {\Cap_N(\mc E^{(\mf p+1)}_{i} , 
\mc F^{(\mf p +1)}_i)}
{\mf g_N(\mc E^{(\mf p+1)}_{i} , \mc F^{(\mf p +1)}_i )} \;=\;  
\frac 12 \sum g(x,y) 
\, [f^{(\mf p+1)}_{i} (y) - f^{(\mf p+1)}_{i} (x)]^2  \;\in\;
(0,\infty) \;,
\end{equation*}
where the sum on the right hand side is performed over all pairs
$(x,y) \in \mf B(\mc E^{(\mf p+1)}_{i} , \mc F^{(\mf p +1)}_i)$.

Let $g^{(\mf p+1)}_i = f^{(\mf p+1)}_{I,J}$ for $I=\{i\}$.  We claim
that $g^{(\mf p+1)}_i = f^{(\mf p+1)}_{i}$, in other words, that for
all $x\in E$,
\begin{equation}
\label{39}
\lim_{N\to\infty} \mb P^{N}_x[T_{\mc E^{(\mf p +1)}_{i}} < 
T_{\mc F^{(\mf p +1)}_i}] \;=\;
\lim_{N\to\infty} \mb P^{N}_x[T_{\mc E^{(\mf p +1)}_{i}} < 
T_{\breve{\mc E}^{(\mf p +1)}_i}] \;.
\end{equation}

Indeed, fix $x\in E$.  Since $\lim_{N\to\infty} \mb P^{N}_x[T_{\mc
  E^{(\mf p)}} =T_y]$, $y\in \mc E^{(\mf p)}$, exists by Lemma
\ref{s29}, and since all sets involved are contained in $\mc E^{(\mf
  p)}$, taking conditional expectation with respect to $T_{\mc E^{(\mf
    p)}}$ and applying the strong Markov property, to prove \eqref{39}
it is enough to show that for all $y\in \mc E^{(\mf p)}$
\begin{equation*}
\lim_{N\to\infty} \mb P^{N}_y[T_{\mc E^{(\mf p +1)}_{i}} < 
T_{\mc F^{(\mf p +1)}_i}] \;=\;
\lim_{N\to\infty} \mb P^{N}_y[T_{\mc E^{(\mf p +1)}_{i}} < 
T_{\breve{\mc E}^{(\mf p +1)}_i}] \;.
\end{equation*}
At this point we may replace the process $\eta^N_t$ by the trace
process $\eta^{N,\mf p}_t$. Since $\breve{\mc E}^{(\mf p +1)}_i$ is
contained in $\mc F^{(\mf p +1)}_i$, by the strong Markov property, to
prove the previous identity we have to show that for every $z\in \mc
F^{(\mf p +1)}_i$
\begin{equation*}
\lim_{N\to\infty} \mb P^{N}_z[T_{\mc E^{(\mf p +1)}_{i}} < 
T_{\breve{\mc E}^{(\mf p +1)}_i}] \;=\; 0\;.
\end{equation*}
Since there is no open path from $\mc F^{(\mf p +1)}_i$ to $\mc
E^{(\mf p +1)}_{i}$, since by \eqref{20} the speeded-up blind process
$X^{N,\mf p}_t$ converges to the Markov process with rates $\mf r_{\mf
  p}$ whose recurrent states are the indices $a \in \{1, \dots,
\nu(\mf p)\}$ such that $\mc E^{(\mf p)}_a \subset \mc E^{(\mf p
  +1)}$, the previous identity holds, proving claim \eqref{39}.

We proved in Lemma \ref{s18} that $\Cap_N (\mc E^{(\mf p +1)}_i , \mc
F^{(\mf p +1)}_i) \approx \Cap_N$ $(\mc E^{(\mf p +1)}_i , \breve{\mc
  E}^{(\mf p +1)}_i)$. Hence, by Lemma \ref{s03}, $G_N (\mc E^{(\mf p
  +1)}_i , \mc F^{(\mf p +1)}_i) \approx G_N$ $(\mc E^{(\mf p +1)}_i ,
\breve{\mc E}^{(\mf p +1)}_i)$. In particular, $\mf g_N (\mc E^{(\mf p
  +1)}_i , \mc F^{(\mf p +1)}_i) =\mf g_N (\mc E^{(\mf p +1)}_i ,
\breve{\mc E}^{(\mf p +1)}_i)$ and, in consequence, $\mf B(\mc E^{(\mf
  p +1)}_i , \mc F^{(\mf p +1)}_i) = \mf B(\mc E^{(\mf p +1)}_i ,
\breve{\mc E}^{(\mf p +1)}_i)$. 

If follows from the previous considerations that
\begin{equation*}
\lim_{N\to\infty} \frac{\Cap_N(\mc E^{(\mf p +1)}_i ,
\breve{\mc E}^{(\mf p +1)}_i)}{\mf g_N (\mc E^{(\mf p +1)}_i ,
\breve{\mc E}^{(\mf p +1)}_i)} \;=\;
\lim_{N\to\infty} \frac{\Cap_N(\mc E^{(\mf p +1)}_i ,
\mc F^{(\mf p +1)}_i)}{\mf g_N (\mc E^{(\mf p +1)}_i ,
\breve{\mc E}^{(\mf p +1)}_i)} \;\in\; (0,\infty)\;,
\end{equation*}
so that
\begin{equation*}
\lim_{N\to\infty} \frac{\Cap_N(\mc E^{(\mf p +1)}_i ,
\breve{\mc E}^{(\mf p +1)}_i)}{\Cap_N(\mc E^{(\mf p +1)}_i ,
\mc F^{(\mf p +1)}_i)} \;=\; 1\;.
\end{equation*}
In consequence, by Lemma \ref{s18}, the following result holds.

\begin{lemma}
\label{s34}
Fix $1\le i\le \nu (\mf p+1)$ and $x$ in $\mc E^{(\mf p +1)}_i$.  The
triple $(\mc E^{(\mf p +1)}_i, \mc E^{(\mf p +1)}_i \cup \Delta^o_{\mf
  p+1} ,x)$ is a valley for the trace process $\{\eta^{N,\mf p}_t :
t\ge 0\}$ of depth $\theta_{N,i} = \mu_N(\mc E^{(\mf p +1)}_i)/$
$\Cap_N (\mc E^{(\mf p +1)}_i , \breve{\mc E}^{(\mf p
  +1)}_i)$. Moreover,
\begin{equation*}
u_{\mf p+1, i} \;:=\; \lim_{N\to\infty} \frac{\mf g_N (\mc E^{(\mf p +1)}_i ,
\breve{\mc E}^{(\mf p +1)}_i)}{\mu_N(\mc E^{(\mf p +1)}_i)}\,
\theta_{N,i} \;\in\; (0,\infty)\;.
\end{equation*}
\end{lemma}

Since the sequences $\mf g_N (\mc E^{(\mf p +1)}_i , \breve{\mc
  E}^{(\mf p +1)}_i)/\mu_N(\mc E^{(\mf p +1)}_i)$, $1\le i \le \nu(\mf
p+1)$, are comparable, repeating the arguments presented in the proof
of Proposition \ref{s24} we deduce the next result.

\begin{lemma}
\label{s28}
The sequences $\{\theta_{N,i} : N\ge 1\}$, $1\le i \le \nu(\mf p+1)$
are comparable.
\end{lemma}

Let $\theta_N(\mf p +1) = \min \{\theta_{N,i} : 1\le i\le \nu(\mf
p+1)\}$ and let $S_{\mf p+1} = \{i : \theta_{N,i} \approx \theta_N(\mf
p +1)\}$. Observe that $\theta_N(\mf p) \prec \theta_N(\mf p +1)$ and
that ({\bf T5}) holds for $k=\mf p+1$ with this definition.  \medskip

Denote by $X^{N,\mf p+1}_t = \Psi_k(\eta^{N, \mf p+1}_{t\theta_N(\mf p
  +1)})$ the speeded up blind process introduced in the statement of
Theorem \ref{s19}.  

\begin{lemma}
\label{s35}
Condition {\rm({\bf T6})} holds for $k=\mf p+1$.
\end{lemma}

\begin{proof}
The arguments presented in Section \ref{ssec2} until Lemma \ref{s25}
apply to the present context and show that conditions \eqref{37} are
fulfilled for $k=\mf p+1$.

It remains to prove the convergence of $X^{N,\mf p+1}_t$. We need to
check that the assumptions of \cite[Theorem 2.7]{bl2} are
fulfilled. On the one hand, condition ({\bf H1}) follows from
condition ({\bf T4}) for $k=\mf p+1$, proved in Lemma \ref{s09}, and
from the fact that $\theta_{N,i} \succeq \theta_N(\mf p+1) \succ
\theta_N(\mf p)$, proved right after Lemma \ref{s28}.  On the other
hand, condition ({\bf H0}) is part of \eqref{37} which has already
been proven.
\end{proof}

To conclude the recurrence argument, it remains to show that property
({\bf T8}) holds for $k=\mf p+1$. We first show that it holds for the
trace process $\eta^{N,\mf p}_t$.

\begin{lemma}
\label{s10}
For all $t>0$,
\begin{equation*}
\lim_{N\to \infty} \max_{x\in \mc E^{(\mf p)}} \mb E^N_x \Big[ \int_0^t 
\mb 1\{ \eta^{N,\mf p}_{s \theta_N(\mf p+1)} \in \Delta^o_{\mf p +1} \} 
\, ds \Big]\;=\; 0\;.
\end{equation*}
\end{lemma}

\begin{proof}
Since $\theta_N(\mf p) \prec \theta_N(\mf p+1)$, a change of variables
in the time integral and the Markov property show that for every
$T>0$ and for every $N$ large enough,
\begin{equation*}
\mb E^N_x \Big[ \int_0^t 
\mb 1\{ \eta^{N,\mf p}_{s \theta_N(\mf p+1)} \in \Delta^o_{\mf p +1} \} 
\, ds \Big] \;\le\; \frac {2t}T\,
\max_{y\in \mc E^{(\mf p)}}  \mb E^N_y \Big[ \int_0^T
\mb 1\{ \eta^{N,\mf p}_{s \theta_N(\mf p)} \in \Delta^o_{\mf p +1} \} 
\, ds \Big]
\end{equation*}
for every $x\in \mc E^{(\mf p)}$. Note that the process on the right
hand side is speeded up by $\theta_N(\mf p)$ and not by $\theta_N(\mf
p+1)$ anymore.

We estimate the expression on the right hand side of the previous
formula. We may, of course, restrict the maximum to $\Delta^o_{\mf p
  +1}$. Let $T_1$ be the first time the trace process hits $\mc
E^{(\mf p+1)}$ and let $T_2$ be the time it takes for the process to
return to $\Delta^o_{\mf p +1}$ after $T_1$:
\begin{equation*}
T_1 \;=\; T_{\mc E^{(\mf p+1)}}\; , \quad
T_2 \;=\; \inf \big\{ s> 0 : \eta^{N,\mf p}_{T_1+s} \in \Delta^o_{\mf p
  +1}\big\} \; .
\end{equation*}

Fix $x\in \Delta^o_{\mf p +1}$ and note that
\begin{equation*}
\begin{split}
&  \mb E^N_x \Big[ \frac 1T \int_0^T
\mb 1\{ \eta^{N,\mf p}_{s \theta_N(\mf p)} \in \Delta^o_{\mf p +1} \} 
\, ds \Big]  \\
& \qquad \;\le\; 
\mb P^{N,\mf p}_x \big[ T_1 > t_0 \theta_N(\mf p) \big]\; +\; 
\mb P^{N,\mf p}_x \big[ T_{2} \le T \theta_N(\mf p) \big] 
\;+\; \frac{t_0}T
\end{split}
\end{equation*}
for all $t_0>0$. We have proved, in Lemma \ref{s18} for instance, that
the first term on the right hand side vanishes as $N\uparrow\infty$
and then $t_0\uparrow\infty$. By the strong Markov property, the
second term is bounded by $\max_{y\in \mc E^{(\mf p+1)}} \mb P^{N,\mf p}_y [ 
T_{\Delta^o_{\mf p +1}} \le T \theta_N(\mf p) ]$. Since there is no
open path from $\mc E^{(\mf p+1)}$ to $\Delta^o_{\mf p +1}$ this
probability vanishes as $N\uparrow\infty$ for all $T>0$. This
concludes the proof.
\end{proof}

Next result follows from Lemma \ref{s16} and Lemma \ref{s10} and
concludes the proof of Theorem \ref{s37}.

\begin{corollary}
\label{s36}
Condition {\rm ({\bf T8})} holds for $k=\mf p+1$: 
\begin{equation*}
\lim_{N\to \infty} \max_{x\in E} \, \mb E^N_x \Big[
\int_0^t \mb 1\{ \eta^N_{s\theta_N(\mf p+1)} \in \Delta_{\mf p+1}\} 
\, ds  \Big] \;=\; 0\;. 
\end{equation*}
\end{corollary}

We conclude this section with a remark. Fix a level $\mf q$ and denote
by $P_N(x,i,j)$, $1\le i\not = j\le \nu(\mf q)$, $x\in \mc E^{(\mf
  q)}_i$, the hitting probabilities
\begin{equation}
\label{38}
P_N(x,i,j) \;:=\; 
\mb P^N_x \big[ T_{\mc E^{(\mf q)}_j} = T_{\breve{\mc E}^{(\mf q)}_i} \big]\;.
\end{equation}
Recall from Lemma \ref{s34} that $\theta_{N,i}= \mu_N(\ms E^{(\mf
  q)}_i)/ \Cap_N(\mc E^{(\mf q)}_i , \breve{\mc E}^{(\mf q)}_i)$. It
follows from Lemma \ref{s28} with $\mf q=\mf p+1$ that $\theta_N(\mf
q)/\theta_{N,i}$ converges to some number denoted by $\Lambda(i) \in
[0,\infty)$.  On the other hand, by Lemma \ref{s29}, $P_N(x,i,j)$
converges to some $P(x,i,j)\in [0,1]$. Since by Lemma \ref{s34} $(\mc
E^{(\mf q)}_i, \mc E^{(\mf q)}_i \cup \Delta^o_{\mf q}, y)$, $y\in \mc
E^{(\mf q)}_i$, is a valley for the trace process $\eta^{N,\mf q}_t$,
it is not difficult to show that the limit $P(x,i,j)$ does not depend
on the starting point $x$. Therefore, by Lemma \ref{s39}, for any
$1\le i\not = j\le \nu(\mf q)$,
\begin{equation}
\label{41}
\mf r_{\mf q}(i,j) \;=\; \lim_{N\to\infty} \frac{\theta_N(\mf
  q)}{\theta_{N,i}}\; \lim_{N\to\infty}
\mb P^N_x \big[ T_{\mc E^{(\mf q)}_j} = T_{\breve{\mc E}^{(\mf q)}_i}
\big]\;. 
\end{equation}

\section{Valleys and Hitting times of the Ising model at low
  temperature}
\label{si2}

The proof of Theorem \ref{t04} follows the strategy presented in the
previous sections. As we have seen, the approach relies on the
characterization of the shallowest valleys of the model and on the
computation of the depths and the hitting times of these valleys. We
present in this section the shallowest valleys of the Ising model at
low temperature and some estimates of the capacities and the hitting
times.

In the present context, a path $\gamma = (\eta_0, \dots, \eta_p)$ is a
sequence of configuration in $\Omega$ such that for each $0\le j<p$,
$\eta_{j+1} = \eta^{x}_j$ for some $x\in\Lambda_L$.  We shall say
that two configurations $\xi$ and $\eta$ in $\Omega$ are neighbors if
$\xi = \eta^x$ for some $x\in \Lambda_L$.

\begin{lemma}
\label{t01}
Fix a configuration $\sigma$ in $\Omega_o$, $\sigma \not = + \mb 1, -
\mb 1$. For all $\beta>0$,
\begin{equation}
\label{e01}
G_\beta(\{\sigma\} , \Omega_\sigma) \;=\; 
\begin{cases}
\mu_\beta (\sigma) \, e^{- \beta [\ell(\sigma) -1] h} 
& \text{if $\ell(\sigma) \le n_0$,} \\
\mu_\beta (\sigma) \, e^{-\beta (2-h)} & \text{otherwise.}
\end{cases}
\end{equation}
Moreover,
\begin{equation*}
G_\beta(\{-\mb 1\} , \Omega_{-\mb 1}) \;=\; 
\mu_\beta (-\mb 1) \, e^{- \beta (8-3 h)} \;,\quad
G_\beta(\{+\mb 1\} , \Omega_{+\mb 1}) \;=\; 
\mu_\beta (+\mb 1) \, e^{- \beta (8+3 h)}\;.
\end{equation*}

\end{lemma}

\begin{proof}
Fix a configuration $\sigma$ satisfying the assumptions of the lemma
and assume that $\ell :=\ell(\sigma) \le n_0$. Fix a positive
rectangle $R$ of $\sigma$ of size $\ell \times m$ and assume that
$m\ge 3$. Consider the sequence of configurations $\sigma = \eta_0,
\dots, \eta_\ell$ obtained by first flipping the spin at a corner of
the rectangle $R$ and then flipping contiguous spins along the smaller
side. The last configuration $\eta_\ell$ is the configuration $\sigma$
where the rectangle $R$ has been replaced by a rectangle $R'\subset R$
of size $\ell \times (m-1)$.

The configuration $\eta_\ell$ belongs to $\Omega_\sigma$ and the path
$\gamma$ to $\Gamma_{\{\sigma\} , \Omega_\sigma}$. A simple
computation shows that $\mu_\beta(\eta_{\ell-1}) =
\min\{\mu_\beta(\eta_k) : 0\le k\le \ell\}$ so that
$G_\beta(\{\sigma\} , \Omega_\sigma) \ge G_\beta(\gamma) =
\mu_\beta(\eta_{\ell-1}) = \mu_\beta(\sigma) e^{-\beta (\ell-1)h}$.

To prove the reverse inequality, note that the configuration $\sigma$
has five types of different neighbors $\sigma^x$. A simple computation
shows that $\mu_\beta(\sigma^x) < \mu_\beta(\sigma) e^{-\beta
  (\ell-1)h}$ in four cases because $\ell \le n_0 < 2/h$. The only
type where this inequality does not hold occurs when we flip the spin
at a corner of a positive rectangle of $\sigma$.

To compute $G_\beta(\{\sigma\} , \Omega_\sigma)$ we need to maximize
$G_\beta(\gamma)$ over all paths $\gamma$ from $\sigma$ to
$\Omega_\sigma$. The previous observations shows that the unique
possible paths are those where we start flipping the corner of a
positive rectangle.

This argument can be iterated. At each step we are only allowed to
flip a positive spin which has two negative neighbors. After $k$ flips
we reach configurations of measure $\mu_\beta(\sigma) e^{-\beta k
  h}$. Since we are not allowed to pass the level $\mu_\beta(\sigma)
e^{-\beta (\ell-1)h}$, the only configurations in $\Omega_o$ which can
be reached after $\ell$ flips are the ones where a rectangle $R$ of
length $\ell \times m$ is replaced by a rectangle $R'\subset R$ of
length $\ell \times (m-1)$.

The case of a rectangle $R$ of size $2\times 2$ is treated in a
similar way. In this case, once one corner is removed, the next spins
of the square flip at rate one to reach the configuration where the
square $R$ is removed. This proves the lemma in the case $\ell
(\sigma) \le n_0$.

Assume now that $\ell (\sigma) > n_0$. Consider the path $\gamma =
(\sigma = \eta_0, \dots, \eta_m)$, where $\eta_1$ is the configuration
obtained from $\sigma$ by flipping a negative spin contiguous to a
positive rectangle, and where $\eta_{j+1}$ is obtained from $\eta_j$,
$2\le j<m$, by flipping a negative spin surrounded by two positive
spins. The final configuration $\eta_m$ is reached when no negative
spin has two positive neighbors.  Clearly, $\mu_\beta(\eta_{1}) =
\min\{\mu_\beta(\eta_k) : 0\le k\le m\}$ so that $G_\beta(\{\sigma\} ,
\Omega_\sigma) \ge G_\beta(\gamma) = \mu_\beta(\eta_{1}) =
\mu_\beta(\sigma) e^{-\beta (2-h)}$.

A similar argument to the one presented in the first part of the proof
of this lemma shows that the path proposed is the optimal one. This
concludes the proof of the first part of the lemma.

Consider the the path $\gamma = (\sigma_0= -\mb 1, \sigma_1, \dots,
\sigma_4)$ where $\sigma_{j+1}$ is the configuration obtained from
$\sigma_j$, $0\le j\le 3$, by flipping a negative spin from a site
with the largest possible number of neighbors with a positive
spin. Hence, $\sigma_4\in \Omega_{-\mb 1}$ is obtained from $-\mb 1$
by flipping the spins of a $2\times 2$ square and $G_\beta(\gamma)=
\mu_\beta (-\mb 1) \, e^{- \beta (8-3 h)}$. In particular,
$G_\beta(\{-\mb 1\} , \Omega_{-\mb 1}) \ge \mu_\beta (-\mb 1) \,
e^{- \beta (8-3 h)}$.

To prove the reverse inequality, consider a path $\gamma =
(\sigma_0, \dots, \sigma_p)$ from $-\mb 1$ to $\Omega_{-\mb 1}$. Let
$\sigma_i$ be the first configuration in the path $\gamma$ which has
three positive spins. A simple computation shows that $\mu_\beta
(\sigma_i) \le \mu_\beta (-\mb 1) \, e^{- \beta (8-3 h)}$. This proves
that $G_\beta(\gamma) \le \mu_\beta (-\mb 1) \, e^{- \beta (8-3 h)}$
so that $G_\beta(\{-\mb 1\} , \Omega_{-\mb 1}) \le \mu_\beta (-\mb 1)
\, e^{- \beta (8-3 h)}$, which proves the penultimate assertion of the
lemma. The last statement is proved in a similar way. 
\end{proof} 

Recall the definition of the transition probabilities $p(\sigma,
\sigma')$, $\sigma\in \Omega_o$, $\sigma'\in \bb S(\sigma)$,
introduced in \eqref{e04}, \eqref{e06}. For $\ell(\sigma) =2$ and
$\ell(\sigma)>n_0$, cases where $\bb S(\sigma) = \bb D(\sigma)$, let
$q(\sigma, \sigma')=p(\sigma, \sigma')$. For $\sigma\in \Omega_o$,
$3\le \ell(\sigma) \le n_0$, $\sigma'\in \bb D(\sigma)$, let
$q(\sigma, \sigma')$ be defined by
\begin{equation*}
q(\sigma, \sigma')\;=\; \frac 1{|\bb D(\sigma)|}\;\cdot
\end{equation*}
Note that $q(\sigma, \sigma') = p(\sigma, \sigma')$ for $\sigma'\in
\bb D(\sigma) \cap \bb S(\sigma) = \bb D(\sigma) \cap
\Omega_{o,\ell(\sigma)-1}$.

\begin{lemma}
\label{t03}
Fix a configuration $\sigma$ in $\Omega_o$, $\sigma \not = + \mb 1, -
\mb 1$, and a configuration $\sigma'\in \bb D(\sigma)$. Then,
\begin{equation*}
\lim_{\beta\to\infty} \mb P^\beta_\sigma \big[ T_{\sigma'} = 
T_{\Omega_\sigma} \big] \;=\; q(\sigma, \sigma')\;.
\end{equation*} 
\end{lemma}

\begin{proof}
Fix a configuration $\sigma$ satisfying the assumptions of the lemma,
a configuration $\sigma'\in \bb D(\sigma)$ and assume that $3\le \ell
(\sigma) \le n_0$.  Denote by $\bb W(\sigma, \sigma')$ the set of
configurations in $\bb W(\sigma)$ which are equal to $\sigma'$ when we
flip the positive spin surrounded by three negative spins. Note that
$|\bb W(\sigma, \sigma')| = \ell (\sigma)$.

We present the proof for $\ell(\sigma)= 3$, the other cases being
analogous.  Since $3 = \ell(\sigma) \le n_0 < 2/h$, we have that
$h<2/3$. For a configuration $\eta$ for which all positive spins are
surrounded by at most two negative spins, let $F_1(\eta)$ be the set
of all configurations obtained from $\eta$ by flipping a positive spin
surrounded by two negative spins.

Let $f_\beta(\eta) = \mb P^\beta_\eta [ T_{\sigma'} =
T_{\Omega_\sigma} ]$ and denote by $f$ a limit point of the sequence
$f_\beta$, as $\beta\uparrow\infty$. We need to show that $f(\sigma) =
1/|\bb D(\sigma)|$. Since $f_\beta$ is harmonic, a simple computation
shows that
\begin{equation}
\label{e02}
f_\beta(\sigma) \;=\; \frac 1{|F_1(\sigma)|} \sum_{\xi \in F_1(\sigma)}
f_\beta(\xi) \;+\; o_1(\beta)\;,
\end{equation}
where $o_1(\beta)$ is an expression absolutely bounded by $C_0 \exp\{
-2\beta [1-h]\}$ for some finite constant $C_0$ independent of $\beta$
which may change from line to line. It follows from this identity that
$f(\sigma) \;=\; |F_1(\sigma)|^{-1} \sum_{\xi\in F_1(\sigma)} f(\xi)$.

A similar argument shows that $f(\eta)=f(\sigma)$ for any
configuration $\eta$ in $F_1(\sigma)$. Let $F_2(\sigma)$ be the set of
configurations obtained from a configuration in $F_1(\sigma)$ by
flipping a positive spin surrounded by two negative spins. By the same
reasons, $f(\xi)=f(\sigma)$ for any configuration $\xi$ in
$F_2(\sigma) \setminus \bb W(\sigma)$. Fix now a configuration $\eta$
in $\bb W(\sigma, \sigma')$. If $\eta$ differs from $\sigma'$ by a
spin in a corner of a positive rectangle of $\sigma$, $f(\eta) = (1/2)
[1+f(\sigma)]$, while if $\eta$ differs from $\sigma'$ by a spin not
in a corner, $f(\eta) = (1/3) [1+2f(\sigma)]$. For a configuration
$\eta$ in $\bb W(\sigma) \setminus \bb W(\sigma, \sigma')$, if $\eta$
differs from $\sigma'$ by a spin in a corner of a positive rectangle
of $\sigma$, $f(\eta) = (1/2) f(\sigma)$, while if $\eta$ differs from
$\sigma'$ by a spin not in a corner, $f(\eta) = (2/3) f(\sigma)$.

Finally, observe that applying the harmonic identity to the terms
$f_\beta(\xi)$ in equation \eqref{e02}, after some elementary
computations we obtain that
\begin{equation*}
\sum_{\xi\in F_1(\sigma)} \sum_{\eta\in F_1(\xi)} 
\{ f_\beta(\eta) - f_\beta(\sigma) \} \;=\;  o_2(\beta) \;,
\end{equation*}
where $o_2(\beta)$ is absolutely bounded by $C_0 \{e^{-\beta h} +
e^{-\beta [2-3h]}\}$. Since $h<2/3$, the right hand side vanishes as
$\beta\uparrow\infty$ so that $\sum_{\eta\in F_2(\sigma)} \{ f(\eta) -
f(\sigma) \} = 0$. By the previous identities, this relation is
reduced to $\sum_{\eta\in \bb W(\sigma)} \{ f(\eta) - f(\sigma) \} =
0$. From this identity and the explicit values of $f$ in $\bb
W(\sigma)$, we obtain that $f(\sigma) = 1/|\bb D(\sigma)|$, which
proves the lemma.

Suppose now that $\ell(\sigma)=2\le n_0$ and note that equation
\eqref{e02} holds. The argument is analogous to the previous one, with
one difference. If $\xi\in F_1(\sigma)$ is configuration in which a
spin of a $2\times 2$ positive square $Q$ of $\sigma$ has been
flipped, we have that $3 f(\xi) = f(\sigma) + f(\eta_1) + f(\eta_2)$,
where $f$ is any limit point of the sequence $f_\beta$ and $\eta_1$,
$\eta_2$ are configurations obtained from $\sigma$ by flipping a row
or a column of the square $Q$. Iterating the argument based on the
harmonicity of $f_\beta$, we conclude that $3 f(\xi) = f(\sigma) + 2
f(\sigma^*)$, where $\sigma^*$ is the configuration obtained from
$\sigma$ by flipping all spins of $Q$.

The proof for $\ell (\sigma)>n_0$ is similar. Observe first that
$n_0=1$ if $h>1$. In this case, it is easier to flip a negative spin
surrounded by a positive spin than to flip a positive spin surrounded
by two negative spins and the proof presented below simplifies. We
assume that $h<1$ so that $n_0\ge 2$.

Recall the definition of the set $F_1(\sigma)$ introduced in the
beginning of the proof.  By the harmonic property of $f_\beta$,
\begin{equation*}
f_\beta(\sigma) \;=\; \frac 1{|F_1(\sigma)|} \sum_{\xi \in F_1(\sigma)}
f_\beta(\xi) \;+\; \frac{e^{-2\beta[1-h]}}{|F_1(\sigma)|^2}
\sum_{\eta \in G_1(\sigma)} \sum_{\xi \in F_1(\sigma)}
[f_\beta(\eta)-f_\beta(\xi)] \;+\; o(\beta)\;,
\end{equation*}
where $G_1(\sigma)$ is the set of configurations obtained from
$\sigma$ by flipping a negative spin surrounded by a positive spin and
where $o(\beta)$ an expression which vanishes faster than
$e^{-2\beta[1-h]}$ as $\beta\uparrow\infty$.

We claim that
\begin{equation}
\label{e03}
\lim_{\beta\to\infty} e^{2\beta[1-h]} \sum_{\xi \in F_1(\sigma)}
\{f_\beta(\xi) - f_\beta(\sigma)\} \;=\; 0\;.
\end{equation}
To prove this claim, denote by $F_k(\sigma)$, $1\le k\le n_0$, the
configurations obtained from $\sigma$ by successively flipping $k$
distinct positive spins surrounded by two negative spins:
$F_{j+1}(\sigma) = \cup_{\xi\in F_{j}(\sigma)} F_1(\xi)$. Denote by
$G_1(\eta)$ the predecessors of $\eta$, that is, the configurations
obtained from $\eta$ by flipping a negative spin surrounded by two
positive spins. Hence, $G_1(\eta) \subset F_{j-1}(\sigma)$ if $\eta$
belongs to $F_j(\sigma)$. By the harmonic property, for every $\eta\in
F_j(\sigma)$, $1\le j< n_0$,
\begin{equation*}
\sum_{\xi \in G_1(\eta)} \{f_\beta(\eta) - f_\beta(\xi)\} \;=\;
e^{-\beta h} \sum_{\zeta \in F_1(\eta)} 
\{f_\beta(\zeta) - f_\beta(\eta)\} \;+\;
O(e^{-\beta [2-h]})\;.
\end{equation*}
Replacing this identity in the sum appearing in \eqref{e03}, we reduce
the proof of \eqref{e03} to the proof that
\begin{equation*}
e^{2\beta[1-h]} e^{-(n_0-1)\beta h} 
\sum_{\xi_1 \in F_1(\sigma)}
\sum_{\xi_2 \in F_1(\xi_1)} \cdots \sum_{\xi_{n_0} \in F_1(\xi_{n_0-1})}
\{f_\beta(\xi_{n_0}) - f_\beta(\xi_{n_0-1})\} 
\end{equation*}
vanishes as $\beta\uparrow\infty$. This holds because $f_\beta$ is
bounded by one and $2/h < n_0+1$.

By the harmonic property of $f_\beta$ at $\xi\in F_1(\sigma)$, $f(\xi)
= f(\sigma)$ for any limit point $f$ of the sequence
$f_\beta$. Moreover, by \eqref{e03} and by the displayed formula
appearing just before \eqref{e03},
\begin{equation*}
\sum_{\eta \in G_1(\sigma)} [f (\eta)-f (\sigma)] \;=\; 0\;.
\end{equation*}
Recall the notation introduced in Section \ref{si1}.  Note that
$G_1(\sigma) = \bb W(\sigma)$ and that $f(\eta) = [j+f(\sigma)]/(j+1)$
if $\eta$ belongs to $\bb W_j(\sigma, \sigma')$, $1\le j\le 3$, while
$f(\eta) = f(\sigma)/(j+1)$ if $\eta\in \bb W_j(\sigma) \setminus \bb
W_j(\sigma, \sigma')$. This observation permits to conclude the proof
of the lemma.
\end{proof} 

Recall the definition of the sets $\Omega_{o,k}$, $1\le k\le n_0$, and
$\bb S(\sigma)$ introduced in Section \ref{si1}.

\begin{corollary}
\label{t05}
Fix a configuration $\sigma$ in $\Omega_{o,k} \setminus
\Omega_{o,k+1}$, $1\le k\le n_0$, $\sigma \not = + \mb 1, - \mb 1$.
Let $\Omega_{k, \sigma} = \Omega_{o,k} \setminus \{\sigma\}$.  
Then, for all $\sigma'\in \Omega_{k, \sigma}$,
\begin{equation*}
\lim_{\beta\to\infty} \mb P^\beta_\sigma \big[ T_{\sigma'} = 
T_{\Omega_{k, \sigma}} \big] \;=\; p(\sigma, \sigma')\;.
\end{equation*} 
\end{corollary}

\begin{proof}
Fix $1\le k\le n_0$ and a configuration $\sigma$ in $\Omega_{o,k}
\setminus \Omega_{o,k+1}$, $\sigma \not = + \mb 1, - \mb 1$.  For
$k=1$ and $k=n_0$, since $\bb D(\sigma) = \bb S(\sigma)$ and
$\sum_{\sigma'\in \bb D(\sigma)} q(\sigma, \sigma')=1$, by Lemma
\ref{t03},
\begin{equation*}
\lim_{\beta\to\infty} \mb P^\beta_\sigma \big[ T_{\bb S(\sigma)} = 
T_{\Omega_{\sigma}} \big] \;=\; 1\;.
\end{equation*}
Since $\bb S(\sigma)\subset \Omega_{k, \sigma} \subset
\Omega_{\sigma}$ we may replace $T_{\bb S(\sigma)}$ by $T_{\Omega_{k,
    \sigma}}$ in the previous equation. The corollary follows now from
Lemma \ref{t03} and the fact that $p=q$ for $k=1$ and $k=n_0$.

Consider now the case $2\le k<n_0$. Fix a configuration $\sigma'\in
\bb D(\sigma) \cap \Omega_{o,k} \subset \bb S(\sigma)$. Since
$T_{\Omega_{\sigma}} \le T_{\Omega_{k,\sigma}}$,  and since $p(\sigma,
\sigma^*) = q(\sigma, \sigma^*)$ for $\sigma^*\in \bb S(\sigma)$, by
Lemma \ref{t03},
\begin{equation*}
\liminf_{\beta\to\infty} \mb P^\beta_\sigma \big[ T_{\sigma'} = 
T_{\Omega_{k, \sigma}} \big] \;\ge\;  
\lim_{\beta\to\infty} \mb P^\beta_\sigma \big[ T_{\sigma'} = 
T_{\Omega_{\sigma}} \big] \;=\; p(\sigma, \sigma')\;.
\end{equation*}
Fix now a configuration $\sigma'\in \bb S(\sigma) \setminus \bb
D(\sigma)$. This configuration is obtained from $\sigma$ by flipping
all spins of a positive $\ell(\sigma) \times \ell(\sigma)$ square of
$\sigma$. Denote by $\sigma_j$, $1\le j\le 4$, the four configurations
obtained from $\sigma$ by flipping all spins from one of the sides of
this square. Of course,
\begin{equation*}
\mb P^\beta_\sigma \big[ T_{\sigma'} = T_{\Omega_{k, \sigma}} \big]
\;\ge\; \sum_{j=1}^4 \mb P^\beta_\sigma 
\big[ T_{\sigma'} = T_{\Omega_{k, \sigma}} \,,\, T_{\sigma_j} =
T_{\Omega_{\sigma}} \big]\;.
\end{equation*}
Since $T_{\Omega_{\sigma}} \le T_{\Omega_{k, \sigma}}$ and
$\sigma_j\not\in\Omega_{k, \sigma}$, by the strong Markov property,
the right hand side is equal to
\begin{equation*}
\sum_{j=1}^4 \mb P^\beta_\sigma 
\big[ T_{\sigma_j} = T_{\Omega_{\sigma}} \big] \, \mb P^\beta_{\sigma_j}
\big[ T_{\sigma'} = T_{\Omega_{k, \sigma}} \big] \;.
\end{equation*}
By Lemma \ref{t03}, $\mb P^\beta_\sigma \big[ T_{\sigma_j} =
T_{\Omega_{\sigma}} \big]$ converges to $q(\sigma, \sigma_j)$ as
$\beta\uparrow\infty$. We also claim that $\mb P^\beta_{\sigma_j}
\big[ T_{\sigma'} = T_{\Omega_{k, \sigma}} \big]$ converges to $1$ as
$\beta\uparrow\infty$ for $1\le j\le 4$. Indeed, for a fixed $j$,
$\ell(\sigma_j) = \ell(\sigma)-1$ and the configuration $\sigma_j$ has
one and only one positive rectangle $R$ with a side of length
$\ell(\sigma)-1$. It follows from Lemma \ref{t03} and from the
definition of the sets $\bb D(\sigma^*)$ that the process first flips
the spins of one side of the rectangle $R$ transforming it into a
positive $[\ell(\sigma)-1]\times [\ell(\sigma)-1]$ square. Then, it
flips the spins of one side of this square transforming it into a
positive $[\ell(\sigma)-2]\times [\ell(\sigma)-1]$ rectangle and so
on, until the process reaches a configuration where the initial
rectangle $R$ is transformed into a $2\times 2$ square, without
flipping in this process any other site which is not contained in the
original $\ell(\sigma) \times \ell(\sigma)$ positive square of
$\sigma$. In the last step, all spins of the $2\times 2$ positive
square are flipped and the process reaches the configuration $\sigma'$
which belongs to $\Omega_{k, \sigma}$ and is the first one to belong
to this set in the evolution just described. This proves the claim.

It follows from this argument that
\begin{equation*}
\liminf_{\beta\to\infty} \mb P^\beta_\sigma \big[ T_{\sigma'} 
= T_{\Omega_{k, \sigma}} \big] \;\ge\; \sum_{j=1}^4 q(\sigma,
\sigma_j) \; =\; p(\sigma, \sigma')\;.
\end{equation*}
Since this inequality holds for all $\sigma'\in \bb S(\sigma)$ and
$\sum_{\sigma'\in \bb S(\sigma)} p(\sigma, \sigma')=1$, the lemma is
proved.
\end{proof}

The proof of Lemma \ref{t03} describes the asymptotic behavior of $\mb
P^\beta_\eta [ T_\sigma < T_{\Omega_\sigma}]$ for some configurations
$\eta$, but not for all. We may not, therefore, apply blindly Lemma
\ref{s21} to deduce the limit of the capacity $\Cap_\beta(\{\sigma\} ,
\Omega_\sigma)$. Next result fills the gaps.

For $3\le \ell (\sigma) \le n_0$, denote by $\bb W_1(\sigma)$ the
configurations in $\bb W(\sigma)$ whose positive spin surrounded by
three negative spins is in the corner of a positive rectangle of
$\sigma$ and denote by $\bb W_2(\sigma)$ the remaining configurations
of $\bb W (\sigma)$. Note that configurations in $\bb W_j(\sigma)$
jump to $\Omega_\sigma$ with probability $(j+1)^{-1} + o(\beta)$ and
that $|\bb W_1(\sigma)| = 4 N_r(\sigma) + 8 N_s(\sigma)$, $|\bb
W_2(\sigma)| = 2[\ell(\sigma)-2] N_r(\sigma) + 4[\ell(\sigma)-2]
N_s(\sigma)$. 

\begin{lemma}
\label{t02}
Fix a configuration $\sigma$ in $\Omega_o$, $\sigma \not = + \mb 1, -
\mb 1$. If $2\le \ell:=\ell(\sigma) \le n_0$,
\begin{equation*}
\lim_{\beta\to\infty}  e^{\beta [\ell -1] h} \, \mu_\beta
(\sigma)^{-1} \, \Cap_\beta(\{\sigma\} , \Omega_\sigma) \;=\; 
\theta(\sigma)\;, 
\end{equation*}
and if $\ell > n_0$,
\begin{equation*}
\lim_{\beta\to\infty} e^{\beta (2-h)} \,
\mu_\beta (\sigma)^{-1} \, \Cap_\beta(\{\sigma\} , \Omega_\sigma) 
\;=\; \theta(\sigma)\;, 
\end{equation*}
where $\theta(\sigma)$ has been defined in \eqref{e07}.
\end{lemma}

\begin{proof}
Fix a configuration $\sigma$ satisfying the assumptions of the lemma
and assume that $3\le \ell :=\ell(\sigma) \le n_0$. By Lemmas
\ref{s03} and \ref{t01}, we know that $\Cap_\beta(\{\sigma\} ,
\Omega_\sigma)$ is of order $\mu_\beta (\sigma) e^{- \beta (\ell -1)
  h}$.

We start with the proof of the upper bound for the capacity.  Recall
that we denote by $\bb W(\sigma)$ the set of saddle configurations of
the valley $(\{\sigma\}, \{\sigma\} \cup \Delta, \sigma)$.  Denote by
$B$ the set of all configurations $\eta$ which do not belong to $\bb
W(\sigma)$ and which can be reached from $\sigma$ by self-avoiding
paths $\gamma = (\sigma = \eta_0, \eta_1, \dots, \eta_p = \eta)$ such
that $\mu_\beta (\eta_k) \ge \mu_\beta (\sigma) e^{- \beta [\ell -1]
  h}$, $0\le k\le p$. It follows from the proof of Lemma \ref{t01}
that all these configurations are obtained from $\sigma$ by
successively flipping at most $\ell-1$ positive spins which are
surrounded by two negative spins. Note that all neighbors $\xi$ of a
configuration $\eta \in B$ which do not belong to $B$ have measure
$\mu_\beta (\xi) < \mu_\beta (\sigma) e^{- \beta [\ell -1] h}$.

Consider the function $f:\Omega \to [0,1]$ defined as follows.  Set
$f(\sigma)=1$, $f=1$ on $B$, $f=j/(j+1)$ on $\bb W_j(\sigma)$ and
$f=0$ elsewhere. By definition of capacity and by definition of the
function $f$, $\Cap_\beta (\{\sigma\} , \Omega_\sigma) \le D_\beta(f)
= \mu_\beta (\sigma) e^{- \beta [\ell -1] h} \{(2/3) |\bb W_2(\sigma)|
+ (1/2) |\bb W_1(\sigma)|\} + o(\beta)$, where $o(\beta) \prec
\mu_\beta (\sigma) e^{- \beta [\ell -1] h}$.  This proves the upper
bound.

To prove the lower bound, consider a function $f$ equal to $1$ at
$\sigma$ and $0$ on $\Omega_\sigma$. Denote by $A_0$ the set of all
configurations $\eta$ which can be reached from $\sigma$ by
self-avoiding paths $\gamma = (\sigma = \eta_0, \eta_1, \dots, \eta_p
= \eta)$ such that $\mu_\beta (\eta_k) \ge \mu_\beta (\sigma) e^{-
  \beta [\ell -1] h}$, $0\le k\le p$, and let $A=A_0 \cup \bb
D(\sigma)$. By definition of the Dirichlet form,
\begin{equation*}
D_\beta(f) \;\ge\; \sum_{\{\eta, \xi\}\subset A} G_\beta(\eta,\xi)
[f(\xi)-f(\eta)]^2 \;.
\end{equation*}
Denote by $f_\beta : A\to [0,1]$ the function which minimizes the
right hand side with the boundary conditions imposed above. It is well
known that $f_\beta(\eta) = \mb P^{A,\beta}_\eta [T_\sigma < T_{\bb
  D(\sigma)}]$ where $\mb P^{A,\beta}$ stands for the probability on
the path space induced by the reversible Markov process whose
Dirichlet form is the one appearing on the right hand side of the
previous formula. The asymptotic behavior of $f_\beta (\eta)$, as
$\beta\uparrow\infty$, has been examined in the previous lemma for
certain configurations. The arguments presented in the proof of the
lower bound of Lemma \ref{s21} permit to conclude.

The proofs for $\ell(\sigma)=2\le n_0$ and $\ell(\sigma)> n_0$ are
simpler and left to the reader.
\end{proof}

Recall the definition of the set $\Omega_{k,\sigma}$ introduced in
Corollary \ref{t05} and fix $2\le k\le n_0$. Since $\Omega_{k,\sigma}
\subset \Omega_{\sigma}$, $\Cap_\beta(\{\sigma\} , \Omega_{k,\sigma})
\le \Cap_\beta(\{\sigma\} , \Omega_{\sigma})$. The method of the proof
of the lower bound for $\Cap_\beta(\{\sigma\} , \Omega_{\sigma})$
together with the asymptotic behavior of the hitting times stated in
Corollary \ref{t05} provide the next result.

\begin{corollary}
\label{t06}
Fix $2\le k<n_0$ and a configuration $\sigma$ in
$\Omega_{k,\sigma}\setminus \Omega_{k+1,\sigma}$. Then,
\begin{equation*}
\lim_{\beta\to\infty}  e^{\beta k h} \, \mu_\beta
(\sigma)^{-1} \, \Cap_\beta(\{\sigma\} , \Omega_{k,\sigma}) \;=\; 
\theta(\sigma)\;.
\end{equation*}
Moreover, for $\sigma$ in $\Omega_{n_0,\sigma}\setminus
\Omega_{n_0+1,\sigma}$,
\begin{equation*}
\lim_{\beta\to\infty}  e^{\beta (2-h)} \, \mu_\beta
(\sigma)^{-1} \, \Cap_\beta(\{\sigma\} , \Omega_{n_0,\sigma}) \;=\; 
\theta(\sigma)\;.
\end{equation*}
\end{corollary}

\section {Proofs of Theorem \ref{t04} and Theorem \ref{t08}}

The proof of Theorem \ref{t04} is based on the theory developed in
the first sections of this article.  A simple computation shows that
$\bb H(\sigma^{x}) - \bb H(\sigma) = \sum_{y: |y-x|=1} \sigma(y)\,
\sigma(x) + h \, \sigma(x)$, where $|\,\cdot\,|$ stands for the
Euclidean norm. Since $0<h<2$, the jump rates $c (x,\sigma)$ may only
assume the values $1$, $e^{-\beta [4+h]}$, $e^{-\beta [2+h]}$,
$e^{-\beta h}$, $e^{-\beta [4-h]}$ and $e^{-\beta [2-h]}$. Assumptions
\eqref{24}, \eqref{30} are therefore satisfied.

Recall the terminology and the notation introduced in Section
\ref{ssec5}. According to the theory developed in the previous
sections, the first step in the proof of the metastable behavior of a
Markov process is the description of the evolution among the
shallowest valleys which we now determine.  Since a negative
(resp. positive) spin surrounded by two (resp. three) positive
(resp. negative) spins flips at rate one, it is not difficult to show
that the leaves $\mc E_1, \dots, \mc E_\nu$ defined in Section
\ref{ssec5} are all the singletons formed by the elements of
$\Omega_o$ so that $\nu = |\Omega_o|$ and $\Delta =
\Omega\setminus\Omega_o$.

Denote by $\mc E_\sigma$ the singleton $\{\sigma\}$, $\sigma \in
\Omega_o$. By Lemma \ref{s06}, Proposition \ref{s30} and Lemma
\ref{t02}, $(\{\sigma\}, \{\sigma\}\cup \Delta, \sigma)$, $\sigma \not
= -\mb 1$, $+\mb 1$, is a valley of depth $e^{\beta [\ell(\sigma)-1]h}
\theta(\sigma)^{-1}$ if $2\le \ell(\sigma)\le n_0$ and of depth
$e^{\beta (2-h)} \theta(\sigma)^{-1}$ if $\ell(\sigma)>
n_0$. Moreover, by Lemma \ref{t01}, Lemma \ref{s03} and the same
results invoked above, $(\{\pm \mb 1\}, \{\pm \mb 1\}\cup \Delta,
\pm\mb 1)$ is a valley whose depth is of order $e^{\beta (8\pm
  3h)}$. The exact depth of these latter valleys is not important at
this stage.

To describe the evolution among the shallowest valleys, recall the
notation introduced in Section \ref{ssec2}. For a subset $F$ of
$\Omega$, denote by $R^F_\beta (\sigma, \sigma')$, $\sigma$,
$\sigma'\in F$, the jump rates of the trace $\sigma^F_t$ of the
process $\sigma_t$ on $F$. Let $r^{F}_\beta(A,B)$, $A$, $B\subset F$,
$A\cap B=\varnothing$, be the average jump rates of $\sigma^{F}_t$
from $A$ to $B$:
\begin{equation*}
r^{F}_\beta(A,B) \;=\; \frac{1}{\mu_\beta(A)}  \sum_{\sigma\in A} 
\mu_\beta(\sigma) \sum_{\sigma'\in B} R^{F}_\beta(\sigma, \sigma')\;.
\end{equation*}
In view of the depths of the valleys $(\{\sigma\}, \{\sigma\}\cup
\Delta, \sigma)$, $\sigma \in \Omega_o$, the set $S_1$ can be
identified to the set $\Omega_{o,1} \setminus \Omega_{o,2}$. Recall
that $\theta_\beta(1) = e^{\beta h}$ if $n_0\ge 2$ and
$\theta_\beta(1) = e^{\beta (2-h)}$ if $n_0= 1$. By Lemma \ref{s25},
Corollary \ref{t05}, the explicit expression for the depth of the
valleys obtained above, and Lemma \ref{s39}, the scaled average rates
$e^{\beta h} \, r_\beta^{\Omega_o} (\sigma,\sigma')$, $\sigma$,
$\sigma'\in \Omega_o$, converge to $r(\sigma, \sigma') =
\theta(\sigma) p(\sigma, \sigma')$, where $p(\sigma, \sigma')$ and
$\theta(\sigma)$ have been introduced in \eqref{e04}--\eqref{e07}.

Recall that we denote by $\sigma^{\beta,1}_t$ the trace of the Markov
process $\sigma^{\beta}_t$ on $\Omega_{o,1}$. By Lemma \ref{s26} with
$\theta_\beta(1) = e^{\beta h}$ and by the observations of the
previous paragraph, the speeded-up process $\sigma^{\beta,1}_{t
  \theta_\beta(1)}$ converges to a Markov process on $\Omega_{o,1}$
with jump rates $r(\sigma, \sigma') = \theta(\sigma) p(\sigma,
\sigma')$. By Proposition \ref{s07} on the time scale
$\theta_\beta(1)$ the time spent in $\Delta$ is negligible.  This
proves Theorem \ref{t04} for $k=1$.

The proof of Theorem \ref{t04} in the longer time scales is based on
Theorem \ref{s37} and follows the strategy presented in Remark
\ref{s38}. Recall the notation introduced in Section \ref{ssec3} and
Assumption {\bf T}. Since Theorem \ref{t04} has been proven for $k=1$,
Assumption {\bf T} holds at level one because all $1$-metastates are
singletons.

Theorem \ref{t04} for $2\le k\le n_0$ follows from Theorem
\ref{s37}. As explained in Remark \ref{s38}, we just need to to
characterize the metastates at each level, the depth of each valley
and the asymptotic rates. This has been done for $2\le k\le n_0$ in
Corollary \ref{t05} and Corollary \ref{t06}, in view of Lemma
\ref{s39}. We present in details the case $k=2$ and leave the rest of
the recursive argument to the reader.

Assume that $n_0\ge 2$. It follows from the dynamics generated by the
rates $r$ introduced above that the leaves at level $2$, $\mc
E^{(2)}_1, \dots, \mc E^{(2)}_{\nu(2)}$, are all the singletons formed
by the elements of $\Omega_{o,2}$ so that $\nu(2) = |\Omega_{o,2}|$
and $\Delta_2 = \Omega\setminus\Omega_{o,2}$, $\Delta^o_2 =
\Omega_{o,1} \setminus \Omega_{o,2}$.

By Theorem \ref{s37} with $\mf p=1$ and Corollary \ref{t06}, the
triples $(\{\sigma\}, \{\sigma\}\cup \Delta^o_2, \sigma)$, $\sigma \in
\Omega_{o,2}$, $\sigma \not = -\mb 1$, $+\mb 1$, are valleys for the
trace process $\sigma^{\beta, 1}_t$ of depth $e^{\beta
  [\ell(\sigma)-1]h} \theta(\sigma)^{-1}$ if $3\le \ell(\sigma) \le
n_0$ and of depth $e^{\beta (2-h)} \theta(\sigma)^{-1}$ if
$\ell(\sigma)> n_0$. Moreover, by Lemma \ref{t01} and Lemma \ref{s03},
$(\{\pm \mb 1\} \{\pm \mb 1\}\cup \Delta^o_2, \pm\mb 1)$ is a valley
for the trace process $\sigma^{\beta, 1}_t$ whose depth is of order
$e^{\beta (8\pm 3h)}$.

Note that the Ising model presents the particularity that the $\mf
p$-metastates are $1$-metastates, and not a union of $1$-metastates.

Recall the definition of the set $S_2$ introduced just after Lemma
\ref{s28}. The set $S_2$ can be identified to the set $\Omega_{o,2}
\setminus \Omega_{o,3}$. Set $\theta_\beta(2) = e^{2 \beta h}$ if
$n_0>2$ and $\theta_\beta(2) = e^{\beta(2-h)}$ if $n_0=2$. By Theorem
\ref{s37} with $\mf p=1$, Lemma \ref{s39}, Corollary \ref{t05} and
Corollary \ref{t06}, $\sigma^{\beta,2}_{t \theta_\beta(2)}$ converges
to a Markov process on $\Omega_{o,2}$ with jump rates $r(\sigma,
\sigma') = \theta(\sigma) p(\sigma, \sigma')$ introduced in
\eqref{e04}--\eqref{e07}. Furthermore, by Theorem \ref{s37} with $\mf
p=1$, on the time scale $\theta_\beta(2)$ the time spent in $\Delta_2$
is negligible.  This proves Theorem \ref{t04} for $k=2$. \qed
\smallskip

We now turn to the proof of Theorem \ref{t08}. It relies on the
following lemma.  Recall the definition of the sets $\bb W(-\mb 1)$,
$\bb W_1(-\mb 1)$ and $\bb W_2(-\mb 1)$ and of the number $\theta(-\mb
1)$ introduced just before the statement of Theorem \ref{t08}.

\begin{lemma}
\label{t07}
For $\beta>0$, $G_\beta (\{-\mb 1\}, \{+\mb 1\}) = \mu_\beta
(\sigma^*)$, for any $\sigma^* \in \bb W(-\mb 1)$.
Moreover,
\begin{equation*}
\lim_{\beta\to\infty}  e^{\beta \, c(h)} \, \mu_\beta
(-\mb 1)^{-1} \, \Cap_\beta(\{- \mb 1\} , \{+\mb 1\}) \;=\; 
\theta(-\mb 1)\;, 
\end{equation*}
where $c(h) = 4 (n_0+1) - h [(n_0+1)n_0+1]$ and $\theta(-\mb 1) =
(2/3) |\bb W_2(-\mb 1)| + (1/2) |\bb W_1(-\mb 1)|$.
\end{lemma}

\begin{proof}
The proof of the first assertion is left to the reader. The proof of
the second one is similar to the one of Lemma \ref{s21}.

Denote by $B$ the set of all configurations $\eta$ which do not
belong to $\bb W(-\mb 1)$ and which can be reached from $-\mb 1$ by
self-avoiding paths $\gamma = (-\mb 1 = \eta_0, \eta_1, \dots, \eta_p
= \eta)$ such that $\mu_\beta (\eta_k) \ge \mu_\beta (\sigma^*)$,
$0\le k\le p$, for some $\sigma^* \in \bb W(-\mb 1)$. All these
configurations are obtained from $-\mb 1$ by flipping at most
$n_0(n_0+1)$ negative spins. Note that all neighbors $\xi$ of a
configuration $\eta \in B$ which do not belong to $B$ have measure
$\mu_\beta (\xi) < \mu_\beta (\sigma^*)$.

Consider the function $f:\Omega \to [0,1]$ defined as follows.  Set
$f(-\mb 1)=1$, $f=1$ on $B$, $f=1/(j+1)$ on $\bb W_j(\sigma)$ and
$f=0$ elsewhere. By definition of capacity and by definition of the
function $f$, $\Cap_\beta (\{-\mb 1\} , \{+\mb 1\}) \le D_\beta(f) =
\mu_\beta (\sigma^*) \{(2/3) |\bb W_2(-\mb 1)| + (1/2) |\bb W_1(-\mb
1)|\} + o(\beta)$, where $o(\beta) \prec \mu_\beta (\sigma^*)$.  This
proves the upper bound.

To prove the lower bound, recall that the function $f:\Omega \to \bb
R$ which minimizes the Dirichlet form under the constraint that
$f(-\mb 1)=1$, $f(+\mb 1)=0$ is the hitting time $g_\beta(\sigma) =
\mb P^\beta_\sigma [T_{-\mb 1} < T_{+\mb 1}]$.

Denote by $A$ the set of all neighbors $\xi$ of $\bb W(-\mb 1)$ which
are obtained from a configuration $\eta \in \bb W(-\mb 1)$ by either
flipping the positive spin surrounded by three negative spins or by
flipping a negative spin surrounded by two positive spins. By
definition of the Dirichlet form,
\begin{equation*}
\Cap_\beta(\{- \mb 1\} , \{+\mb 1\}) \;=\; D_\beta(g_\beta) \;\ge\; 
\sum_{\eta \in \bb W(-\mb 1), \xi \in A} G_\beta(\eta,\xi)
[g_\beta (\xi)- g_\beta (\eta)]^2 \;. 
\end{equation*}
It follows from Corollary \ref{t05} that $g_\beta (\xi)$ converges to
$0$ (resp. $1$) as $\beta\uparrow\infty$ if $\xi$ is a configuration
obtained from a configuration in $\bb W(-\mb 1)$ by flipping a
negative spin surrounded by two positive spins (resp. by flipping the
positive spin surrounded by three negative spins). On the other hand,
since $g_\beta$ is harmonic and since a configuration $\eta\in \bb
W_j(-\mb 1)$ jumps to configurations in $A$, where the asymptotic
behavior of $g_\beta$ is known, at rates of order one, and jumps to
other configurations at rate $o(\beta)$, $g_\beta (\eta)$ converges,
as $\beta\uparrow\infty$, to $1/(j+1)$ if $\eta\in W_j(-\mb 1)$. This
proves the lower bound since $G_\beta(\eta,\xi) = \mu_\beta
(\sigma^*)$ for $\eta \in \bb W(-\mb 1)$, $\xi \in A$.
\end{proof}

We are now in a position to prove Theorem \ref{t08} which relies on
Theorem \ref{s37} and the strategy presented in Remark \ref{s38}. Up
to this point we proved Assumption {\bf T} at level $n_0$. In view of
the asymptotic dynamics of the trace process $\sigma^{\beta, n_0}_t$
described in Theorem \ref{t04}, there are only two
$(n_0+1)$-metastates, $\{- \mb 1\}$ and $\{+ \mb 1\}$.  By Theorem
\ref{s37} and by Lemma \ref{t07}, $(\{- \mb 1\} , \{- \mb 1\} \cup
\Delta^o_{n_0+1}, - \mb 1)$ is a valley for the trace process
$\sigma^{\beta, n_0}_t$ of depth $e^{\beta c(h)} \theta(-\mb
1)^{-1}$. A similar computation to the one presented in Lemma
\ref{t07} shows that $(\{+ \mb 1\} , \{+ \mb 1\}\cup \Delta^o_{n_0+1},
+ \mb 1)$ is a valley for the trace process $\sigma^{\beta, n_0}_t$
whose depth is of magnitude larger than the one of the valley $(\{-
\mb 1\} , \{- \mb 1\} \cup \Delta^o_{n_0+1}, - \mb 1)$. Recall that
$\theta_\beta(n_0+1) = e^{\beta c(h)}$ and note that we may identify
the set $S_{n_0+1}$ with the singleton $\{- \mb 1\}$.

Since the state space of the trace process $\sigma^{\beta, n_0+1}_{t}$
is a pair, by Theorem \ref{s37}, by the explicit computation of the
depth of the valley $(\{- \mb 1\} , \{- \mb 1\} \cup \Delta^o_{n_0+1},
- \mb 1)$ and by Lemma \ref{s39}, the speeded-up trace process
$\sigma^{\beta, n_0+1}_{t \theta_\beta(n_0+1)}$ converges to the
Markov process on $\{-\mb 1, +\mb 1\}$ in which $+\mb 1$ is an
absorbing state and which jumps from $-\mb 1$ to $+ \mb 1$ at rate
$\theta(-\mb 1)$. The second assertion of Theorem \ref{t08} also
follows from Theorem \ref{s37}. \qed

\section{General results}
\label{sec2}

We state in this section some general results on metastability of
continuous time Markov chains used in the previous sections. We assume
that the reader is familiar with the notation and terminology of
\cite{bl2}.

Fix a sequence $(E_N: N\ge 1)$ of countable state spaces. The elements
of $E_N$ are denoted by the Greek letters $\eta$, $\xi$. For each
$N\ge 1$ consider a matrix $R_N : E_N \times E_N \to \bb R$ such that
$R_N(\eta, \xi) \ge 0$ for $\eta \not = \xi$, $-\infty < R_N (\eta,
\eta)\le 0$ and $\sum_{\xi\in E_N} R_N(\eta,\xi)=0$ for all $\eta\in
E_N$.  

Let $\{\eta^N_t : t\ge 0\}$ be the {\sl minimal} right-continuous
Markov process associated to the jump rates $R_N(\eta,\xi)$ \cite{n}.
It is well known that $\{\eta^N_t : t\ge 0\}$ is a strong Markov
process with respect to the filtration $\{\mc F^N_t : t\ge 0\}$ given
by $\mc F^N_t = \sigma (\eta^N_s : s\le t)$. Let $\mb P_{\eta}$,
$\eta\in E_N$, be the probability measure on $D(\bb R_+,E_N)$ induced
by the Markov process $\{\eta^N_t : t\ge 0\}$ starting from $\eta$.

Consider two sequences $\ms W = (W_N\subseteq E_N : N\ge 1)$, $\ms B =
(B_N \subseteq E_N : N\ge 1)$ of subsets of $E_N$, the second one
containing the first and being properly contained in $E_N$: $W_N
\subseteq B_N \subsetneqq E_N$.  Fix a point $\bs \xi = (\xi_N \in W_N
: N\ge 1)$ in $\ms W$ and a sequence of positive numbers $\bs \theta =
(\theta_N : N\ge 1)$.

\renewcommand{\theenumi}{\arabic{enumi}}
\renewcommand{\labelenumi}{(\theenumi)}

\begin{lemma}
\label{s04}
Assume that the triple $(\ms W, \ms B, \bs \xi)$ is a valley of depth
$\bs \theta$ and attractor $\bs \xi$. Let $\ms C = (C_N \subset E_N :
N\ge 1)$ be a sequence of sets such that
\begin{enumerate}
\item $C_N \cap B_N = \varnothing$,

\item For every $\delta>0$,
\begin{equation}
\label{03}
\lim_{N\to\infty}
\sup_{\eta\in C_N} \mb P_{\eta} \Big[ \frac 1{\theta_N} T_{(\ms B \cup
  \ms C)^c} > \delta \Big] \;=\; 0\; .
\end{equation}
\end{enumerate}
Then, the triple $(\ms W, \ms B \cup \ms C, \bs \xi)$ is a
valley of depth $\bs \theta$ and attractor $\bs \xi$.
\end{lemma}

\begin{proof}
We need to check the three conditions of \cite[Definition 2.1]{bl2}. 
Since $(B_N \cup C_N)^c \subset B_N^c$, condition ({\bf V1}) is
clearly fulfilled.

To prove ({\bf V3}), decompose the event $T_{(\ms B \cup \ms C)^c}
(\bs \Delta \cup \ms C) > \delta \theta_N$ according to whether $T_{\ms
  C} < T_{(\ms B \cup \ms C)^{c}}$ or $T_{\ms C} > T_{(\ms B \cup \ms
  C)^{c}}$. In the latter case, $T_{(\ms B \cup \ms C)^c} (\bs \Delta
\cup \ms C) = T_{\ms B^c} (\bs \Delta)$ so that for every point $\bs
\eta = (\eta^N : N\ge 1)$ in $\ms W$,
\begin{equation*}
\begin{split}
& \limsup_{N\to\infty} \mb P_{\eta^N} \Big[ \frac 1{\theta_N} 
T_{(\ms B \cup \ms C)^c} (\bs \Delta \cup \ms C) > \delta 
\, , \, T_{\ms C} > T_{(\ms B \cup \ms C)^{c}}  \Big] \\
&\quad \le\; \lim_{N\to\infty} \mb P_{\eta^N} \Big[ \frac 1{\theta_N} 
T_{\ms B^c} (\bs \Delta) > \delta \Big] \;=\; 0  \; , 
\end{split}
\end{equation*}
where the last identity follows from the fact that the triple $(\ms W,
\ms B, \bs \xi)$ is a valley and from condition ({\bf V3}) in the
definition of a valley. On the other hand, since on the set $T_{\ms C}
< T_{(\ms B \cup \ms C)^{c}}$,
\begin{equation*}
\begin{split}
T_{(\ms B \cup \ms C)^c} (\bs \Delta \cup \ms C) 
\; & =\; \int_0^{T_{\ms C}} \mb 1\{ \eta^N_s \in \Delta_N \cup
C_N\} \, ds \;+\; \int_{T_{\ms C}} ^{T_{(\ms B \cup \ms C)^c}} 
\mb 1\{ \eta^N_s \in \Delta_N \cup C_N\} \, ds \\
& =\; \int_0^{T_{\ms B^c}} \mb 1\{ \eta^N_s \in \Delta_N\} 
\, ds \;+\; \int_{T_{\ms C}} ^{T_{(\ms B \cup \ms C)^c}} 
\mb 1\{ \eta^N_s \in \Delta_N \cup C_N\} \, ds\;,
\end{split}
\end{equation*}
by the strong Markov property,
\begin{equation*}
\begin{split}
& \mb P_{\eta^N} \Big[ \frac 1{\theta_N} 
T_{(\ms B \cup \ms C)^c} (\bs \Delta \cup \ms C) > \delta 
\, , \, T_{\ms C} < T_{(\ms B \cup \ms C)^{c}}  \Big]  
\;\le\; \mb P_{\eta^N} \Big[ \frac 1{\theta_N} 
T_{\ms B^c} (\bs \Delta) > \delta/2 \Big] \\
&\quad  
\;+\; \sup_{\eta\in C_N} \mb P_{\eta} \Big[ \frac 1{\theta_N} 
\int_{0} ^{T_{(\ms B \cup \ms C)^c}} 
\mb 1\{ \eta^N_s \in \Delta_N \cup C_N\} \, ds > \delta/2 \Big]\; .
\end{split}
\end{equation*}
The right hand side of this inequality vanishes as $N\uparrow\infty$
by hypothesis \eqref{03} and by the fact that the triple $(\ms W, \ms
B, \bs \xi)$ is a valley.

Putting together the two previous estimates, we obtain that for every
$\delta>0$ and every point $\bs \eta = (\eta^N : N\ge 1)$ in $\ms W$,
\begin{equation*}
\lim_{N\to\infty} \mb P_{\eta^N} \Big[ \frac 1{\theta_N} 
T_{(\ms B \cup \ms C)^c} (\bs \Delta \cup \ms C) > \delta \Big]\;=\;
0\; .
\end{equation*}
This shows that the triple $(\ms W, \ms B \cup \ms C, \bs \xi)$
satisfies assumption ({\bf V3}) of a valley with depth $\theta_N$.

It remains to check that the assumption ({\bf V2}) of a valley is
fulfilled. On the one hand, since $T_{(\ms B \cup \ms C)^c} \ge T_{\ms
  B^c}$ and since the triple $(\ms W, \ms B , \bs \xi)$ is a valley of
depth $\theta_N$, for every $t>0$ and every point $\bs \eta = (\eta^N
: N\ge 1)$ in $\ms W$,
\begin{equation}
\label{04}
\liminf_{N\to\infty} \mb P_{\eta^N} \Big[ \frac 1{\theta_N} 
T_{(\ms B \cup \ms C)^c} > t \Big] \;\ge\;
\lim_{N\to\infty} \mb P_{\eta^N} \Big[ \frac 1{\theta_N} 
T_{\ms B^c} > t \Big] \;=\; e^{-t}\; .
\end{equation}

On the other hand, decompose the set $\{ T_{(\ms B \cup \ms C)^c} > t
\theta_N\}$ according to the partition $T_{\ms C} < T_{(\ms B \cup \ms
  C)^{c}}$, $T_{\ms C} > T_{(\ms B \cup \ms C)^{c}}$. In the latter
set, $T_{(\ms B \cup \ms C)^c} = T_{\ms B^c}$, while in the first one,
$T_{(\ms B \cup \ms C)^c} = T_{\ms B^c} + T_{(\ms B \cup \ms C)^{c}}
\circ T_{\ms C}$.  Therefore, for every $t>0$ and every point $\bs
\eta = (\eta^N : N\ge 1)$ in $\ms W$,
\begin{equation*}
\begin{split}
\mb P_{\eta^N} \Big[ \frac 1{\theta_N} T_{(\ms B \cup \ms C)^c} > t
\Big] \; &=\; \mb P_{\eta^N} \Big[ \frac 1{\theta_N} T_{\ms B^c} > t 
\,,\,  T_{\ms C} > T_{(\ms B \cup \ms C)^{c}} \Big] \\
& +\; \mb P_{\eta^N} \Big[ 
T_{\ms B^c} + T_{(\ms B \cup \ms C)^{c}} \circ T_{\ms C}  > t \theta_N
\,,\,  T_{\ms C} < T_{(\ms B \cup \ms C)^{c}} \Big]\;. 
\end{split}
\end{equation*}
By the strong Markov property, the second term on the right hand side
is bounded above by
\begin{equation*}
\sup_{\eta\in C_N} \mb P_{\eta} \Big[ T_{(\ms B \cup \ms C)^{c}}  
> \delta \theta_N \Big] \;+\; 
\mb P_{\eta^N} \Big[ T_{\ms B^c}   > (t-\delta) \theta_N
\,,\,  T_{\ms C} < T_{(\ms B \cup \ms C)^{c}} \Big]  
\end{equation*}
for every $\delta>0$. Therefore, in view of the two previous displayed
formulas, for every $\delta>0$,
\begin{equation*}
\mb P_{\eta^N} \Big[ \frac 1{\theta_N} T_{(\ms B \cup \ms C)^c} > t
\Big] \; \le\; \mb P_{\eta^N} \Big[ \frac 1{\theta_N} T_{\ms B^c} >
t -\delta \Big]
\; +\; \sup_{\eta\in C_N} \mb P_{\eta} \Big[ T_{(\ms B \cup \ms C)^{c}}  
> \delta \theta_N \Big]\;.
\end{equation*}
By \eqref{03}, the second term on the right hand side vanishes as
$N\uparrow\infty$ for every $\delta>0$. Since the triple $(\ms W, \ms
B , \bs \xi)$ is a valley of depth $\theta_N$, by condition ({\bf V2})
of a valley, the first term converges to $e^{-(t-\delta)}$ as
$N\uparrow\infty$. Hence, letting $\delta\downarrow 0$ after
$N\uparrow\infty$, we obtain that for every $t>0$ and every point $\bs
\eta = (\eta^N : N\ge 1)$ in $\ms W$,
\begin{equation*}
\limsup_{N\to\infty} \mb P_{\eta^N} \Big[ \frac 1{\theta_N} 
T_{(\ms B \cup \ms C)^c} > t \Big] \;\le\; e^{-t}\;.
\end{equation*}
This estimate together with \eqref{04} shows that the triple $(\ms W,
\ms B \cup \ms C, \bs \xi)$ satisfies condition ({\bf V2}) of a
valley with depth $\theta_N$.
\end{proof}

Of course, this result is only interesting if the process may jump
from $B_N$ to $C_N$.

\subsection{The positive recurrent reversible case}

We assume from now on that the Markov process $\{\eta^N_t : t\ge 0\}$
is positive recurrent and reversible with respect to its unique
invariant probability measure denoted by $\mu_N$.

Fix $N\ge 1$ and a proper subset $F_N$ of $E_N$. Denote by
$R^{F_N}(\eta, \xi)$ the jump rates of the trace of the process
$\{\eta^N_t : t\ge 0\}$ on the set $F_N$. We refer to \cite[Section
6.1]{bl2} for a precise definition. For each pair $A$, $B$ of disjoint
subsets of $F_N$, denote by $r_{F_N}(A,B)$ the average rate at which
the trace process on $F_N$ jumps from $A$ to $B$:
\begin{equation*}
r_{F_N}(A,B) \;=\; \frac 1{\mu_N(A)} \sum_{\eta\in A} \mu_N(\eta)
\sum_{\xi \in B} R^{F_N}(\eta, \xi)\;.
\end{equation*}

We claim that
\begin{equation}
\label{18}
r_{F_N}(A,B)\; \le\; \frac{\Cap_N(A,B)}{\mu_N(A)}\; ,
\end{equation}
where $\Cap_N(A,B)$ stands for the capacity between $A$ and $B$ for
the process $\{\eta^N_t : t\ge 0\}$. Indeed, denote by $R^{A\cup B}$
the jump rates of the trace of $\{\eta^N_t : t\ge 0\}$ on $A\cup
B$. By \cite[Corollary 6.2]{bl2}, $R^{A\cup B} (\eta,\xi) \ge
R^{F_N}(\eta, \xi)$ for every $\eta$, $\xi\in A\cup B$, $\eta\not =
\xi$. Hence, by definition of the average rates and by \cite[Lemma
6.7]{bl2},
\begin{equation*}
r_{F_N}(A,B) \;=\; \frac 1{\mu_N(A)} \sum_{\eta\in A} \mu_N(\eta)
\sum_{\xi \in B} R^{F_N}(\eta, \xi) \;\le \; r_{A\cup B}(A,B)
\;=\; \frac{\Cap_N(A,B)}{\mu_N(A)}\; ,
\end{equation*}
which proves \eqref{18}.

Fix a finite number of disjoint subsets $\ms E^1_N, \dots, \ms
E^\kappa_N$, $\kappa\ge 2$, of $E_N$: $\ms E^x_N\cap \ms
E^y_N=\varnothing$, $x\neq y$. Let $\ms E_N=\cup_{x\in S}\ms E^x_N$
and let $\breve{\ms E}^x_N := \ms E_N\setminus \ms E^x_N$. 

Denote by $r_N(\ms E^x_N,\ms E^y_N)$ the average rates $r_{\ms
  E_N}(\ms E^x_N,\ms E^y_N)$. The next result shows that if the
average rates appropriately rescaled converge, their limit can be
expressed in terms of the depth of the metastates and their hitting
probabilities.

Denote by $\Lambda_N (x)$, $1\le x\le \kappa$, the inverse of the
depth of a metastate and by $P_N(\eta,x,y)$, $1\le x\not = y\le
\kappa$, $\eta\in \ms E^x_N$, the hitting probabilities among
metastates:
\begin{equation*}
\Lambda_N(x) \;:=\; \frac{\Cap_N(\ms E^x , \breve{\ms E}^{x})}
{\mu_N(\ms E^{x})}\;, \quad P_N(\eta,x,y) \;:=\; 
\mb P_\eta \big[ T_{\ms E^{y}} = T_{\breve{\ms E}^{x}} \big]\;.
\end{equation*}

\begin{lemma}
\label{s39}
Suppose that for each $1\le x\le \kappa$ there exists a point $\bs
\xi_x = (\xi^N_x : N\ge 1)$ in $\ms E^x$ such that the triple $(\ms
E^x, \ms E^x \cup \bs \Delta, \bs \xi_x)$ is a valley of depth
$\mu_N(\ms E^{x})/\Cap_N(\ms E^x , \breve{\ms E}^{x})$ and such that
\begin{equation*}
\lim_{N\to\infty} \sup_{\eta \in \ms E^{x}_N} 
\frac{\Cap_N(\ms E^{x} , \breve{\ms E}^{x})}
{\Cap_N(\xi^N_x, \eta)} \;=\; 0\; .
\end{equation*}
Suppose, furthermore, that there exists a sequence $(\theta_N : N\ge
1)$ for which the mean rates, the depth and the jump probabilities
converge: For any $1\le x\not = y \le \kappa$ and any sequence
$(\eta^N : N\ge 1)$ in $\ms E^x$,
\begin{equation*}
\begin{split}
& \lim_{N\to\infty} \theta_N \, \Lambda_N(x) \;=\; \Lambda(x)\; , \quad
\lim_{N\to\infty} P_N(\eta^N,x,y) \;=\; P(x,y)\; \\
& \qquad \lim_{N\to\infty} \theta_N \, r_N(\ms E^x, \ms E^y) 
\;=\; r(x,y)\;.
\end{split}
\end{equation*}
Then, $r(x,y) = \Lambda(x) \, P(x,y)$.
\end{lemma}

\begin{proof}
Note that we assumed that the limit $P(x,y)$ does not depend on the
sequence $(\eta^N : N\ge 1)$.

It follows from \cite[Theorem 2.7]{bl2} that for any $1\le x\le
\kappa$ and any sequence $(\eta^N : N\ge 1)$ in $\ms E^x$, under the
measure $\mb P_{\eta^N}$ the speeded-up process $X^N_t =
\Psi(\eta^N_{t\theta_N})$ converges to a Markov process on $\{1,
\dots, \kappa\}$ with jump rates $r(y,z)$ starting from $x$. In
particular, if we denote by $\tau^N_1$ the time of the first jump of
$X^N_t$, $\tau^N_1$ converges to an exponential time of rate $\lambda
(x) = \sum_{y\not =x} r(x,y)$ and $X^N_{\tau^N_1}$ converges to a
random variable with distribution $p(y)= r(x,y)/\lambda(y)$.

On the other hand, since the triple $(\ms E^x, \ms E^x \cup \bs
\Delta, \bs \xi_x)$ is a valley of depth $\Lambda_N(x)^{-1}$,
$\tau^N_1\Lambda_N(x)$ converge to a mean one exponential time so that
$\Lambda(x) = \lambda(x)$. Moreover, $\mb P_{\eta^N} [X^N_{\tau^N_1} =
y] = \mb P_{\eta^N} [ T_{\ms E^{y}} = T_{\breve{\ms E}^{x}} ]$
converges to $P(x,y)$ so that $P(x,y)=p(y)$, which proves the lemma.
\end{proof}


\begin{thebibliography}{99}

\bibitem{bl2} J. Beltr\'an, C. Landim: Tunneling and metastability of
  continuous time Markov chains. J. Stat. Phys. {\bf 140}, 1065--1114,
  (2010).

\bibitem{bl3} J. Beltr\'an, C. Landim: Metastability of reversible
  condensed zero range processes on a finite set
  (2009). arXiv:0910.4089

\bibitem{bc} G. Ben Arous, R. Cerf: Metastability of the
  three-dimensional Ising model on a torus at very low temperature,
  Electron. J. Probab. {\bf 1} Research Paper 10 (1996).

\bibitem{b2} A. Bovier. Metastability: a potential theoretic approach.
  International Congress of Mathematicians. Vol. III, 499--518, Eur.
  Math. Soc., Z\"{u}rich, 2006.

\bibitem{begk1} A. Bovier, M. Eckhoff, V. Gayrard, M. Klein.
  Metastability in stochastic dynamics of disordered mean field
  models. Probab. Theory Relat. Fields {\bf 119}, 99-161 (2001)

\bibitem{begk2} A. Bovier, M. Eckhoff, V. Gayrard, M. Klein.
  Metastability and low lying spectra in reversible Markov chains.
  Commun. Math. Phys. {\bf 228}, 219--255 (2002).

\bibitem{bm} A. Bovier, F. Manzo. Metastability in Glauber dynamics in
  the low-temperature limit: beyond exponential asymptotics. J.
  Stat. Phys. {\bf 107}, 757--779 (2002).

\bibitem{cgov} M. Cassandro, A. Galves, E. Olivieri, M. E. Vares.
  Metastable behavior of stochastic dynamics: A pathwise approach.
  J. Stat. Phys. {\bf 35}, 603--634 (1984).

\bibitem{ghnos1} A. Gaudilli\`{e}re, F. Den Hollander, F. R. Nardi, E.
  Olivieri, E. Scoppola: Ideal gas approximation for a two-dimensional
  rarefied gas under Kawasaki dynamics. Stochastic Process. Appl.
  {\bf 119}, 737--774 (2009).

\bibitem{h1} F. Den Hollander: Metastability under stochastic
  dynamics. Stochastic Process.  Appl. {\bf 114}, 1--26 (2004).
 
\bibitem{h2} F. Den Hollander: Three lectures on metastability under
  stochastic dynamics. In Methods of Contemporary Mathematical
  Statistical Physics (R. Koteck\'y, ed.). Lecture Notes in Math. 1970.
  Springer, Berlin. (2009).

\bibitem{hnos1} F. den Hollander, E. Olivieri, E. Scoppola:
  Metastability and nucleation for conservative dynamics,
  J. Math. Phys. {\bf 41}, 1424--1498 (2000). 

\bibitem{hnos2} F. den Hollander, E. Olivieri, E. Scoppola: Nucleation
  in fluids: some rigorous results, Physica A {\bf 279}, 110--122
  (2000).

\bibitem{hnos} F. Den Hollander, F. R. Nardi, E.  Olivieri, E.
  Scoppola: Droplet growth for three-dimensional Kawasaki dynamics.
  Probab. Theory Related Fields {\bf 125}, 153--194 (2003).

\bibitem{jlt1} M. Jara, C. Landim, A. Teixeira: Quenched scaling
  limits of trap models. To appear in Ann. Probab. (2011).

\bibitem{jlt2} M. Jara, C. Landim, A. Teixeira: Quenched scaling
  limits of trap models in random graphs. In preparation.

\bibitem{ko1} R. Kotecky and E. Olivieri: Droplet dynamics for
  asymmetric Ising model, J. Stat. Phys. {\bf 70}
  1121--1148 (1993).

\bibitem{ko2} R. Kotecky and E. Olivieri: Shapes of growing
  droplets--a model of escape from a metastable phase,
  J. Stat. Phys. {\bf 75}, 409--506 (1994).

\bibitem{ns1} E. J. Neves, R. H. Schonmann: Critical droplets and
  metastability for a Glauber dynamics at very low temperatures. Comm.
  Math. Phys. {\bf 137}, 209--230 (1991).

\bibitem{ns2} E. J. Neves, R. H. Schonmann: Behavior of droplets for a
  class of Glauber dynamics at very low temperature.  Probab. Theory
  Related Fields {\bf 91}, 331--354 (1992).

\bibitem{n} J. R. Norris. {\em Markov chains}. Cambridge University
  Press, Cambridge (1997).

\bibitem{ov} E. Olivieri and M. E. Vares. {\em Large deviations and
    metastability}. Encyclopedia of Mathematics and its Applications,
  vol. 100. Cambridge University Press, Cambridge, 2005.

\bibitem{s} E. Scoppola. Renormalization group for Markov chains and
  application to metastability.  J. Stat. Phys. {\bf 73}, 83--121
  (1993). 

\end{thebibliography}
\end{document}